\documentclass[nolayout]{imsart}
\RequirePackage[OT1]{fontenc}
\RequirePackage{amsmath}
\RequirePackage[colorlinks,citecolor=blue,urlcolor=blue]{hyperref}

\usepackage{amsthm}

\usepackage[utf8]{inputenc}
\usepackage{amssymb}
\usepackage{makeidx}
\usepackage[english]{babel}
\usepackage{graphicx}
\usepackage{amsfonts,amssymb}
\usepackage{oldgerm}
\usepackage{mathrsfs}
\usepackage[active]{srcltx}
\usepackage{verbatim}
\usepackage{enumitem} %label option for enumerate \alph* \Alph* \roman* \Roman* or other... [label=$\bullet$], [label=\alph*]
\usepackage{aliascnt}
\usepackage{bbm}
\usepackage{array}
\usepackage{hyperref}
\usepackage[disable]{todonotes}
\usepackage{xargs}
\usepackage{cellspace}

\usepackage[Symbolsmallscale]{upgreek}

\usepackage{accents}
\usepackage{dsfont}
\usepackage{aliascnt}
\usepackage{cleveref}
\makeatletter
\newtheorem{theorem}{Theorem}
\crefname{theorem}{theorem}{Theorems}
\Crefname{Theorem}{Theorem}{Theorems}

\newaliascnt{lemma}{theorem}
\newtheorem{lemma}[lemma]{Lemma}
\aliascntresetthe{lemma}
\crefname{lemma}{lemma}{lemmas}
\Crefname{Lemma}{Lemma}{Lemmas}

\newaliascnt{corollary}{theorem}
\newtheorem{corollary}[corollary]{Corollary}
\aliascntresetthe{corollary}
\crefname{corollary}{corollary}{corollaries}
\Crefname{Corollary}{Corollary}{Corollaries}

\newaliascnt{proposition}{theorem}
\newtheorem{proposition}[proposition]{Proposition}
\aliascntresetthe{proposition}
\crefname{proposition}{proposition}{propositions}
\Crefname{Proposition}{Proposition}{Propositions}

\newaliascnt{definition}{theorem}

\aliascntresetthe{definition}
\crefname{definition}{definition}{definitions}
\Crefname{Definition}{Definition}{Definitions}

\newaliascnt{remark}{theorem}
\newtheorem{remark}[remark]{Remark}
\aliascntresetthe{remark}
\crefname{remark}{remark}{remarks}
\Crefname{Remark}{Remark}{Remarks}

\crefname{example}{example}{examples}
\Crefname{Example}{Example}{Examples}

\crefname{figure}{figure}{figures}
\Crefname{Figure}{Figure}{Figures}

\newtheorem{assumption}{\textbf{H}\hspace{-3pt}}
\Crefname{assumption}{\textbf{H}\hspace{-3pt}}{\textbf{H}\hspace{-3pt}}
\crefname{assumption}{\textbf{H}}{\textbf{H}}

\newtheorem{assumptionL}{\textbf{L}\hspace{-3pt}}
\Crefname{assumptionL}{\textbf{L}\hspace{-3pt}}{\textbf{L}\hspace{-3pt}}
\crefname{assumptionL}{\textbf{L}}{\textbf{L}}

\newtheorem{assumptionG}{\textbf{G}\hspace{-3pt}}
\Crefname{assumptionG}{\textbf{G}\hspace{-3pt}}{\textbf{G}\hspace{-3pt}}
\crefname{assumptionG}{\textbf{G}}{\textbf{G}}

\def\rmw{\mathrm{w}}
\def\measSet{\mathbb{M}}
\def\Setdrift{\mathcal{E}}
\newcommandx{\functionspace}[2][1=+]{\mathbb{F}_{#1}(#2)}
\def\convSym{\mathrm{c}}
\def\StSym{\mathrm{s}}

\newcommandx{\VarDeux}[3][3=]{\operatorname{Var}^{#3}_{#1}\left\{#2 \right\}}
\newcommand{\VarDeuxLigne}[2]{\operatorname{Var}_{#1}\{#2 \}}
\def\bornevar{D_n(\gaStep)}

\newcommand{\1}{\mathbbm{1}}

\newcommand{\B}{\mathcal{B}}
\newcommand{\borelSet}{\mathcal{B}}

\newcommand{\LeftEqNo}{\let\veqno\@@leqno}

\newcommand{\floor}[1]{\left\lfloor #1 \right\rfloor}
\newcommand{\ceil}[1]{\left\lceil #1 \right\rceil}

\newcommand{\N}{\ensuremath{\mathbb{N}}}

\newcommand{\PE}{\mathbb{E}}

\newcommand{\abs}[1]{\left\vert #1 \right\vert}
\newcommand{\tvnorm}[1]{\| #1 \|_{\mathrm{TV}}}

\newcommandx{\Vnorm}[2][1=V]{\| #2 \|_{#1}}
\newcommandx{\VnormEq}[2][1=V]{\left\| #2 \right\|_{#1}}
\newcommandx{\norm}[2][1=]{\ifthenelse{\equal{#1}{}}{\left\Vert #2 \right\Vert}{\left\Vert #2 \right\Vert^{#1}}}
\newcommandx{\normsup}[2][1=]{\ifthenelse{\equal{#1}{}}{\left\Vert #2 \right\Vert_{\infty}}{\left\Vert #2 \right\Vert^{#1}_{\infty}}}
\newcommandx{\normLigne}[2][1=]{\ifthenelse{\equal{#1}{}}{\Vert #2 \Vert}{\Vert #2\Vert^{#1}}}

\newcommand{\parenthese}[1]{\left(#1 \right)}
\newcommand{\parentheseDeux}[1]{\left[ #1 \right]}
\newcommand{\defEns}[1]{\left\lbrace #1 \right\rbrace }

\newcommand{\ps}[2]{\left\langle#1,#2 \right\rangle}
\newcommand{\eqdef}{\overset{\text{\tiny def}} =}

\newcommand{\proba}[1]{\mathbb{P}\left( #1 \right)}

\newcommandx\probaMarkovTilde[2][2=]
{\ifthenelse{\equal{#2}{}}{{\widetilde{\mathbb{P}}_{#1}}}{\widetilde{\mathbb{P}}_{#1}\left[ #2\right]}}

\newcommand{\expe}[1]{\PE \left[ #1 \right]}

\newcommand{\bigO}{\ensuremath{\mathcal O}}

\newcommand{\filtrationTilde}{\widetilde{\mathcal{F}}}

\newcommand{\plusinfty}{+\infty}

\newcounter{hypoconbis}
\newcounter{saveconbis}
\newcommand\debutH{\begin{list}
{\textbf{H\arabic{hypoconbis}}}{\usecounter{hypoconbis}}\setcounter{hypoconbis}{\value{saveconbis}}}
\newcommand\finH{\end{list}\setcounter{saveconbis}{\value{hypoconbis}}}

\def\ie{i.e.}
\def\eqsp{\;}
\newcommand{\coint}[1]{\left[#1\right)}
\newcommand{\ocint}[1]{\left(#1\right]}
\newcommand{\ooint}[1]{\left(#1\right)}
\newcommand{\ccint}[1]{\left[#1\right]}

\newcommand{\indi}[1]{\1_{#1}}
\newcommandx{\weight}[2][2=n]{\omega_{#1,#2}^N}

\newcommand{\boule}[2]{\operatorname{B}(#1,#2)}
\newcommand{\ball}[2]{\operatorname{B}(#1,#2)}

\def\TV{\mathrm{TV}}

\def\as{\ensuremath{\text{a.s}}}

\def\rmd{\mathrm{d}}
\newcommandx\sequence[3][2=,3=]
{\ifthenelse{\equal{#3}{}}{\ensuremath{( #1_{#2})}}{\ensuremath{( #1_{#2})_{ #2 \in #3}}}}
\newcommandx{\sequencen}[2][2=n\in\N]{\ensuremath{(#1)_{#2}}}
\newcommandx\sequenceDouble[4][3=,4=]
{\ifthenelse{\equal{#3}{}}{\ensuremath{\{ (#1_{#3},#2_{#3}) \}}}{\ensuremath{\{  (#1_{#3},#2_{#3}), \eqsp #3 \in #4 \}}}}
\newcommandx{\sequencenDouble}[3][3=n\in\N]{\ensuremath{\{ (#1_{n},#2_{n}), \eqsp #3 \}}}
\newcommand{\wrt}{w.r.t.}

\def\iid{i.i.d.}
\def\rme{\mathrm{e}}

\def\rset{\mathbb{R}}

\def\nset{\mathbb{N}}

\newcommandx{\CPE}[3][1=]{{\mathbb E}^{#3}_{#1}\left[#2 \right]} %%%% esperance conditionnelle
\newcommandx{\CPVar}[3][1=]{\mathrm{Var}^{#3}_{#1}\left\{ #2 \right\}}
\newcommand{\CPP}[3][]
{\ifthenelse{\equal{#1}{}}{{\mathbb P}\left(\left. #2 \, \right| #3 \right)}{{\mathbb P}_{#1}\left(\left. #2 \, \right | #3 \right)}}

\def\generator{\mathscr{A}}

\def\PLang{\PL}

\newcommandx{\osc}[2][1=]{\mathrm{osc}_{#1}(#2)}
\def\bEa{c}
\def\cVEa{\chi}
\def\VEa{\Vconv{\convSym}}
\def\thetaEa{\lambda}
\def\KUa{R_{\convSym}}
\newcommandx{\boundConvTV}[1][1=]{F^{#1}_{n,p}(x;\gamma)}

\def\rhoU{\rhoUl}

\newcommand{\Vconv}[1]{W_{#1}}

\newcommand{\VStD}{W_{\StSym}}
\def\Gammasum{\GaStep}
\def\mcf{\mathcal{F}}

\newcommand{\chunk}[4][]%
{\ifthenelse{\equal{#1}{}}{\ensuremath{{#2}_{#3:#4}}}{\ensuremath{#2^#1}_{#3:#4}}
}

\def\KL{\operatorname{KL}}

\def\Id{\operatorname{Id}}
\def\IdM{\operatorname{I}_d}

\def\Ltwo{\mathrm{L}^2}
\def\Lone{\mathrm{L}^1}
\newcommand\densityPi[1]{\frac{\rmd #1}{\rmd \pi}}
\newcommand\densityPiLigne[1]{\rmd #1 /\rmd \pi}
\newcommand\density[2]{\frac{\rmd #2}{\rmd #1}}
\newcommand\densityLigne[2]{\rmd #2/\rmd #1}

\def\Csetfunction{\mathrm{C}}

\def\Gammabf{\mathbf{\Gamma}}

\def\CPoincare{C_{\mathrm{P}}}

\def\Clogsob{C_{\mathrm{LS}}}

\newcommand\Ent[2]{\mathrm{Ent}_{#1}\left(#2\right)}

\def\bargaStep{\bar{\gaStep}}
\def\Funfo{F}
\def\FunfoD{G}
\def\globmin{\xstar}

\def\VOne{\VEa}
\def\thetaOne{\theta}
\def\KOne{K}
\def\bOne{\beta}
\def\gaStep{\gamma}
\def\GaStep{\Gamma}
\def\RKer{R}
\def\QKer{Q}

\def\VSAlpha{{\Vly_{\varsigma}}}
\def\thetaSAlpha{\theta_{\varsigma}}
\def\KSAlpha{K_{\varsigma}}
\def\bSAlpha{\beta_{\varsigma}}

\def\ConstVnorm{C_{\varsigma}}
\def\rateExpVnorm{\upsilon_{\varsigma}}

\def\ConstVnormUndemi{C_{1/2}}
\def\rateExpVnormUndemi{\upsilon_{1/2}}

\def\ConstVnormUnquart{C_{1/4}}
\def\rateExpVnormUnquart{\upsilon_{1/4}}
\def\Calpha{a_{\alpha}}
\def\Mrho{M_{\rho}}

\def\bDriftRgaStep{c}

\def\KDriftRgaStep{K}

\def\lDriftgaStep{\lambda}

\def\Vlysqrt{\Vly^{1/2}}
\def\Vlyb{\Vly}

\def\bUndemi{\beta_{1/2}}
\def\KUndemi{K_{1/2}}
\def\thetaUndemi{\theta_{1/2}}

\def\kgaStep{\mathrm{k}_\gaStep}
\def\qgaStep{\mathrm{q}_\gaStep}
\def\rgaStep{\mathrm{r}_\gaStep}

\def\pinumber{\uppi}
\def\boreleanA{\mathrm{A}}

\def\boundDeuxTVa{B(\gaStep,V(x))}

\def\boundDeuxTVaccarre{B^2(\gaStep,v)}
\def\boundDeuxTVab{B(\gaStep,1)}

\def\rankP{p_0}
\def\nFunp{n(p)}

\def\BM{B}
\def\AtermProofBloc{A}
\def\BtermProofBloc{B}

\def\LU{L}
\def\PL{P}%^{\mathrm{L}}}
\def\YL{Y}%^{\mathrm{L}}}
\def\YLtilde{\tilde{Y}}%^{\mathrm{L}}}
\def\YLbar{\bar{Y}}%^{\mathrm{L}}}
\def\XE{X}
\def\ZE{\mathrm{Z}}
\def\rhoUl{\eta}
\def\RUl{M_{\rhoU}}
\def\VUl{U}
\def\xstar{x^{\star}}
\def\Vly{V}
\def\XSDE{\mathbf{X}}
\def\PSDE{\mathbf{P}}

\def\YSDE{\mathbf{Y}}

\def\ZE{Z}
\def\fune{e}
\def\shift{\mathrm{S}}
\def\RSt{M_{\StSym}}
\def\constSt{m}
\def\constStD{\tilde{m}_{\StSym}}
\def\mStD{\constSt}
\def\RStV{\RSt}
\def\constStV{\constSt}
 \def\time{t}
\def\generatorL{\generator^{\mathrm{L}}}

\def\barB{\bar{B}}

\def\bfPhi{\mathbf{\Phi}}
\newcommand{\PEMb}[2]{\tilde{\mathbb{E}}_{#1}\left[#2 \right]}
\newcommand{\PEMbLigne}[2]{\tilde{\mathbb{E}}_{#1}[#2 ]}
\newcommand{\PPMb}[2]{\tilde{\mathbb{P}}_{#1}\left(#2 \right)}

\newcommand{\2}[1]{\1_{\{#1\}}}
\def\diagSet{\Delta}
\def\tpsRetour{\mathrm{T}}

\def\eqspD{\hspace{0.2mm}}
\def\LF{\LU}

\def\thetaESt{\uplambda}

\def\boundM{\mathrm{K}}

\def\convexSet{\mathsf{K}}

\def\UD{U}
\def\bESt{\mathrm{c}}

\def\VdriftS{\mathrm{W}_{\StSym}}

\def\deltaSS{\delta}
\def\RdriftSS{\mathrm{R}}

\def\Thetadrift{\Theta}
\def\closed{G}
\def\geneLang{\generatorL}
\def\boundMd{\mathrm{D}}

\def\gaStep{\gamma}

\def\Fsmall{\omega}
\def\tauConv{\kappa}
\def\nbTpsR{\ell}
\def\Vdrifta{V}

\def\bdrift{\beta}
\def\thetadrift{\theta}
\def\LUl{\LU}

\def\tauConvSt{\kappa}
\def\Gammasum{\GaStep}

\def\RStD{\tilde{M}_{\StSym}}

\def\Uun{U_1}
\def\Ude{U_2}
\def\kappaS{\varpi}
\def\xstarun{\xstar_1}
\def\xunstar{\xstar_1}
\def\Lun{L_1}
\def\funtestcontrolEntropie{\psi}
\def\sigmagaStep{\sigma_{\gaStep}}

\begin{document}

\begin{frontmatter}

\title{Non-asymptotic convergence analysis for the Unadjusted Langevin Algorithm}
\runtitle{Non-asymptotic convergence analysis for the ULA}

\begin{aug}
\author{\fnms{Alain} \snm{Durmus} \thanksref{e1} \ead[label=e1,mark]{alain.durmus@telecom-paristech.fr}} \and
\author{ \fnms{\'Eric}  \snm{Moulines} \thanksref{e2} \ead[label=e2,mark]{eric.moulines@polytechnique.edu}}

\affiliation{Institut Mines-Télécom ; Télécom ParisTech ; CNRS LTCI }
\address{LTCI, Telecom ParisTech \& CNRS, \\
46 rue Barrault, 75634 Paris Cedex 13, France.\\
\printead{e1}\\}
\affiliation{Centre de Math\'ematiques Appliqu\'ees, UMR 7641, Ecole Polytechnique}
\address{Centre de Math\'ematiques Appliqu\'ees, UMR 7641,\\
 Ecole Polytechnique,\\
 route de Saclay,
 91128 Palaiseau cedex, France.\\
 \printead{e2}
}

\runauthor{A.~Durmus, \'E.~Moulines}
\end{aug}

\begin{abstract}
  :~In this paper, we study a method to sample from a target
  distribution $\pi$ over $\rset^d$ having a positive density with
  respect to the Lebesgue measure, known up to a
  normalisation factor. This method is based on the Euler
  discretization of the overdamped Langevin stochastic differential
  equation associated with $\pi$. For both constant and decreasing
  step sizes in the Euler discretization, we obtain non-asymptotic
  bounds for the convergence to the target distribution $\pi$ in total
  variation distance. A particular attention is paid to the dependency
  on the dimension $d$, to demonstrate the
  applicability of this method in the high dimensional setting.  These
  bounds improve and extend the results of \cite{dalalyan:2014}.
\end{abstract}

% \begin{abstract}
%  :~Sampling distributions over high-dimensional state-spaces is a problem which has recently attracted a lot of research efforts; applications include Bayesian non-parametrics, Bayesian inverse problems and aggregation of estimators.
% All these problems boil down to sample a target distribution $\pi$ having a density  \wrt\ the Lebesgue measure on $\rset^d$, known up to a normalisation factor  (see \cite{cotter:roberts:stuart:white:2013} for details) $x \mapsto  \rme^{-U(x)}/\int_{\rset^d}  \rme^{-U(y)} \rmd y$ where $U$ is continuously differentiable with Lipschitz gradient. In this paper, we study a sampling technique based on  the Euler discretization of the Langevin stochastic differential equation. Contrary to the Metropolis Adjusted Langevin Algorithm (MALA), we do not apply a Metropolis-Hastings correction. For both constant and decreasing step sizes in the Euler discretization, we obtain non-asymptotic bounds for the convergence to the target distribution $\pi$ in total variation  distance. A particular attention is paid on the dependence on the dimension of the state space, to demonstrate the applicability of  this method in the high dimensional setting, at least when $U$ is convex.
%  These bounds improve and extend the results  of \cite{dalalyan:2014}.
% \end{abstract}


\end{frontmatter}

\maketitle
\section{Introduction}
Sampling distributions over high-dimensional state-spaces is a
  problem which has recently attracted a lot of research efforts in
  computational statistics and machine learning (see \cite{cotter:roberts:stuart:white:2013} and \cite{andrieu:defreitas:doucet:jordan:2003}
  for details); applications include Bayesian non-parametrics,
  Bayesian inverse problems and aggregation of estimators.  All these
  problems boil down to sample a target distribution $\pi$ having a
  density \wrt\ the Lebesgue measure on $\rset^d$, known up to a
  normalisation factor  $x \mapsto \rme^{-U(x)}/\int_{\rset^d} \rme^{-U(y)}
  \rmd y$ where $U$ is continuously differentiable.  We consider a
sampling method based on the Euler discretization of the overdamped Langevin
stochastic differential equation (SDE)
\begin{equation}
\label{eq:langevin}
\rmd \YL_t = -\nabla U (\YL_t) \rmd t + \sqrt{2}\, \rmd \BM_t^d \eqsp,
\end{equation}
 where $(B_t^d)_{t\geq0}$ is a $d$-dimensional Brownian motion.
%When $y \mapsto \nabla U(y)$ is Lipschitz, then \eqref{eq:langevin} admits a unique strong solution $(\YL_t)_{ t \geq 0}$ which is a Markov process.
 It is well-known that the Markov semi-group associated with the Langevin diffusion $(\YL_t)_{t \geq 0}$ is reversible \wrt~$\pi$. Under suitable conditions, the convergence to $\pi$ takes place at geometric rate. Precise quantitative estimates of the rate of convergence with explicit
dependency on the dimension $d$ of the state space have been recently obtained using either
functional inequalities such as Poincar{\'e} and log-Sobolev inequalities (see \cite{bakry:cattiaux:guillin:2008,cattiaux:guillin:2009} \cite{bakry:gentil:ledoux:2014}) or by coupling techniques (see \cite{eberle:2015}).
The Euler-Maruyama discretization scheme associated to the Langevin diffusion yields the discrete time-Markov chain given by
\begin{equation}
\label{eq:euler-proposal}
\XE_{k+1}= \XE_k - \gaStep_{k+1} \nabla U(\XE_k) + \sqrt{2 \gaStep_{k+1}} \ZE_{k+1}
\end{equation}
where $(\ZE_k)_{k \geq 1}$ is an \iid\ sequence of standard Gaussian $d$-dimensional random vectors and $(\gaStep_k)_{k \geq 1}$ is a sequence of step sizes,
which can either be held constant or be chosen to decrease to $0$.
The idea of using the Markov chain $(\XE_k)_{k \geq 0}$ to sample approximately from the target $\pi$ has  been first introduced in the physics literature by \cite{parisi:1981} and popularised in the computational statistics community by \cite{grenander:1983} and \cite{grenander:miller:1994}. It has been studied in depth by \cite{roberts:tweedie:1996}, which proposed to use a Metropolis-Hastings step at each iteration to enforce reversibility \wrt\ $\pi$ leading to the Metropolis Adjusted Langevin Algorithm (MALA).
They coin the term \emph{unadjusted} Langevin algorithm (ULA) when the Metropolis-Hastings step is skipped.

The purpose of this paper is to study  the convergence of the ULA algorithm. The emphasis is put on non-asymptotic computable bounds; we pay a particular attention to the way these bounds scale with the dimension $d$ and constants characterizing the smoothness and curvature of the potential $U$. Our study covers both constant and decreasing step sizes and we analyse both the "finite horizon"  (where the total number of simulations is specified before running the algorithm) and "any-time" settings (where the algorithm can be stopped after any iteration).
%It is of utmost interest to obtain computable non asymptotic bounds, to be able
%to construct "honest" confidence intervals and / or to "optimize" the
%simulation parameters to obtain a "guaranteed" performance.

When the step size $\gaStep_k= \gaStep$ is  constant, under appropriate conditions (see \cite{roberts:tweedie:1996}), the Markov chain $(\XE_n)_{n \geq 0} $  is $V$-uniformly geometrically ergodic with a stationary distribution $\pi_\gaStep$. With few exceptions, the stationary distribution $\pi_\gaStep$ is different from the target $\pi$. If the step size $\gaStep$ is small enough, then the stationary distribution of this chain is in some sense close to $\pi$.  We provide non-asymptotic bounds of the $V$-total variation distance between $\pi_\gamma$ and $\pi$, with explicit dependence on the step size $\gaStep$ and the dimension $d$. Our results complete and extend the recent works by \cite{dalalyan:tsybakov:2012} and \cite{dalalyan:2014}.
 % We first consider the case of densities which are super-exponential in the tails, for which we prove convergence in total variation. We also address the case where $\pi$ is globally log-concave and strongly log-concave outside a ball.

When $(\gaStep_k)_{k \geq 1}$ decreases to zero, then $(\XE_k)_{k \geq 0}$ is a non-homogeneous Markov chain.
If in addition $\sum_{k=1}^\infty \gaStep_k = \infty$,  we show that the marginal distribution  of this non-homogeneous chain converges, under some mild additional conditions,  to the target distribution $\pi$, and provide explicit bounds for the convergence.
 % We also provide explicit expression for the convergence rate, emphasizing the role of the dimension and the rate of decrease of the sequence $(\gaStep_k)_{k\geq 1}$.
Compared to the related works by \cite{lamberton:pages:2002}, \cite{lamberton:pages:2003}, \cite{lemaire:2005} and \cite{lemaire:menozzi:2010}, we establish not only the weak convergence of the weighted empirical measure of the path to the target distribution  but  a much stronger convergence in total variation, similarly to \cite{dalalyan:2014}, where the strongly log-concave case is considered.

The paper is organized as follows.
%\Cref{sec:langevin_result} is divided in two parts:
In \Cref{sec:analysis-ula-under}, the main convergence results are stated under abstract assumptions. We
then specialize in \Cref{sec:pract-cond-geom} these results to different classes of densities. 
%In \Cref{subsec:superexponential}, these results are stated for a class
%of densities which are superexponential in the tails. In
%\Cref{subsec:log-concave} these results are sharpened for
%densities which are log-concave. In
%\Cref{subsec:strongly-log-concave}, densities which are log-concave
%and strongly log-concave in the tails are dealt with. Finally in
%\Cref{sec:bound-pert-strongly}, bounded pertubations of
%strongly log-concave densities are considered.
The proofs are gathered in
\Cref{sec:proof}. Some general convergence results for diffusions
based on reflection coupling, which are of independent interest, are
stated in \Cref{sec:quant-conv-bounds}.

%%% Local Variables:
%%% mode: latex
%%% TeX-master: "main"
%%% End:

\subsection*{Notations and conventions}
$\mathcal{B}(\rset^d)$ denotes the Borel $\sigma$-field of $\rset^d$ and $\functionspace[]{\rset^d}$ the set of all Borel measurable functions on $\rset^d$. For $f \in \functionspace[]{\rset^d}$ set $\Vnorm[\infty]{f}= \sup_{x \in \rset^d} \abs{f(x)}$. Denote by $\measSet(\rset^d)$ the space of finite signed measure on $(\rset^d, \mathcal{B}(\rset^d))$ and $\measSet_0(\rset^d) =  \{\mu \in \measSet(\rset^d)\ | \ \mu(\rset^d) = 0\}$. For $\mu \in \measSet(\rset^d)$  and $f \in \functionspace[]{\rset^d}$ a $\mu$-integrable function, denote by $\mu(f)$ the integral of $f$ \wrt~$\mu$. Let $\Vly: \rset^d \to \coint{1,\infty}$ be a measurable function. For $f \in \functionspace[]{\rset^d}$, the $\Vly$-norm of $f$ is given by $\Vnorm[\Vly]{f}= \sup_{x \in \rset^d} |f(x)|/\Vly(x)$. For $\mu \in \measSet(\rset^d)$, the $\Vly$-total variation distance of $\mu$ is defined as
\begin{equation*}
%\label{eq:definition_TV}
\Vnorm[\Vly]{\mu} = \sup_{f \in \functionspace[]{\rset^d}, \Vnorm[\Vly]{f} \leq 1}  \abs{\int_{\rset^d } f(x) \rmd \mu (x)} \eqsp.
\end{equation*}
If $\Vly \equiv 1$, then $\Vnorm[\Vly]{\cdot}$ is the total variation  denoted by $\tvnorm{\cdot}$.

For $p \geq 1$, denote by $\mathrm{L}^p(\pi)$ the set of  measurable functions such that $\pi(|f|^p) < \infty$.  For
$f \in \Ltwo(\pi)$, the variance of $f$
under $\pi$ is denoted by $\VarDeux{\pi}{f}$.
% is given by:
%$$
%\VarDeux{\pi}{f} = \int_{\rset^d} \defEns{f(x) - \int_{\rset^d} f(y) \rmd \pi (y)}^2 \rmd \pi(x) \eqsp.
%$$
For all functions $f$ such that $f \log(f) \in \Lone(\pi)$, the
entropy of $f$ with respect to $\pi$   is defined by
\begin{equation*}
%\label{eq:definition_entropie}
\Ent{\pi}{f} = \int_{\rset^d} f(x)\log(f(x)) \rmd \pi (x) \eqsp.
\end{equation*}
Let $\mu$ and $\nu$ be two probability measures on $\rset^d$. If $\mu
\ll \nu$, we denote by $\densityLigne{\nu}{\mu}$ the Radon-Nikodym
derivative of $\mu$ \wrt~$\nu$. Denote for all $x,y \in \rset^d$ by
$\ps{x}{y}$ the scalar product of $x$ and $y$ and $\norm{x}$ the
Euclidean norm of $x$. For $k \geq 0$, denote by $C^k(\rset^d)$, the
set of $k$-times continuously differentiable functions $f : \rset^d
\to \rset$. For $f \in C^2(\rset^d)$, denote by $\nabla f$ the
gradient of $f$ and $\Delta f$ the Laplacian of $f$.  For all $x \in
\rset^d$ and $M >0$, we denote by $\boule{x}{M}$, the ball centered at
$x$ of radius $M$.  Denote for $K \geq 0$, the oscillation of a
function $f \in C^0(\rset^d)$ in the ball $\boule{0}{K}$ by
$\osc[K]{f}= \sup_{\boule{0}{K}}(f) - \inf_{\boule{0}{K}}(f)$. Denote
the oscillation of a bounded function $f \in C^0(\rset^d)$ on
$\rset^d$ by $\osc[\rset^d]{f}= \sup_{\rset^d}(f) -
\inf_{\rset^d}(f)$.  In the sequel, we take the convention that
$\sum_{p}^n =0$ and $\prod_p ^n = 1$, for $n,p \in \nset$, $n <p$.

% \section{Main results}
% \label{sec:langevin_result}

% In this section, we will present our main results. First, we derive a
% bound on the convergence of the ULA to the target distribution $\pi$
% when the Langevin diffusion is geometrically ergodic, and the
% Markov kernel associated with the EM discretization, satisfies a
%  a Foster-Lyapunov drift
% inequality. In the second part, we give practical conditions on the
% potential $U$ which provide explicit expressions for the constants appearing in the results of the first part.

\section{General conditions for the convergence of ULA}
\label{sec:analysis-ula-under}
In this section, we derive a bound on the convergence of the ULA to
the target distribution $\pi$ when the Langevin diffusion is
geometrically ergodic and the Markov kernel associated with the EM
discretization satisfies a Foster-Lyapunov drift inequality.
%  In
% \Cref{sec:pract-cond-geom}, explicit expressions for the constants
% appearing in the results of the first part are provided under
% different conditions on the potential $U$.

Consider the following assumption on the potential $U$:
\begin{assumptionL}
\label{assum:regularity}
The function $U$ is continuously differentiable on $\rset^d$ and gradient Lipschitz, \ie~there exists $\LU \geq 0$ such that  for all $x,y \in \rset^d$,
$$
\norm{\nabla U(x) - \nabla U(y)} \leq \LU \norm{x-y} \eqsp.
$$
\end{assumptionL}
Under \Cref{assum:regularity},
by  \cite[Theorem 2.4-3.1]{ikeda:watanabe:1989} for every initial point $x \in \rset^d$, there exists a unique strong
solution $(\YL_t(x))_{t \geq 0}$ to the Langevin SDE \eqref{eq:langevin}.
Define for all $t \geq 0$, $x \in \rset^d$ and  $\boreleanA \in \B(\rset^d)$,
$\PL_t(x,\boreleanA) = \proba{\YL_t(x) \in \boreleanA}$. The semigroup $(\PL_t)_{t \geq 0}$ is reversible \wrt~$\pi$, and hence admits $\pi$ as its (unique) invariant distribution.
In this section, we consider the case where $(\PL_t)_{t \geq 0}$ is geometrically ergodic, \ie~there exists  $\kappa \in \coint{0,1}$ such that for  any initial distribution $\mu_0$ and $t > 0$,
\begin{equation}
\label{eq:ergodicity-langevin-expo}
\norm{\mu_0 \PL_t - \pi}_{\TV} \leq C(\mu_0) \kappa^t    \eqsp,
\end{equation}
for some constant $C(\mu_0) \in \ccint{0,\plusinfty}$.
Denote by $\generatorL$ the generator    associated with the  semigroup $(\PL_t)_{t \geq 0}$, given for all $f\in C^2(\rset^d)$ by
\begin{equation*}
%\label{eq:generator-langevin-diffusion}
\geneLang f = -\ps{\nabla U}{\nabla f} + \Delta f \eqsp.
\end{equation*}
 A twice continuously differentiable  function $\Vly: \rset^d \to \coint{1,\infty}$ is a \emph{Lyapunov function} for the generator $\geneLang$ if there exist $\theta > 0$, $\beta \geq 0$ and $\Setdrift \subset \borelSet$ such that,
\begin{equation}
\label{eq:lyapunov-condition-V}
\geneLang \Vly \leq -\theta \Vly + \beta \1_{\Setdrift} \eqsp.
\end{equation}
By \cite[Theorem~2.2]{roberts:tweedie:1996}, if $\Setdrift$ in \eqref{eq:lyapunov-condition-V} is a non-empty compact set, then the Langevin diffusion is  geometrically ergodic.
% The value of $C(\mu_0)$ and the rate $\kappa$ can be made explicit,  but in general they are pessimistic and depend exponentially on the dimension.
 % We will see in the sequel other methods to obtain \eqref{eq:ergodicity-langevin-expo} with tighter control on $C(\mu_0)$ and $\kappa$.
 % These upper bounds depend crucially on the assumptions that we are ready to make on the potential $U$ and also on the methods of proof (functional inequalities and coupling techniques yield different values for this constant). The dependency on the dimension $d$ is of primary interest in all the derivations that follow.

Consider now the EM discretization of the diffusion \eqref{eq:euler-proposal}.
Let $(\gaStep_k)_{k \geq 1}$ be a sequence of positive and nonincreasing step sizes and for $0 \leq n \leq p$, denote by
 \begin{equation}
 \label{eq:def_GaStep}
 \GaStep_{n,p} = \sum_{k=n}^p \gaStep_k \eqsp, \qquad \GaStep_n = \GaStep_{1,n} \eqsp.
\end{equation}
%\alain{modif}
%Recall that the Euler discretization $(X_n)_{n \geq 0}$ associated with \eqref{eq:langevin}, this sequence of step size and started at $X_0$ is given by
%\eqref{eq:euler-proposal} where $(Z_k)_{k \geq 1}$ is independent of $X_0$.
%%\begin{equation}
%%\label{eq:euler-proposal}
% \eqsp,
%%\end{equation}
%%where $(Z_k)_{k \geq 1}$ is an \iid\ sequence of standard Gaussian random variables, independent of $X_0$.
%\alain{fin modif}
For $\gaStep >0$, consider the Markov kernel $\RKer_\gaStep$ given for all $\boreleanA \in \mathcal{B}(\rset^d)$
and $x \in \rset^d$ by
\begin{equation*}
%\label{eq:definition-RgaStep}
\RKer_\gaStep(x,\boreleanA) =
\int_\boreleanA (4\uppi \gaStep)^{-d/2} \exp \parenthese{-(4 \gaStep)^{-1}\norm[2]{y-x+ \gaStep \nabla U(x)}} \rmd y
\eqsp.
\end{equation*}
The discretized Langevin diffusion $(\XE_n)_{n \geq 0}$ given in \eqref{eq:euler-proposal} is a  time-inhomogeneous
Markov chain, for $p \geq n \geq 1$ and $f \in  \functionspace[+]{\rset^d}$,  $\CPE{f(\XE_p)}{\mcf_n}= \QKer^{n,p}_\gaStep  f(\XE_n)$ where $\mcf_n= \sigma(\XE_\ell, 0 \leq \ell \leq n)$ and
%This Markov chain is inhomogeneous if the step sizes $(\gaStep_k)_{k \geq 1}$ are not constant and homogeneous otherwise.
\begin{equation*}
%\label{eq:iterate_kernel}
\QKer^{n,p}_\gaStep = \RKer_{\gaStep_n} \cdots \RKer_{\gaStep_p}  \eqsp, \qquad \QKer^{n}_\gaStep = \QKer^{1,n}_\gaStep \eqsp,
\end{equation*}
with the convention that for $n,p \geq 0$, $n < p$, $\QKer^{p,n}_\gaStep$ is the identity operator.
Under \Cref{assum:regularity}, the Markov kernel $\RKer_\gaStep$ is strongly Feller, irreducible and strongly aperiodic. We will say that a function $V: \rset^d \to \coint{1,\infty}$  satisfies a Foster-Lyapunov drift condition for $\RKer_\gaStep$ if  there exist constants $\bar{\gaStep} >0$, $\lambda \in \coint{0,1}$ and $c > 0$ such that, for all $\gaStep \in \ocint{0,\bar{\gaStep}}$
\begin{equation}
\label{eq:foster-lyapunov-drift}
\RKer_\gaStep V \leq \lambda^\gaStep V + \gaStep c  \eqsp.
\end{equation}
The particular form of \eqref{eq:foster-lyapunov-drift} reflects how
the mixing rate of the Markov chain depends upon the step size $\gaStep
> 0$. If $\gaStep= 0$, then $\RKer_0(x,\boreleanA)= \delta_x(\boreleanA)$ for $x \in
\rset^d$ and $\boreleanA \in \mathcal{B}(\rset^d)$. A Markov chain with
transition kernel $\RKer_0$ is not mixing. Intuitively, as $\gaStep$
gets larger, then it is expected that the mixing of
$\RKer_\gaStep$ increases.  If for some $\gaStep > 0$, $\RKer_\gaStep$ satisfies
\eqref{eq:foster-lyapunov-drift}, then $\RKer_\gaStep$
admits a unique stationary distribution $\pi_\gaStep$.
 % and the Markov
% kernel is $V$-uniformly geometrically ergodic, \ie~there exist
% constants $C(\gaStep) < \infty$ and $\lambda \in \coint{0,1}$, such
% that for all $x \in \rset^d$,
% \begin{equation*}
% %\label{eq:ergodicity-langevin-discretized}
% \norm{\RKer^k_{\gaStep}(x,\cdot) - \pi_\gaStep}_{V} \leq C(\gaStep) \lambda^{\gaStep \; k} V(x)  \eqsp.
% \end{equation*}
% The constants $C(\gaStep)$ and $\lambda$ depend once again on the assumptions on the potential $U$. For reasons that will become obvious in the sequel, we will not use estimates of the form \eqref{eq:ergodicity-langevin-discretized} but rather
We  use \eqref{eq:foster-lyapunov-drift}  to control quantitatively the moments of the time-inhomogeneous chain.
The types of bounds which are needed, are summarised in the following elementary Lemma.
\begin{lemma}
\label{lem:bound_inho_moment}
Let $\bar{\gaStep} >0$. Assume that for all $x \in \rset^d$ and $\gaStep \in \ocint{0,\bar{\gaStep}}$, \eqref{eq:foster-lyapunov-drift} holds for some constants $\lambda \in \ooint{0,1}$ and $c >0$. Let $(\gaStep_k)_{k \geq 1}$ be a sequence of nonincreasing step sizes such that
$\gaStep_k \in \ocint{0,\bar{\gaStep}}$ for all $k \in \nset^*$. Then for all $n \geq 0$ and $x \in \rset^d$, $  \QKer_{\gaStep}^n \Vly(x) \leq \Funfo(\lambda,\Gammasum_{n},c,\gaStep_1,\Vly(x))$ where
\begin{equation}
\label{eq:def_Funfo}
\Funfo(\lambda, a, c, \gaStep, w) =  \lambda^{a}w+  c(-\lambda^{\gaStep} \log(\lambda))^{-1} \eqsp.
\end{equation}
\end{lemma}
\begin{proof}
  The proof is postponed to \Cref{sec:proof:lem:inh}.
\end{proof}
Note that \Cref{lem:bound_inho_moment} implies that $\sup_{k \geq 0} \{ \QKer_{\gaStep}^k \Vly(x) \} \leq \FunfoD(\lambda,c,\gaStep_1,\Vly(x))$ where
\begin{equation}
  \label{eq:def_FunfoD}
  \FunfoD(\lambda,c,\gaStep,w) = w + c(-\lambda^{\gaStep} \log(\lambda))^{-1} \eqsp.
\end{equation}
We give below the main ingredients which are needed to obtain a quantitative bound for $\tvnorm{\delta_x \QKer_{\gaStep}^p -\pi}$ for all $x \in \rset^d$.
This quantity is decomposed as follows: for all $0 \leq n < p$,
\begin{multline}
\label{eq:triangle_ineq_TV}
\tvnorm{\delta_x \QKer_{\gaStep}^p -\pi} \\
\leq  \tvnorm{\delta_x \QKer_{\gaStep}^n \QKer_{\gaStep}^{n+1,p} - \delta_x \QKer^n_\gaStep \PL_{\GaStep_{n+1,p}}} + \tvnorm{\delta_x \QKer^n_\gaStep \PL_{\GaStep_{n+1,p}} - \pi}
 \eqsp.
\end{multline}
To control the first term on the right hand side, we use a method
introduced in \cite{dalalyan:tsybakov:2012} and elaborated in
\cite{dalalyan:2014}. The second term is bounded using the convergence of
the semi-group to $\pi$, see \eqref{eq:ergodicity-langevin-expo}.
\begin{proposition}
  \label{propo:girsanov_comparison}
  Assume that \Cref{assum:regularity} and
  \eqref{eq:ergodicity-langevin-expo} hold. Let $(\gaStep_k)_{k \geq 0}$ be a sequence of nonnegative
  step sizes. Then for all $x \in \rset^d$, $n \geq 0$, $p \geq 1$, $n < p$,
\begin{multline}
  \label{eq:eq_base}
  \tvnorm{\delta_x \QKer_{\gaStep}^p -\pi} \\
\leq
2^{-1/2} \LU \parenthese{ \sum_{k=n}^{p-1} \defEns{(\gaStep_{k+1}^3/3) A(\gaStep,x) + d \gaStep_{k+1}^2}}^{1/2}
+ C(\delta_x \QKer_{\gaStep}^n) \kappa^{\GaStep_{n+1,p}} \eqsp,
\end{multline}
where $\kappa, C(\delta_x \QKer_{\gaStep}^n)$ are defined in \eqref{eq:ergodicity-langevin-expo} and
\begin{equation}
\label{eq:eq_base1}
A(\gaStep,x) = \sup_{k \geq 0}  \int_{\rset^d}\norm[2]{\nabla U(z)} \QKer_{\gaStep}^{k}(y,\rmd z) \eqsp.
\end{equation}
\end{proposition}

\begin{proof}
The proof follows the same lines as  \cite[Lemma 2]{dalalyan:2014} but is given for completeness.
For $0 \leq s \leq t$, let $\Csetfunction(\ccint{s,t},\rset^d)$ be  the space of continuous functions on $\ccint{s,t}$ taking values in $\rset^d$.
For all  $y \in \rset^d$, denote by $\mu_{n,p}^y$ and $\bar{\mu}_{n,p}^y$ the laws on $\Csetfunction(\ccint{\GaStep_n,\GaStep_p},\rset^d)$ of the Langevin diffusion $(\YL_t(y))_{\GaStep_n \leq t \leq \GaStep_p}$ and of the continuously-interpolated Euler discretization $(\YLbar_t(y))_{\GaStep_n \leq t \leq \GaStep_p}$, both started at $y$ at time $\GaStep_n$.
Denote by $(Y_t(y),\overline{Y}_t(y))_{t \geq \GaStep_n}$ the unique strong solution started at $(y,y)$ at time $t=\Gamma_n$
of  the time-inhomogeneous diffusion   defined  for $t \geq \GaStep_n$,  by
\begin{equation}
\label{eq:definition_couplage}
\begin{cases}
\rmd\YL_t = - \nabla U(\YL_t) \rmd t + \sqrt{2} \rmd B_t^d  \\
\rmd \YLbar_t= - \overline{\nabla U}(\YLbar,t) \rmd t + \sqrt{2} \rmd B_t^d \eqsp,
\end{cases}
\end{equation}
where for any continuous function $\rmw: \rset_+ \to \rset^{d}$  and $t \geq \Gamma_n$
\begin{equation}
\label{eq:definition-nablaU-discret}
\overline{\nabla U}(\rmw,t)= \sum_{k=n}^{\infty} \nabla U(\rmw_{\GaStep_k}) \indi{\coint{\GaStep_k,\GaStep_{k+1}}}(t) \eqsp.
\end{equation}

Girsanov's Theorem \cite[Theorem~5.1, Corollary~5.16, Chapter~3]{karatzas:shreve:1991} shows that $\mu_{n,p}^y$ and $\bar{\mu}_{n,p}^y$ are mutually absolutely continuous and in addition, $\bar{\mu}_{n,p}^y$-almost surely
\begin{multline}
\label{eq:girsanov_base}
\density{\bar{\mu}^y_{n,p}}{\mu^y_{n,p}}= \exp \left(\frac{1}{2} \int_{\GaStep_n}^{\GaStep_p}
  \ps{\nabla U(\YLbar_s(y)) - \overline{\nabla U}(\YLbar(y),s)}{\rmd \YLbar_s(y)} \right.  \\
\left. - \frac{1}{4} \int_{\GaStep_n}^{\GaStep_p}  \left\{ \norm[2]{\nabla U(\YLbar_s(y))} - \norm[2]{\overline{\nabla U}(\YLbar(y),s)} \right\} \rmd s \right) \eqsp.
\end{multline}
Under \Cref{assum:regularity},
\eqref{eq:girsanov_base} implies for all $y \in \rset^d$:
\begin{align}
\nonumber
\KL(\mu_{n,p}^y \vert \bar{\mu}_{n,p}^y)
&\leq 4^{-1} \int_{\GaStep_n}^{\GaStep_p} \expe{\norm[2]{\nabla U(\YLbar_s(y)) - \overline{\nabla U}(\YLbar(y),s)}} \rmd s \\
\nonumber
& \leq 4^{-1} \sum_{k=n}^{p-1} \int_{\GaStep_k}^{\GaStep_{k+1}} \expe{\norm[2]{\nabla U(\YLbar_s(y)) -\nabla U(\YLbar_{\GaStep_k}(y))}} \rmd s\\
\label{eq:pinsker_girsanov_2}
&\leq 4^{-1} \LU^2 \sum_{k=n}^{p-1} \defEns{(\gaStep_{k+1}^3/3) \int_{\rset^d}\norm[2]{\nabla U(z)} \QKer_{\gaStep}^{n+1,k}(y,\rmd z) + d \gaStep_{k+1}^2} \eqsp.
\end{align}
By the Pinsker inequality, $\tvnorm{\delta_y \QKer_{\gaStep}^{n+1,p} - \delta_y \PL_{\GaStep_{n+1,p}} } \leq \sqrt{2}\{\KL(\mu_{n,p}^y \vert \bar{\mu}_{n,p}^y)\}^{1/2}$.
The proof is concluded by combining this inequality, \eqref{eq:pinsker_girsanov_2} and \eqref{eq:ergodicity-langevin-expo} in \eqref{eq:triangle_ineq_TV}.
%  we have for all $x \in \rset^d$, $n \geq 0$, $p \geq 1$, $n < p$,
% \begin{equation}
%   \label{eq:eq_base}
%   \tvnorm{\delta_x \QKer_{\gaStep}^p -\pi} \leq
% 2^{-1/2} \LU \parenthese{ \sum_{k=n}^{p-1} \defEns{(\gaStep_{k+1}^3/3) A(\gaStep,x) + d \gaStep_{k+1}^2}}^{1/2} + C(\delta_x \QKer_{\gaStep}^n) \kappa^{\GaStep_{n+1,p}} \eqsp,
% \end{equation}
% where $\kappa, C(\delta_x \QKer_{\gaStep}^n)$ are defined in \eqref{eq:ergodicity-langevin-expo} and
% \begin{align}
% \label{eq:eq_base1}
% A(\gaStep,x) = \sup_{k \geq 0}  \int_{\rset^d}\norm[2]{\nabla U(z)} \QKer_{\gaStep}^{k}(y,\rmd z) \eqsp.
% \end{align}
\end{proof}
In the sequel, depending on the conditions on the potential $U$ and the techniques of proof,
for any given $x \in \rset^d$, $C(\delta_x \QKer_{\gaStep}^n)$ can
have two kinds of upper bounds, either of the form
$-\log(\gaStep_n)W(x)$, or $\exp(a \Gamma_n) W(x)$, for some function
$W :\rset^d \to \rset$ and $a > 0$. In both cases, as shown in
\Cref{propo:limit-zeros-tv-1}, it is possible to choose $n$ as a
function of $p$, so that $\lim_{p \to \plusinfty} \tvnorm{\delta_x
  \QKer_{\gaStep}^p-\pi}=0$ under appropriate conditions on the
sequence of step sizes $(\gaStep_{k})_{k \geq 1}$.
%\tcr{However for applications, it
%  should be mentioned that using strictly decreasing step sizes leads
%  to statistical estimator which can have high asymptotic variance}.
\begin{proposition}
\label{propo:limit-zeros-tv-1}
Assume that \Cref{assum:regularity} and \eqref{eq:ergodicity-langevin-expo} hold. Let $(\gaStep_k)_{ k\geq 1}$ be a
nonincreasing sequence satisfying $\lim_{k \to \plusinfty} \GaStep_k
=\plusinfty$ and  $\lim_{k \to \infty} \gaStep_k= 0$.
 Then, $\lim_{n \to \infty} \tvnorm{\delta_x \QKer_{\gaStep}^n -
  \pi}=0$ for any $x \in \rset^d$ for which one of the two following conditions holds:
\begin{enumerate}[label=(\roman*)]
\item
\label{propo:limit-zeros-tv-1a}
$ A(\gaStep,x) < \infty$ and $ \limsup_{n \to \plusinfty} C(\delta_x
\QKer_{\gaStep}^n)/ (-\log(\gaStep_n)) < \plusinfty$, where $A(\gaStep,x)$ is defined in \eqref{eq:eq_base1}.
% \begin{equation*}
% A(\gaStep,x) < \infty \quad \text{ and } \quad \limsup_{n \to \plusinfty} C(\delta_x
% \QKer_{\gaStep}^n)/ (-\log(\gaStep_n)) < \plusinfty \eqsp,
% \end{equation*}
\item
\label{propo:limit-zeros-tv-1b}
 $ \sum_{k =1}^{\infty} \gaStep_k^2 <
\plusinfty$, $A(\gaStep,x) < \infty$ and $\limsup_{n \to \plusinfty} \log\{C(\delta_x
\QKer_{\gaStep}^n) \}/ \GaStep_n < \plusinfty$.
% \begin{equation*}
% A(\gaStep,x) < \infty \quad \text{ and } \quad \limsup_{n \to \plusinfty} \log\{C(\delta_x
% \QKer_{\gaStep}^n) \}/ \GaStep_n < \plusinfty \eqsp.
% \end{equation*}
\end{enumerate}
\end{proposition}
\begin{proof}
\begin{enumerate}[label=(\roman*),wide=0pt, labelindent=\parindent]
\item
There exists $\rankP \geq 1$ such that for all $p \geq \rankP$, $\kappa^{\gaStep_p} > \gaStep_p$ and $\kappa^{\GaStep_p} \leq \gamma_1$. Therefore, we can define for all $p \geq \rankP$,
\begin{equation}
\label{eq:def_nFunp}
\nFunp \eqdef \min \defEns{k \in \defEns{0,\cdots,p-1} | \kappa^{\GaStep_{k+1,p}} >  \gaStep_{k+1}} \eqsp.
\end{equation}
and $\nFunp \geq 1$. We first show that $ \liminf_{p \to \infty} \nFunp= \infty$. The proof goes by contradiction. If  $ \liminf_{p \to \infty} \nFunp < \infty$ we could extract a bounded subsequence $(n(p_k))_{k \geq 1}$. For such sequence, $(\gaStep_{n(p_k)+1})_{k \geq 1}$ is
bounded away from $0$,
%lower bounded by the positive quantity $\inf\{ \gamma_{i} \eqsp, i \leq \max_{k \geq 1} n(p_k) \}$
but $\lim_{k \to \plusinfty} \kappa^{\GaStep_{n(p_k)+1,p_k}} = 0$ which yields to a contradiction.
The definition of $\nFunp$ implies that $\kappa^{\GaStep_{\nFunp,p}} \leq \gaStep_{\nFunp}$, showing that
\begin{multline*}
\limsup_{p \to \plusinfty} C(\delta_x \QKer_{\gaStep}^{\nFunp}) \kappa^{\GaStep_{\nFunp,p}}\\ \leq \limsup_{p \to \plusinfty} \frac{C(\delta_x
\QKer_{\gaStep}^{\nFunp})}{ -\log(\gaStep_{\nFunp})} \,  \limsup_{p \to \plusinfty} \defEns{\gaStep_{\nFunp}(-\log(\gaStep_{\nFunp}))}= 0  \eqsp.
\end{multline*}
On the other hand, since $(\gaStep_k)_{k \geq 1}$ is nonincreasing, for any $\ell \geq 2$,
\[
\sum_{k=\nFunp+1}^{p} \gaStep_k^\ell \leq \gaStep_{\nFunp+1}^{\ell-1} \GaStep_{\nFunp+1,p} \leq \gaStep_{\nFunp+1}^{\ell-1} \log(\gaStep_{\nFunp+1})/\log(\kappa) \eqsp.
\]
The proof follows from \eqref{eq:eq_base} using $\lim_{p \to \infty} \gamma_{\nFunp}= 0$.
\item
  For all $p \geq 1$, define $n(p)=
  \max(0,\floor{\log(\GaStep_p)})$. Note that since $\lim_{k \to \plusinfty}
  \GaStep_k = \plusinfty$, we have  $\lim_{p \to \plusinfty} n(p) =
  \plusinfty$. Using $ \sum_{k
    =1}^{\plusinfty} \gaStep_k^2 < \plusinfty$ and $(\gaStep_k)_{k \geq 1}$ is a
  nonincreasing sequence, we get for all $\ell \geq 2$,
\[
\lim_{p \to
    \plusinfty} \sum_{k=n(p)}^p \gaStep^{\ell}_k = 0 \eqsp,
\]
 which shows that
  the first term in the right side of \eqref{eq:eq_base} goes to $0$ as
  $p$ goes to infinity. As for the second term, since $\limsup_{n \to \plusinfty} \log\{C(\delta_x
\QKer_{\gaStep}^n) \}/ \GaStep_n < \plusinfty$, we get
  using that $(\gaStep_k)_{k \geq 1}$ is nonincreasing and $n(p) \leq \log(\GaStep_p)$,
 \begin{align*}
&   C(\delta_x \QKer_{\gaStep}^{n(p)}) \kappa^{\GaStep_{n(p),p}}\\
%\leq a_1 \exp \defEns{\log(\kappa)\GaStep_p + (a_2-\log(\kappa)) \GaStep_{n(p)}}
&\leq \exp \parenthese{\log(\kappa)\GaStep_p + \parentheseDeux{ \{\log(C(\delta_x \QKer_{\gaStep}^{n(p)}))/\GaStep_{n(p)} \}_+
-\log(\kappa)} \GaStep_{ n(p)}}\\
&\leq \exp \parenthese{\log(\kappa)\GaStep_p + \parentheseDeux{\sup_{k \geq 1} \{ \log(C(\delta_x \QKer_{\gaStep}^k))/\GaStep_k \}_+
-\log(\kappa)} \gaStep_1 \log(\GaStep_p)}
 \eqsp.
 \end{align*}
 Using $\kappa < 1$ and  $\lim_{k \to \plusinfty} \GaStep_k = \plusinfty$, we have
 $\lim_{p \to \plusinfty}C(\delta_x \QKer_{\gaStep}^{n(p)})
 \kappa^{\GaStep_{n(p),p}} = 0$, which concludes the proof.
\end{enumerate}
\end{proof}

Using \eqref{eq:eq_base}, we can also assess the convergence of the algorithm for
constant step sizes $\gaStep_k = \gaStep$ for all $k \geq 1$. Two different kinds of results can be derived. First, for a given
precision $\varepsilon > 0$, we can try to optimize the step size $\gamma$ to minimize the number of iterations $p$ required to achieve $\tvnorm{\delta_x \QKer_{\gaStep}^p - \pi} \leq \varepsilon$.
Second if the total number of iterations is fixed $p \geq 1$, we may determine the step size $\gaStep >0$ which minimizes $\tvnorm{\delta_x \QKer_{\gaStep}^{p}- \pi }$.
\begin{lemma}
  \label{propo:precision}
Assume that \eqref{eq:eq_base} holds.
  Assume that there exists $\bargaStep > 0$ such that
%for any $\gamma \in \ocint{0,\bargaStep}$ \eqref{eq:eq_base} holds for  $\gaStep_k= \gaStep$ for all $k \geq 1$. In addition, assume that
 $\bar{C}(x) = \sup_{\gaStep \in \ocint{0,\bargaStep}} \sup_{n \geq 1} C(\delta_x \RKer_{\gaStep}^n)
  <\plusinfty$ and  $\sup_{\gaStep \in \ocint{0,\bargaStep}} A(\gaStep,x) \leq \bar{A}(x)$, where $ C(\delta_x \RKer_{\gaStep}^n)$ and $A(\gaStep,x)$ are defined in \eqref{eq:ergodicity-langevin-expo} and \eqref{eq:eq_base1} respectively. Then for all $\varepsilon >0$, we get $\tvnorm{\delta_x  \RKer_{\gaStep}^p-\pi} \leq \varepsilon$ if
\begin{equation}
  \label{eq:precision}
  p > T \gaStep^{-1} \quad \text{ and } \quad \gaStep \leq \frac{-d+\sqrt{d^2 +(2/3) \bar{A}(x) \varepsilon^2 (L^2T)^{-1} }}{2 \bar{A}(x)/3} \wedge \bargaStep \eqsp,
\end{equation}
where
\begin{equation*}
  T = \parenthese{ \log\{ \bar{C}(x) \}-\log(\varepsilon/2)} \Big/(- \log(\kappa)) \eqsp.
\end{equation*}
\end{lemma}

\begin{proof}
For $ p > T \gaStep^{-1}$, set $n = p - \floor{T \gaStep^{-1}}$. Then using the stated expressions of $\gaStep$ and $T$ in \eqref{eq:eq_base} concludes the proof.
%The proof is postponed to \Cref{proof:precision}.
\end{proof}

Note that an upper bound for $\gamma$ defined in \eqref{eq:precision}
is $\epsilon^2 (L^2Td)^{-1}$. The dependency of $T$ on the dimension $d$ will be addressed in \Cref{sec:pract-cond-geom}.%, which implies that $\gamma$ has to be at least smaller than $d^{-1}$ as $d$ goes to infinity.

\begin{lemma}
  \label{propo:fixed_budget}
Assume that \Cref{assum:regularity} and \eqref{eq:ergodicity-langevin-expo} hold.
  In addition, assume that there exist $\bargaStep > 0$  and
  $n \in \nset$, $n >0$, such that $ \bar{C}_n(x)= \sup_{\gaStep \in \ocint{0,\bargaStep}} C(\delta_x \RKer_{\gaStep}^n)
  <\plusinfty$ and  $\sup_{\gaStep \in \ocint{0,\bargaStep}} A(\gaStep,x) \leq \bar{A}(x)$. For all $p >n$ and all $x \in \rset^d$, if $\gaStep = \log(p-n)\{ (p-n) (-\log(\kappa)) \}^{-1} \leq \bargaStep$, then
  \begin{multline*}
    \tvnorm{\delta_x \RKer_{\gaStep}^p - \pi} \\
\leq (p-n)^{-1/2}\{\bar{C}_n(x) (p-n)^{-1/2}+ \log(p-n)(d+ \bar{A}(x)\log(p-n)(p-n)^{-1})^{1/2} \} \eqsp.
  \end{multline*}
\end{lemma}

\begin{proof}
The proof is a  straightforward calculation using \eqref{eq:eq_base}.
\end{proof}

To get quantitative bounds for the total variation distance $\tvnorm{\delta_x \QKer_{\gaStep}^p -\pi}$
it is therefore required to get bounds on  $\kappa$, $A(\gaStep,x)$ and to control $C(\delta_x \QKer_\gaStep^n)$.
We will consider in the sequel  two different approaches to get \eqref{eq:ergodicity-langevin-expo}, one based on functional inequalities, the other on
coupling techniques.  We will
consider also increasingly stringent assumptions for the potential
$U$. Whereas we will always obtain the same type of
exponential bounds, the dependency of the constants on the dimension
will be markedly different. In the worst case, the dependency is exponential. It is polynomial when $U$ is convex.

\section{Practical conditions for geometric ergodicity of the Langevin diffusion and their consequences for ULA}
\sectionmark{Practical conditions}
\label{sec:pract-cond-geom}
\subsection{Superexponential densities}
\label{subsec:superexponential}
Assume first that the potential is superexponential outside a ball. This is a rather weak assumption (we do not assume convexity here).
\begin{assumption}
\label{assum:drift-condition-reinforced}
The potential $U$ is twice continuously differentiable  and there exist  $\rho >0$, $\alpha \in \ocint{1,2}$ and $M_\rho \geq 0$ such that for all $x \in \rset^d$, $\norm{x - \globmin } \geq M_\rho$, $\ps{\nabla U(x)}{x - \globmin} \geq \rho \norm{x -\globmin}^{\alpha}$.
\end{assumption}
The price to pay will be constants which are exponential in the dimension.
Under \Cref{assum:drift-condition-reinforced}, the potential $U$ is unbounded off compact set. Since $U$ is continuous, it has a global minimizer $\globmin$, which is a point at which $\nabla U(\globmin)= 0$. Without loss of generality, it is assumed that $U(\globmin)= 0$.
\begin{lemma}
  \label{lem:majo:U}
  Assume \Cref{assum:regularity} and \Cref{assum:drift-condition-reinforced}. Then for all $x \in \rset^d$,
\begin{equation}
\label{eq:calpha}
U(x) \geq \rho \norm[\alpha]{x - \globmin}/(\alpha+1) - \Calpha \quad \text{with} \quad \Calpha =  \rho \Mrho^{\alpha}/(\alpha+1)  + \Mrho^2 \LU/2   \eqsp.
\end{equation}
\end{lemma}
\begin{proof}
The elementary proof is postponed to \Cref{sec:proof:majo:U}.
\end{proof}
Following \cite[Theorem~2.3]{roberts:tweedie:1996}, we first establish a drift condition for the diffusion.
\begin{proposition}
\label{prop:poincare-constant-lyapunov}
\label{coro:poincare-constant-lyapunov}
Assume \Cref{assum:regularity} and \Cref{assum:drift-condition-reinforced}.
For any $\varsigma \in \ooint{0,1}$, the drift condition \eqref{eq:lyapunov-condition-V} is satisfied with the Lyapunov function $\VSAlpha(x) = \exp(\varsigma U(x))$,
    $\thetaSAlpha = \varsigma dL$, $\Setdrift_{\varsigma} = \ball{\xstar}{\KSAlpha}$, $\KSAlpha = \max(\{2dL/(\rho(1-\varsigma))\}^{1/(2(\alpha-1))}, M_\rho)$ and $\bSAlpha = \varsigma dL  \sup\nolimits_{\{y \in \Setdrift_{\varsigma} \}}\{ \VSAlpha(y)\}$.
Moreover, there exist constants
$\ConstVnorm < \infty$ and $\rateExpVnorm >0$ such that for all $t \in \rset_+$ and
probability measures $\mu_0$ and $\nu_0$ on $(\rset^d, \mathcal{B}(\rset^d))$, satisfying $\mu_0(\VSAlpha) +
\nu_0(\VSAlpha) < \plusinfty$,
\[
\Vnorm[\VSAlpha]{\mu_0 \PL_t - \nu_0 \PL_t} \leq \ConstVnorm \rme^{-\rateExpVnorm t}\Vnorm[\VSAlpha]{\mu_0 - \nu_0} \eqsp,
% \text{ and }
\
\Vnorm[\VSAlpha]{\mu_0 \PL_t - \pi} \leq \ConstVnorm \rme^{-\rateExpVnorm t}\mu_0(\VSAlpha)
 \eqsp.
\]
\end{proposition}
\begin{proof}
The proof, adapted from \cite[Theorem~2.3]{roberts:tweedie:1996} and \cite[Theorem 6.1]{meyn:tweedie:1993:3}, is postponed to \Cref{proof:prop:poincare-constant-lyapunov}.
\end{proof}
%The constants $\ConstVnorm$ and $\rateExpVnorm$ depend on the drift and minorization conditions.
Under \Cref{assum:drift-condition-reinforced}, explicit expressions for  $\ConstVnorm$ and $\rateExpVnorm$ have been developed in the literature but these estimates are in general very
conservative. We now turn to establish \eqref{eq:foster-lyapunov-drift} for the Euler discretization.

\begin{proposition}
\label{prop:drift-R-gaStep}
Assume   \Cref{assum:regularity} and \Cref{assum:drift-condition-reinforced}.
Let $\bargaStep \in \ooint{0,L^{-1}}$. For all $\gaStep \in \ocint{0,\bargaStep}$ and $x \in \rset^d$,  $\RKer_{\gaStep}$ satisfies the drift condition \eqref{eq:foster-lyapunov-drift} with $\Vly(x) = \exp(U(x)/2)$,  $\KDriftRgaStep=\max ( M_\rho , (8\log(\lDriftgaStep)/\rho^2)^{1/(2(\alpha-1))})$, $\bDriftRgaStep = -2 \log(\lDriftgaStep)  \lDriftgaStep^{-\bargaStep} \sup\nolimits_{\{y \in \ball{\xstar}{ \KDriftRgaStep}\}} \Vly(y)$ and $\lDriftgaStep = \rme^{-dL / \{2(1-L \bargaStep)\}}$.
\end{proposition}
\begin{proof}
The proof is postponed to \Cref{proof:prop:drift-R-gaStep}.
\end{proof}

\begin{theorem}
\label{theo:convergence_TV_non_quantitatif}
Assume \Cref{assum:regularity} and \Cref{assum:drift-condition-reinforced}. Let $(\gaStep_k)_{ k\geq 1}$ be a nonincreasing sequence with $\gaStep_1 < \bargaStep$,  $\bargaStep \in \ooint{0,\LU^{-1}}$. Then,
for all $n\geq 0$, $p \geq 1$, $n<p$, and $x \in \rset^d$, \eqref{eq:eq_base} holds with $ \log(\kappa) = -\rateExpVnormUndemi$ and
\begin{equation*}
%\label{eq:definition-B-n-p-1}
%\nonumber
A(\gaStep,x) \leq L^{2} \parenthese{\frac{\alpha+1}{\rho}\parentheseDeux{\Calpha + \frac{4(2-\alpha)(\alpha+1)}{\alpha \rho} +2 \log\defEns{\FunfoD(\lDriftgaStep, \bDriftRgaStep,\gaStep_1,\Vly(x)) }}}^{2/\alpha}
\end{equation*}
\begin{equation}
 \label{eq:definition-B-n-p-0}
%\nonumber
C(\delta_x \QKer_{\gaStep}^n) \leq \ConstVnormUndemi \Funfo(\lDriftgaStep,\GaStep_{1,n}, \bDriftRgaStep, \gaStep_1,\Vly(x)) \eqsp,
\end{equation}
where $\ConstVnormUndemi$, $\rateExpVnormUndemi$ are given by \Cref{coro:poincare-constant-lyapunov}, $\Funfo$ by \eqref{eq:def_Funfo}, $\Vly$, $\lDriftgaStep$, $\bDriftRgaStep$ in \Cref{prop:drift-R-gaStep}, $\FunfoD$ by \eqref{eq:def_FunfoD}, $\Calpha$   in \eqref{eq:calpha}.
\end{theorem}
\begin{proof}
The proof is postponed to \Cref{proof:theo:convergence_TV_non_quantitatif}.
\end{proof}
\Cref{eq:definition-B-n-p-0} implies that for all $x \in \rset^d$, we have $\sup_{n \geq 0} C(\delta_x \QKer_{\gaStep}^n) \leq \FunfoD(\lDriftgaStep, \bDriftRgaStep, \gaStep_1,\Vly(x))$, so \Cref{propo:limit-zeros-tv-1}-\ref{propo:limit-zeros-tv-1a} shows that $\lim_{p \to \plusinfty} \tvnorm{\delta_x \QKer_{\gaStep}^p-\pi} = 0$ for all $x \in \rset^d$
provided that $\lim_{k \to \plusinfty} \gaStep_k = 0$ and $\lim_{k \to \plusinfty} \GaStep_k = \plusinfty$. In addition, for the case of constant step size $\gaStep_k = \gaStep$ for all $k \geq 1$, \Cref{propo:precision} and \Cref{propo:fixed_budget} can be applied.

Let $V: \rset^d \to \rset$, defined for all $x \in \rset^d$ by
$V(x)=\exp(U(x)/2)$. By \Cref{coro:poincare-constant-lyapunov},
$(\PL_t)_{t \geq 0}$ is a contraction operator on the space of finite
signed measure $\mu \in \measSet_0$, $\mu(V^{1/2}) < \plusinfty$,
endowed with the norm $\Vnorm[V^{1/2}]{\cdot}$. It is therefore
possible to control $\Vnorm[V^{1/2}]{\delta_x \QKer_{\gaStep}^p
  -\pi}$. To simplify the notations, we limit our discussion to
constant step sizes.
\begin{theorem}
\label{theo:convergence_TV_gaStep_const}
\label{coro:asympt_bias}
Assume \Cref{assum:regularity} and \Cref{assum:drift-condition-reinforced}. Then, for all $p \geq 1$, $x \in \rset^d$
 and $\gaStep \in \ooint{0,\LU^{-1}}$, we have
\begin{equation}
\label{eq:convergence_TV_gaStep_const_1}
  \Vnorm[\Vlysqrt]{\delta_x \RKer_{\gaStep}^p -\pi} \leq \ConstVnormUnquart \kappa^{\gaStep p} \Vlysqrt(x)  + \boundDeuxTVa \eqsp,
\end{equation}
where $\log(\kappa)= -\rateExpVnormUnquart$, $\ConstVnormUnquart,\rateExpVnormUnquart, \thetaUndemi,\bUndemi$ are defined in \Cref{coro:poincare-constant-lyapunov}, $V,\lDriftgaStep,\bDriftRgaStep$ in \Cref{prop:drift-R-gaStep}, $\FunfoD$ in \eqref{eq:def_FunfoD}  and
\begin{multline*}
%\label{eq:theo:convergence_TV_gaStep_const}
  \boundDeuxTVaccarre=
 L^2 \max(1,\ConstVnormUnquart^2)(1+\gaStep)(1-\kappa)^{-2} \parenthese{2\FunfoD(\lDriftgaStep,\bDriftRgaStep,\gaStep,v) +  \bUndemi /\thetaUndemi} \\
 \times
   \parenthese{  \gaStep  d+3^{-1}\gaStep^2 \Vnorm[\Vlyb^{1/2}]{\nabla U}^2 \, \FunfoD(\lDriftgaStep,\bDriftRgaStep, \gaStep,v)  } \eqsp.
\end{multline*}
Moreover, $\RKer_{\gaStep}$ has a unique invariant distribution $\pi_{\gaStep}$ and
\[
\Vnorm[\Vlysqrt]{\pi-\pi_{\gaStep}} \leq \boundDeuxTVab \eqsp.
\]
\end{theorem}
\begin{proof}
  The proof of \eqref{eq:convergence_TV_gaStep_const_1} is postponed
  to \Cref{proof:theo:convergence_TV_gaStep_const}.  The bound for
  $\Vnorm[\Vlysqrt]{\pi-\pi_{\gaStep}}$ is an easy consequence of
  \eqref{eq:convergence_TV_gaStep_const_1}: by
  \Cref{propo:drift-Euler} and
  \cite[Theorem~16.0.1]{meyn:tweedie:2009}, $\RKer_{\gaStep}$ is
  $\Vlysqrt$-uniformly ergodic: $\lim_{p \to \plusinfty}
  \Vnorm[\Vlysqrt]{\delta_x \RKer_{\gaStep}^p-\pi_{\gaStep}} = 0$ for
  all $x \in \rset^d$. Finally,
  \eqref{eq:convergence_TV_gaStep_const_1} shows that for all $x \in
  \rset^d$,
\begin{equation*}
\Vnorm[\Vlysqrt]{\pi-\pi_{\gaStep}} \leq \lim_{p \to \plusinfty}\defEns{ \Vnorm[\Vlysqrt]{\delta_x  \RKer_{\gaStep}^p-\pi}+  \Vnorm[\Vlysqrt]{\delta_x  \RKer_{\gaStep}^p-\pi_{\gaStep}} } \leq \boundDeuxTVa \eqsp.
\end{equation*}
Taking the minimum over $x \in \rset^d$ concludes the proof.
\end{proof}
  Note that \Cref{theo:convergence_TV_gaStep_const} implies that there exists
  a constant $C\geq 0$ which does not depend on $\gaStep$ such that $
  \Vnorm[\Vlysqrt]{\pi-\pi_{\gaStep}} \leq C \gaStep^{1/2}$.
\begin{remark}
  It is shown in \cite[Theorem 4]{talay:tubaro:1991} that for $\phi
  \in C^{\infty}(\rset^d)$ with polynomial growth,
  $\pi_{\gaStep}(\phi) - \pi(\phi) = b(\phi) \gaStep +
  \bigO(\gaStep^2)$, for some constant $b(\phi) \in \rset$, provided
  that $U \in C^{\infty}(\rset^d)$ satisfies \Cref{assum:regularity}
  and \Cref{assum:drift-condition-reinforced}. Our result does not
  match this bound since $\boundDeuxTVab=
  \bigO(\gaStep^{1/2})$. However the bound $\boundDeuxTVab$ is uniform
  over the class of measurable functions $\phi$ satisfying for all
  $x\in \rset^d$, $\abs{\phi(x)} \leq V^{1/2}(x)$. Obtaining such
    uniform bounds in total variation is important in Bayesian
    inference, for example to compute high posterior density credible regions.
  Our result also strengthens and completes \cite[Corollary
  7.5]{mattingly:stuart:higham:2002}, which states that under
  \Cref{assum:drift-condition-reinforced} with $\alpha =2$, for any
  measurable functions $\phi : \rset^d \to \rset$ satisfying for all
  $x,y \in \rset^d$,
\begin{equation*}
  \abs{\phi(x)-\phi(y)} \leq C \norm{x-y}\{ 1+\norm{x}^{k} + \norm{y}^k\}  \eqsp,
\end{equation*}
for some $C \geq 0$, $k \geq 1$,
 $\abs{\pi_{\gaStep}(\phi) - \pi(\phi)}
\leq C \gaStep^{\chi}$ for some constants $C \geq 0$ and $\chi \in
\ooint{0,1/2}$, which does not depend on $\phi$.
\end{remark}
% If we assume that the total number of
% iterations $p$ is fixed (finite horizon setting), a straightforward
% optimization in the step size $\gaStep$ shows that $\gamma$ should be
% chosen as $\gamma= 1/2 \, \log(p)/ p \log(\kappa)$ (inversely
% proportional to the number of iterations up to a logarithmic factor)
% and that the associated bound is of order $\bigO((\log(p)
% p^{-1})^{1/2})$.  Using the "doubling trick" (see
% \cite{hazan:kale:2014}) this can be converted to an anytime algorithm
% with a bound on $\tvnorm{\delta_x \RKer^p_\gaStep - \pi}$ of order
% $\bigO((\log(p) p^{-1})^{1/2})$ for all $x \in \rset^d$ and $p \geq
% 1$.

The bounds in \Cref{theo:convergence_TV_non_quantitatif} and \Cref{theo:convergence_TV_gaStep_const} depend upon the constants  appearing in \Cref{coro:poincare-constant-lyapunov} which are computable but are known to be pessimistic in general; see \cite{roberts:tweedie:2000}.
 More explicit  rates of convergence for the semigroup can be obtained using  Poincar{\'e} inequality; see~\cite{bakry:cattiaux:guillin:2008}, \cite{cattiaux:guillin:2009} and \cite[Chapter~4]{bakry:gentil:ledoux:2014} and the references therein. The probability measure
$\pi$ is said to satisfy a Poincar{\'e} inequality with the constant  $\CPoincare$ if for every locally Lipschitz function $h$,
\begin{equation}
\label{eq:poincare}
  \VarDeux{\pi}{h} \leq \CPoincare \int_{\rset^d} \norm[2]{\nabla h(x)} \pi(\rmd x) \eqsp.
\end{equation}
This inequality implies by \cite[Theorem~2.1]{cattiaux:guillin:2009} that for all $t \geq 0$ and any initial distribution $\mu_0$, such that $\mu_0 \ll \pi$,
\begin{equation}
\label{eq:borne-poincare-TV}
\norm{\mu_0 \PL_t - \pi}_{\TV} \leq \exp(- t/ \CPoincare) \left( \VarDeux {\pi}{\rmd \mu_0/ \rmd \pi}\right)^{1/2} \eqsp.
\end{equation}
 \cite[Theorem~1.4]{bakry:barthe:cattiaux:guillin:2008} shows that if the Lyapunov condition \eqref{eq:lyapunov-condition-V} is satisfied,  then the Poincar{\'e} inequality \eqref{eq:poincare} holds with an explicit constant. Denote by
\begin{equation}
\label{eq:bornevar}
\bornevar \eqdef  \parenthese{4 \uppi \defEns{\prod_{k=1}^n (1-L \gaStep_k)}^2\sum_{i=1}^n \gaStep_i(1-L\gaStep_i)^{-1}}^{-d/2} \eqsp.
\end{equation}
\begin{theorem}
\label{theo:convergence_TV_quantitatif_Poincare}
Assume \Cref{assum:regularity} and  \Cref{assum:drift-condition-reinforced}. Let $(\gaStep_k)_{k \geq 1}$ be a non increasing sequence. Then for all $n \geq 1$ and $x \in \rset^d$, \Cref{eq:ergodicity-langevin-expo} holds with
\begin{align*}
&\log(\kappa) = \parenthese{- \thetaUndemi^{-1} \left\{ 1 + (4\bUndemi \KUndemi^2/ \uppi^2) \rme^{\osc[\KUndemi]{U}} \right\}}^{-1} \eqsp, \\
&C(\delta_x \QKer^n_{\gaStep}) \leq \frac{ (\alpha+1)^d (2\pinumber)^{(d+1)/2} (d-1)!}{\rho ^d \Gammabf((d+1)/2)} \bornevar  \rme^{ \Calpha} \, \rme^{U(x)}
\eqsp,
\end{align*}
where $\Gammabf$ is the Gamma function and the constants $\bUndemi ,\thetaUndemi,\KUndemi, \Calpha$ are given in \Cref{prop:poincare-constant-lyapunov} and \eqref{eq:calpha} respectively.
\end{theorem}
\begin{proof}
The proof is postponed to \Cref{sec:theo:convergence_TV_quantitatif_Poincare}.
\end{proof}
Note that for all $x \in \rset^d$, $C(\delta_x \QKer^n_{\gaStep})$ satisfies the conditions of \Cref{propo:limit-zeros-tv-1}-\ref{propo:limit-zeros-tv-1b}. Therefore using in addition the bound on $A(\gaStep,x)$ for all $x \in \rset^d$ and $\gaStep \in \ooint{0,L^{-1}}$ given in \Cref{theo:convergence_TV_non_quantitatif}, we get $\lim_{k \to \plusinfty } \tvnorm{\delta_x \QKer^p_{\gaStep}-\pi} = 0$ if $\lim_{n \to \plusinfty} \GaStep_n = \plusinfty$ and $\lim_{n \to \plusinfty} \sum_{k=1}^n \gaStep_k^2 < \plusinfty$.
\subsection{Log-concave densities}
\label{subsec:log-concave}
We now consider the following additional assumption.
\begin{assumption}
  \label{assum:convex}
  $U$ is convex and admits a minimizer $\xstar$ for $\VUl$. Moreover there exist $\rhoUl >0$ and $\RUl \geq 0$ such that for all $x \in \rset^d$, $\norm{x-\xstar} \geq \RUl$,
\begin{equation}
\label{eq:superexpo_potential}
\VUl(x) - \VUl(\xstar) \geq \rhoUl \norm{x-\xstar} \eqsp.
\end{equation}
\end{assumption}
It is shown in \cite[Lemma 2.2]{bakry:barthe:cattiaux:guillin:2008}
that if $U$ satisfies \Cref{assum:regularity} and is convex, then
\eqref{eq:superexpo_potential} holds for some constants $\rhoUl,\RUl$ which
depend in an intricate way on $U$. Since the constants $\rhoUl,\RUl$ appear explicitly in the bounds we derive, we must assume that these constants are explicitly computable. We still assume in this
section that $U(\xstar)=0$.  Define the function $\VEa : \rset^d \to
\coint{1,\plusinfty}$ for all $x \in \rset^d$ by
\begin{equation}
\label{eq:def_Wc}
\VEa(x) = \exp((\rhoU/4)(\norm{x-\xstar}^2+1)^{1/2}) \eqsp.
\end{equation}
We now derive a drift inequality for $\RKer_{\gaStep}$ under \Cref{assum:convex}.
\begin{proposition}
  \label{propo:drift-Euler}
Assume \Cref{assum:regularity} and \Cref{assum:convex}. Let $\bargaStep \in \ocint{0,\LUl^{-1}}$.
Then for all   $\gaStep \in \ocint{0,\bargaStep}$,
$\VEa$ satisfies \eqref{eq:foster-lyapunov-drift} with $\thetaEa = \rme^{-2^{-4} \rhoUl^2(2^{1/2}-1)} $, $\KUa = \max(1,2 d / \rhoUl,\RUl)$,
\begin{equation}
\label{eq:def_bEa}
  \bEa =  \defEns{ (\rhoUl/4)(d+(\rhoUl \bargaStep/4))  - \log( \thetaEa )}  \rme^{\rhoUl (\KUa^2+1)^{1/2}/4 + (\rhoUl \bargaStep/4)(d+(\rhoUl \bargaStep/4))} \eqsp.
\end{equation}
\end{proposition}
\begin{proof}
  The proof is postponed to \Cref{sec:proof-drift-conv}
\end{proof}

\begin{corollary}
\label{coro:convex_A_gamma}
Assume \Cref{assum:regularity} and \Cref{assum:convex}. Let $(\gaStep_k)_{ k\geq 1}$ be a nonincreasing sequence with $\gaStep_1 \leq  \bargaStep$,  $\bargaStep \in \ocint{0,\LU^{-1}}$.
Then, for all $n\geq 0$, $p \geq 1$, $n<p$, and $x \in \rset^d$, 
\begin{equation}
\label{eq:borne_A_conv}
 A(\gaStep,x)  = \LU^2 \parenthese{4 \rhoUl^{-1} \parentheseDeux{ 1+\log \defEns{G(\thetaEa,\bEa,\gaStep_1,\VEa(x))}}}^2 \eqsp,
\end{equation}
where $A(\gaStep,x)$ is defined by \eqref{eq:eq_base1} and  $\FunfoD$,  $\VEa$, $\thetaEa$, $\bEa$,  are given in 
\eqref{eq:def_FunfoD},  \eqref{eq:def_Wc}, \Cref{propo:drift-Euler}  respectively.
\end{corollary}
\begin{proof}
The proof is postponed to \Cref{sec:proof-crefth-stepsb_coro}.
\end{proof}

If $U$ is convex, \cite[Theorem~1.2]{bobkov:1999} shows that $\pi$ satisfies a
Poincar{\'e} inequality with a constant depending only on the variance of $\pi$.
\begin{theorem}
\label{theo:convergence_TV_dec-stepsize-Bob}
Assume \Cref{assum:regularity} and \Cref{assum:convex}. Let $(\gaStep_k)_{ k\geq 1}$ be a nonincreasing sequence with $\gaStep_1 \leq  \bargaStep$,  $\bargaStep \in \ocint{0,\LU^{-1}}$.
Then, for all $n\geq 0$, $p \geq 1$, $n<p$, and $x \in \rset^d$, \eqref{eq:eq_base} holds with $A(\gaStep,x)$ given in \eqref{eq:borne_A_conv},
\begin{subequations}
\begin{align}
\label{eq:constant_bobkov}
\log(\kappa) &= \parenthese{-432 \int_{\rset^d} \norm[2]{ x - \int_{\rset^d} y \pi (\rmd y)} \pi(\rmd x)}^{-1} \\
\label{eq:Cdelta_xQ_gamma_bobkov}
C(\delta_x \QKer_{\gaStep}^n) & = \parenthese{\frac{(2\pinumber)^{(d+1)/2}(d-1)!}{\rhoU^d \Gammabf((d+1)/2)} +
   \frac{\pinumber^{d/2} \RUl^{d}}{\Gammabf(d/2+1)}} \bornevar \exp(U(x)) \eqsp,
\end{align}
\end{subequations}
where $\bornevar$ is given in \eqref{eq:bornevar}.
\end{theorem}
\begin{proof}
The proof is postponed to \Cref{sec:proof-crefth-stepsb}.
\end{proof}
For all $x \in \rset^d$, $C(\delta_x \QKer^n_{\gaStep})$ satisfies the conditions of \Cref{propo:limit-zeros-tv-1}-\ref{propo:limit-zeros-tv-1b}. Therefore, if $\lim_{n \to \plusinfty} \GaStep_n = \plusinfty$ and $\lim_{n \to \plusinfty} \sum_{k=1}^n \gaStep_k^2 < \plusinfty$, we get $\lim_{k \to \plusinfty } \tvnorm{\delta_x \QKer^p_{\gaStep}-\pi} = 0$.

There are two difficulties when applying
\Cref{theo:convergence_TV_dec-stepsize-Bob}. First the Poincaré
constant \eqref{eq:constant_bobkov} is in closed form but is not
computable, although it can be bounded by a $\bigO(d^{-2})$ . Second,
the bound of
$\VarDeuxLigne{\pi}{\rmd \delta_x \QKer_{\gaStep}^n / \rmd \pi}$  is likely to be suboptimal. To circumvent these two issues, we now give new
quantitative results on the convergence of $(\PL_t)_{t \geq 0}$ to
$\pi$ in total variation. Instead of using functional inequality, we
use in the proof the coupling by reflection, introduced in \cite{lindvall:rogers:1986}.  Define the function
$\Fsmall : \ooint{0,1} \times \rset_+^* \to \rset_+$ for all
$\epsilon \in \ooint{0,1}$ and $R \geq 0$, by
\begin{equation}
  \label{eq:Fsmall}
  \Fsmall(\epsilon,R) = R^{2}/ \defEns{2\bfPhi^{-1}(1-\epsilon/2)}^{2} \eqsp,
\end{equation}
where $\bfPhi$ is the cumulative distribution function of the standard
Gaussian distribution and $\bfPhi^{-1}$ is the associated quantile
function.   Before stating the theorem, we
first show that \eqref{eq:lyapunov-condition-V} holds and provide
explicit expressions for the constants which come into play. These
constants will be used to obtain the explicit convergence rate of the semigroup
$(\PL_t)_{t\geq 0}$ to $\pi$ which is derived in \Cref{theo:quantative-bound-lang-convex}.
\begin{proposition}
  \label{propo:drift_exp_conv}
Assume \Cref{assum:regularity} and \Cref{assum:convex}. Then $\VEa$ satisfies the drift condition \eqref{eq:lyapunov-condition-V} with
$\thetaOne= \rhoU^2/8$, $\Setdrift = \ball{\xstar}{\KOne}$, $\KOne = \max(1,\RUl,4d/ \rhoU)$ and
\begin{equation*}
%\label{eq:definition-b-rho}
%\hspace{-0.1cm}
\bOne =(\rhoU/4)\parenthese{(\rhoU/4)\KOne +d }
 \max \defEns{1,(\KOne^2 +1)^{-1/2} \exp(\rhoU( \KOne^2+1)^{1/2}/4)} \eqsp.
\end{equation*}
%where $\rhoU, \RUl$ are given in \eqref{eq:superexpo_potential}.
\end{proposition}
\begin{proof}
The proof is adapted from \cite[Corollary 1.6]{bakry:barthe:cattiaux:guillin:2008} and is postponed to \Cref{proof:prop:drift_exp_conv}.
\end{proof}
\begin{theorem}
\label{theo:quantative-bound-lang-convex}
Assume \Cref{assum:regularity} and \Cref{assum:convex}.   Then  for all $x \in \rset^d$, $ \norm{\delta_x \PL_t - \pi}_{\TV} \leq 2 \Lambda(x) \rme^{-\thetaOne t/4  } + 4\varpi^{t}$,
% \begin{equation*}
% \norm{\delta_x \PL_t - \pi}_{\TV} \leq \Lambda(x) \rme^{-\thetaOne t/4  } + 2\varpi^{t}\eqsp,
% \end{equation*}
where 
\begin{subequations}
\begin{align}
\label{eq:kappa-convex}
&\log(\varpi)
= -  \log(2) (\thetaOne/4) \\
\nonumber
& \phantom{aaaaa}\times \parentheseDeux{\log \defEns{\thetaOne^{-1} \bOne \parenthese{3+ 4 \rme^{4^{-1} \thetaOne \Fsmall\parenthese{2^{-1},(8/\rhoU) \log( 4 \thetaOne^{-1} \bOne)}}}} +\log(2) }^{-1}   \eqsp, \\
\label{eq:Cdelta_xQ_gamma_reflection}
&\Lambda(x) =  (1/2)( \VOne(x)+\thetaOne^{-1}\bOne) 
+ 2 \thetaOne^{-1} \bOne\rme^{4^{-1} \thetaOne \Fsmall(2^{-1},(8/\rhoU) \log( 4 \thetaOne^{-1} \bOne))} \eqsp,
% \label{eq:Cdelta_xQ_gamma_reflection}
% & \leq  4+(1/2)\defEns{\bOne \thetaOne^{-1} + \Funfo(\thetaEa,\Gammasum_{1,n},\bEa,\gaStep_1,\VEa(x))} \eqsp,
\end{align}
\end{subequations}
 the function $\VEa$ is defined in \eqref{eq:def_Wc}, the constants $\thetaOne,
\bOne$ in 
\Cref{propo:drift_exp_conv}.
% where the functions $\Funfo$ and $\VEa$ are defined in
% \eqref{eq:def_Funfo} and \eqref{eq:def_Wc}, the constants $ \thetaEa, \bEa,\thetaOne,
% \bOne$ in \Cref{propo:drift-Euler} and
% \Cref{propo:drift_exp_conv} respectively.
\end{theorem}
\begin{proof}
The proof is postponed to \Cref{sec:proof-crefth-bound-quant-conv}.
\end{proof}
Note that the bound, we obtain is a little different from
\eqref{eq:ergodicity-langevin-expo}. The initial condition is isolated
on purpose to get a better bound. A consequence of this result is the
following bound on the convergence of the sequence $(\delta_x
\QKer_{\gaStep}^n)_{n \geq 0}$ to $\pi$.
\begin{corollary}
\label{coro:quantative-bound-lang-convex}
Assume \Cref{assum:regularity} and \Cref{assum:convex}. Let $(\gaStep_k)_{k \geq 0}$ be a sequence of nonnegative
  step sizes. Then for all $x \in \rset^d$, $n \geq 0$, $p \geq 1$, $n < p$,
\begin{multline*}
%  \label{eq:eq_base}
  \tvnorm{\delta_x \QKer_{\gaStep}^p -\pi} \leq
2^{-1/2} \LU \parenthese{ \sum_{k=n}^{p-1} \defEns{(\gaStep_{k+1}^3/3) A(\gaStep,x) + d \gaStep_{k+1}^2}}^{1/2}
\\ +2 \Lambda(\delta_x \QKer_{\gaStep}^n) \rme^{-\thetaOne \Gamma_{n+1,p}/4  } + 4\varpi^{\Gamma_{n+1,p}} \eqsp,
\end{multline*}
where $A(\gaStep,x)$, $\varpi$ are given by \eqref{eq:borne_A_conv} and  \eqref{eq:kappa-convex} respectively and 
\begin{multline}
\label{eq:Cdelta_xQ_gamma_reflection_coro}
\Lambda(\delta_x \QKer_{\gaStep}^n) =  (1/2)( \Funfo(\thetaEa,\Gammasum_{n},\gaStep_1,\bEa,\VEa(x)) +\thetaOne^{-1}\bOne)\\ 
+ 2 \thetaOne^{-1} \bOne\rme^{4^{-1} \thetaOne \Fsmall(2^{-1},(8/\rhoU) \log( 4 \thetaOne^{-1} \bOne))} \eqsp,
% \label{eq:Cdelta_xQ_gamma_reflection}
% & \leq  4+(1/2)\defEns{\bOne \thetaOne^{-1} + \Funfo(\thetaEa,\Gammasum_{1,n},\bEa,\gaStep_1,\VEa(x))} \eqsp,
\end{multline}
 the functions $\Funfo$ and $\VEa$ are defined in
\eqref{eq:def_Funfo} and \eqref{eq:def_Wc}, the constants $ \thetaEa, \bEa,\thetaOne,
\bOne$ in \Cref{propo:drift-Euler} and
\Cref{propo:drift_exp_conv} respectively.
\end{corollary}
\begin{proof}
%  By \Cref{theo:quantative-bound-lang-convex}, we have for all $x \in \rset^d$,
% \begin{equation*}
% \tvnorm{\delta_x \PLang_{\Gamma_{n+1,p}} - \pi} \leq \Lambda(\delta_x \QKer_{\gaStep}^n) \rme^{-\thetaOne \Gamma_{n+1,p}/4  } + 2\varpi^{\Gamma_{n+1,p}} \Lambda(x) \rme^{-\thetaOne \GaStep_{n+1,p}/4  } + 2\varpi^{\GaStep_{n+1,p}}\eqsp.
% \end{equation*}
By \Cref{theo:quantative-bound-lang-convex}, \Cref{propo:drift-Euler} and \Cref{lem:bound_inho_moment}, we have for all $x \in \rset^d$,
\begin{equation*}
\tvnorm{\delta_x Q^n_\gamma \PLang_{\Gamma_{n+1,p}} - \pi} \leq  2 \Lambda(\delta_x \QKer_{\gaStep}^n) \rme^{-\thetaOne \Gamma_{n+1,p}/4  } + 4\varpi^{\Gamma_{n+1,p}}\eqsp.
\end{equation*}
Finally the proof follows the same line as the one of \Cref{propo:girsanov_comparison}.
\end{proof}
 Contrary to
\eqref{eq:Cdelta_xQ_gamma_bobkov},
\eqref{eq:Cdelta_xQ_gamma_reflection_coro} is uniformly bounded in $n$. By
\Cref{coro:quantative-bound-lang-convex} and \eqref{eq:borne_A_conv},
we can apply
\Cref{propo:limit-zeros-tv-1}-\ref{propo:limit-zeros-tv-1a}, which
implies the convergence to $0$ of $\tvnorm{\delta_x
  \QKer_{\gaStep}^p-\pi}$ as $p$ goes to infinity, if $\lim_{k \to
  \plusinfty} \gaStep_k = 0$ and $\lim_{k \to \plusinfty} \GaStep_k =
\plusinfty$.  Since $\log(\bOne)$ in \Cref{propo:drift_exp_conv} is of
order $d$, we get that the rate of convergence $\log(\varpi)$ is of
order $d^{-2}$ as $d$ goes to infinity (note indeed that the leading
term when $d$ is large is $ \thetaOne
\Fsmall\parenthese{2^{-1},(8/\rhoU) \log( 4 \thetaOne^{-1} \bOne)}$
which is of order $d^2$).  
In the case of constant step sizes
$\gaStep_k = \gaStep$ for all $k \geq 0$, we adapt 
\Cref{propo:precision}  to the bound given by \Cref{coro:quantative-bound-lang-convex}.

\begin{corollary}
  \label{coro:precision-convex}
Assume \Cref{assum:regularity} and \Cref{assum:convex}. Let $(\gaStep_k)_{k \geq 0}$ be a sequence of nonnegative
  step sizes. Then for all $\varepsilon >0$, we get $\tvnorm{\delta_x  \RKer_{\gaStep}^p-\pi} \leq \varepsilon$ if
$p$ and $\gaStep$ satisfy \eqref{eq:precision}
% \begin{equation}
%   \label{eq:precision-convex}
%   p > T \gaStep^{-1} \quad \text{ and } \quad \gaStep \leq \frac{-d+\sqrt{d^2 +(2/3) A(\gaStep,x) \varepsilon^2 (L^2T)^{-1} }}{2 A(\gaStep,x)/3} \wedge \LU^{-1} \eqsp,
% \end{equation}
with
% \begin{equation*}
%   T = \max\defEns{4 \thetaOne^{-1}\log\parenthese{4 \varepsilon^{-1}\sup_{n \geq 0} \Lambda(\delta_x \QKer_{\gaStep}^n)},\parenthese{ \log\{ 2 \}-\log(\varepsilon/4)} \Big/(- \log(\varpi))} \eqsp.
% \end{equation*}
\begin{align*}
  T &= \max\defEns{4 \thetaOne^{-1}\log\parenthese{8 \varepsilon^{-1}\tilde{\Lambda}(x)}, \log(16\varepsilon^{-1}) \Big/(- \log(\varpi))} \\
  \tilde{\Lambda}(x) &=  (1/2)( \FunfoD(\thetaEa,\gaStep_1,\bEa,\VEa(x)) +\thetaOne^{-1}\bOne)+ 2 \thetaOne^{-1} \bOne\rme^{4^{-1} \thetaOne \Fsmall(2^{-1},(8/\rhoU) \log( 4 \thetaOne^{-1} \bOne))} \eqsp,
% \label{eq:Cdelta_xQ_gamma_reflection}
% & \leq  4+(1/2)\defEns{\bOne \thetaOne^{-1} + \Funfo(\thetaEa,\Gammasum_{1,n},\bEa,\gaStep_1,\VEa(x))} \eqsp,
\end{align*}
where $A(\gaStep,x)$, $\varpi$ are given by \eqref{eq:borne_A_conv},  \eqref{eq:kappa-convex} respectively,
 the functions $\FunfoD$ and $\VEa$ are defined in
\eqref{eq:def_FunfoD} and \eqref{eq:def_Wc}, the constants $ \thetaEa, \bEa,\thetaOne,
\bOne$ in \Cref{propo:drift-Euler} and
\Cref{propo:drift_exp_conv} respectively.
\end{corollary}

\begin{proof}
  The proof follows the same line as the one of \Cref{propo:precision}
  using \Cref{coro:quantative-bound-lang-convex} and that $\sup_{n \geq 0} \Lambda(\delta \QKer_{\gaStep}^n) <
  \tilde{\Lambda}(x)$ for all $x \in \rset^d$.
\end{proof}

%, which implies that $\gamma$ has to be at
                           %least smaller than $d^{-1}$ as $d$ goes to
                           %infinity.

% We summarise in \Cref{tab:dep-conv-const-stepsize} the dependency of
% the step size $\gaStep >0$ and the minimum number of iterations $p
% \geq 0$, given by \Cref{propo:precision}, to achieve $\tvnorm{\delta_x
%   \QKer_{\gaStep}^p - \pi} \leq \varepsilon$ for $\varepsilon>0$.

\begin{table}
\centering
\fbox{
\begin{tabular}{c|c|c|c}
& $d$ & $\varepsilon$ & $\LUl$   \\
    \hline
$\gaStep$ & $\bigO(d^{-3})$ & $\bigO(\varepsilon^2/\log(\varepsilon^{-1}))$ & $\bigO(\LUl^{-2})$\\
    \hline
$p$ & $\bigO(d^5)$ & $\bigO(\varepsilon^{-2}\log^{2}(\varepsilon^{-1}))$ & $\bigO(\LUl^2)$\\
\end{tabular}
}
\vspace{0.3cm}
\\
\caption{\label{tab:dep-conv-const-stepsize} For constant step sizes, dependency of $\gaStep$ and $p$ in $d$, $\varepsilon$ and parameters of $\VUl$ to get $\tvnorm{\delta_x \RKer_{\gaStep}^p - \pi} \leq \varepsilon$ using \Cref{coro:precision-convex}}
\end{table}

% \begin{corollary}
%   \label{propo:fixed_budget-convex}
% Assume \Cref{assum:regularity} and \Cref{assum:convex}. Let $(\gaStep_k)_{k \geq 0}$ be a sequence of nonnegative
%   step sizes. For all $p >n$ and all $x \in \rset^d$, if $\gaStep = \log(p-n)\{ (p-n) (-\log(\varpi)) \}^{-1} \leq \bargaStep$, then
%   \begin{multline*}
%     \tvnorm{\delta_x \RKer_{\gaStep}^p - \pi} \\
% \leq \frac{2 (p-n)^{-1/2}+ \log(p-n)(d+ A(\gaStep,x)\log(p-n)/(p-n))^{1/2} }{(p-n)^{1/2}} \eqsp.
%   \end{multline*}
% \end{corollary}

In particular, with the notation of \Cref{coro:precision-convex},
since $\max(\log(\bOne),\log(\bEa))$ and $-(\log(\varpi))^{-1}$ are of
order $d$ and $d^2$ as $d$ goes to infinity respectively, $T$ is of
order $d^2$. Therefore, $\gaStep$ defined by
\eqref{eq:precision} is of order $d^{-3}$ which implies a
number of iteration $p$ of order $d^{5}$ to get $\tvnorm{\delta_x
  Q^p_{\gaStep}-\pi} \leq \varepsilon$ for $\varepsilon >0$; see also
\Cref{tab:dep-conv-const-stepsize}.

\Cref{coro:precision-convex} can be compared with the
results which establishes the dependency on the dimension for two
kinds of Metropolis-Hastings algorithms to sample from a log-concave
density: the random walk Metropolis algorithm (RWM) and the
hit-and-run algorithm.  It has been shown in \cite[Theorem
2.1]{lovasz:vempala:2007} that for $\varepsilon >0$, the hit-and-run
and the RWM reach a ball centered at $\pi$, of radius $\varepsilon $
for the total variation distance, in a number of iteration $p$ of
order $d^{4}$ as $d$ goes to infinity. It should be stressed that
\cite[Theorem 2.1]{lovasz:vempala:2007} does not assume any kind of
smoothness about the density $\pi$ contrary to
\Cref{theo:quantative-bound-lang-convex}.  However, this result
assumes that the target distribution is near-isotropic, \ie~there
exists $C \geq 0$ which does not depend on the dimension such that for
all $x \in \rset^d$,
\begin{equation*}
  C^{-1}\norm[2]{x}\leq \int_{\rset^d}\ps{x}{y}^2 \pi(\rmd y) \leq C \norm[2]{x} \eqsp.
\end{equation*}
Note that this condition implies that the variance of $\pi$ is upper
bounded by $C d$.  

To conclude our study on convex potential, we also mention
\cite{bubley:dyer:jerrum:1998} which studies the sampling of the
uniform distribution over a convex subset $\convexSet \subset \rset^d$
using coupling techniques.  Let $C > 0$. A convex set $\convexSet
\subset \rset^d$ is $C$-well rounded if $\ball{0}{1} \subset
\convexSet \subset \ball{0}{Cd}$. \cite{bubley:dyer:jerrum:1998} shows
that a number of iteration of order $d^{9}$ as $d$ goes to infinity is
sufficient to sample uniformly over any $C$-well rounded convex
set. Comparison with our result is difficult since we assume that
$\pi$ is positive on $\rset^d$, continuously differentiable, while
\cite{bubley:dyer:jerrum:1998} studies the case of uniform
distributions over a convex body. An adaptation of our result to non
continuously differentiable potentials will appear in a forthcoming
paper \cite{durmus:moulines:pereyra:2016}.

\subsection{Strongly log-concave densities}
\label{subsec:strongly-log-concave}
More precise bounds can be obtained in the case where $\VUl$ is assumed to be strongly convex outside some ball; this assumption has been considered by \cite{eberle:2015} for convergence in the Wasserstein distance; see also \cite{bolley:gentil:guillin:2012}.
\begin{assumption}[$\RSt$]
  \label{assum:strongConvexityOutsideBallDriftV}
  $U$ is convex and there exist $\RSt \geq 0$ and $\constStV >0$, such that for all
  $x,y \in \rset^d$ satisfying $\norm{x-y} \geq \RStV$,
  \[
  \ps{\nabla \VUl(x) -\nabla \VUl(y)} {x-y} \geq  \constStV
  \norm[2]{x-y} \eqsp.
  \]
\end{assumption}
We will see in the sequel that under this assumption the convergence rate in \eqref{eq:ergodicity-langevin-expo} does not depend on the dimension $d$ but only on the constants $\constStV$ and $\RStV$.

\begin{proposition}
  \label{propo:drift_R_gamma_Stconv}
 Assume \Cref{assum:regularity} and \Cref{assum:strongConvexityOutsideBallDriftV}($\RSt$). Let $\bargaStep \in  \ooint{0,2\constSt \LUl^{-2}}$. For all $\gaStep \in \ocint{0,\bargaStep}$, $V(x)= \norm[2]{x-\xstar}$ satisfies  \eqref{eq:foster-lyapunov-drift}
with  $\thetaESt= \rme^{-2\constSt+\bargaStep \LUl^2}$ and  $\bESt =  2(d + \constStV  \RSt^2)$.
\end{proposition}
\begin{proof}
The proof is postponed to \Cref{sec:proof_drift_strong_conv}.
\end{proof}
\begin{theorem}
\label{theo:convergence_TV_dec-stepsize-StV}
Assume \Cref{assum:regularity}  and \Cref{assum:strongConvexityOutsideBallDriftV}($\RSt$).  Let $\sequencen{\gaStep_k}[k \geq 1]$ be a nonincreasing sequence with $\gaStep_1 \leq \bargaStep$,  $\bargaStep \in \ooint{0,2\constSt \LUl^{-2}}$.
Then, for all $n\geq 0$, $p \geq 1$, $n<p$, and $x \in \rset^d$, \eqref{eq:eq_base} holds with
\begin{align}
%  \label{eq:definition-rate-conv}
\nonumber
\log(\tauConvSt ) &= -(\constSt/2) \log(2) \\
\nonumber
&\quad \times \parentheseDeux{\log\defEns{\parenthese{1+\rme^{\constSt \Fsmall(2^{-1},\max(1,\RSt))/4}}(1+\max(1,\RSt))} +\log(2) }^{-1} \\
\nonumber
  C(\delta_x \QKer_{\gaStep}^n) &\leq 6+ 2\parenthese{d/\constStV + \RStV^2}^{1/2} + 2 \Funfo^{1/2}(\thetaESt,\Gamma_{1,n},\bESt,\gaStep_1,\norm[2]{x-\xstar}) \\
%\label{eq:a_ga_x_ht}
\nonumber
A(\gaStep,x) & \leq \LUl^{2}\, \FunfoD(\thetaESt,\bESt,\gaStep_1, \norm[2]{x-\xstar}) \eqsp,
\end{align}
where $\Funfo,\FunfoD,\Fsmall$ are defined by \eqref{eq:def_Funfo}, \eqref{eq:def_FunfoD}, \eqref{eq:Fsmall} respectively, and $\thetaESt,\bESt$ are given in \Cref{propo:drift_R_gamma_Stconv}.
\end{theorem}
\begin{proof}
The proof is postponed to \Cref{sec:proof-crefth-steps-1}.
\end{proof}

Note that the conditions of
\Cref{propo:limit-zeros-tv-1}-\ref{propo:limit-zeros-tv-1a} are
fulfilled. For constant step sizes $\gaStep_k = \gaStep$ for all $k
\geq 1$, \Cref{propo:precision} and \Cref{propo:fixed_budget} can be
applied.  We give in \Cref{tab:dep-conv-const-stepsize-St} the
dependency of the step size $\gaStep >0$ and the minimum number of
iterations $p \geq 0$, provided in \Cref{propo:precision}, on the
dimension $d$ and the other constants related to $U$, to get
$\tvnorm{\delta_x \QKer_{\gaStep}^p - \pi} \leq \varepsilon$, for a
target precision $\varepsilon >0$. We can see that the dependency on
the dimension is milder than for the convex case. The number of
iteration requires to reach a target precision $\varepsilon$ is just
of order $\bigO(d \log(d))$.

Consider the case where $\pi$ is the $d$-dimensional standard Gaussian
distribution. Then for all $p \in \nset$, $\gaStep \in \ooint{0,1}$ and
$x \in \rset^d$, $\delta_x \RKer_{\gaStep}^p$ is the $d$-dimensional
Gaussian distribution with mean $(1-\gaStep)^{p} x$ and covariance
matrix $\sigma_{\gaStep}\IdM$, with $\sigmagaStep=(1-(1-\gaStep)^{2(p+1)})(1-\gaStep/2)^{-1}$. Therefore using the Pinsker
inequality, we get:
\begin{align*}
  \tvnorm{\delta_x \RKer_{\gaStep}^p-\pi}^2 &\leq 2 \KL \parenthese{\left. \delta_x \RKer_{\gaStep}^p \right| \pi }  \\
&\leq d \parentheseDeux{\log\parenthese{\sigmagaStep} -1 +  \sigmagaStep^{-1}\defEns{1+(1-\gaStep)^{2p}\norm[2]{x} d^{-1}} } \eqsp.
\end{align*}
Using the inequalities for all $t \in \ooint{0,1}$, $ (1-t)^{-1} \leq 1+t(1-t)^{-2}$ and for all $s \in \ooint{0,1/2}$, $-\log(1-s) \leq s+2s^2$, we have:
\begin{multline*}
  \tvnorm{\delta_x \RKer_{\gaStep}^p-\pi}^2 \leq d \defEns{\gaStep^2/2 + (1-\gaStep)^{2(p+1)}(1-\gaStep/2)(1-(1-\gaStep)^{2(p+1)})^{-2}} \\
+ \sigmagaStep^{-1}(1-\gaStep)^{2p}\norm[2]{x} \eqsp.
%\frac{\gaStep^{2k}}
\end{multline*}
This inequality implies that in order to have $\tvnorm{\delta_x \RKer_{\gaStep} -\pi} \leq \varepsilon$ for  $\varepsilon >0$, the step size $\gaStep$ has to be of order $d^{-1/2}$ and $p$ of order $d^{1/2} \log(d)$. Therefore, the dependency on the dimension reported in \Cref{tab:dep-conv-const-stepsize-St} does not match this particular example. However it does not imply that this dependency can be improved.

% $x \in \ooint{0,1}$:
% \begin{equation*}
%   \label{eq:1}
%   (1-x)^{-1} = 1+ \int_0^x (1-t)^{-2} \rmd t \leq 1+x(1-x)^{-2} \eqsp.
% \end{equation*}
% $x \in \ooint{0,1/2}$:
% \begin{equation*}
%   \label{eq:1}
%   -\log(1-x) = -\log(1)  + [(t-x)/(1-t)]_0^x - \int_0^x(t-x)/(1-t)^2 \rmd t \leq x +2x^2 \eqsp.
% \end{equation*}

\begin{table}
%\begin{normalsize}
\centering
\fbox{
\begin{tabular}{c|c|c|c|c|c}
& $d$ & $\varepsilon$ & $\LF$ & $\constSt$& $\RSt$  \\
    \hline
$\gaStep$ & $\bigO(d^{-1})$ & $\bigO(\varepsilon^2/\log(\varepsilon^{-1}))$ & $\bigO(\LF^{-2})$& $\bigO(\constSt)$ & $\bigO(\RSt^{-4})$  \\
    \hline
$p$ & $\bigO(d \log(d))$ & $\bigO(\varepsilon^{-2}\log^{2}(\varepsilon^{-1}))$ & $\bigO(\LF^2)$&$\bigO(\constSt^{-2})$ & $\bigO(\RSt^{8})$\\
\end{tabular}
}
\vspace{0.3cm}
\\
  \caption{ \label{tab:dep-conv-const-stepsize-St} For constant step sizes,
  dependency of $\gaStep$ and $p$ in $d$,
  $\varepsilon$ and parameters of $\VUl$
  to get $\tvnorm{\delta_x \QKer_{\gaStep}^p - \pi} \leq \varepsilon$ using \Cref{theo:convergence_TV_dec-stepsize-StV}}
%\end{normalsize}
\end{table}

\subsection{Bounded perturbation of strongly log-concave densities}
\label{sec:bound-pert-strongly}
We now consider the case where $U$ is a bounded perturbation of a strongly convex potential.
\begin{assumption}
  \label{assum:pertSt}
  The potential $U$ may be expressed as $U = \Uun +\Ude$, where
  \begin{enumerate}[label = (\alph*)]
  \item
\label{assum:pertSt-a}
 $\Uun : \rset^d \to \rset$ satisfies  \Cref{assum:strongConvexityOutsideBallDriftV}($0$) (\ie\ is strongly convex) and there exists $\Lun \geq 0$ such that for all $x,y \in \rset^d$,
 \begin{equation*}
   \norm{\nabla \Uun(x) - \nabla \Uun(y)} \leq \Lun \norm{x-y} \eqsp.
 \end{equation*}
\item
\label{assum:pertSt-b}
$\Ude : \rset^d \to \rset$ is continuously differentiable and $ \normsup{\Ude} + \normsup{\nabla \Ude} < \plusinfty$  .
\end{enumerate}
\end{assumption}
% \begin{assumption}
% \label{assum:potentialU}
% $U$ is strongly convex, \ie~there exists $\mStD >0$ such that for all $x,y \in \rset^d$,
% \begin{equation}
%   \label{eq:convex_forte_contra2}
% \ps{\nabla U(y) - \nabla U(x)}{y-x} \geq \mStD \norm[2]{y-x} \eqsp.
% \end{equation}
% \end{assumption}
%Note that \Cref{assum:regularity} and \Cref{assum:strongConvexityOutsideBallDriftV}$(0)$ imply that $L \geq \mStD$.
The probability measure $\pi$ is said to  satisfy a log-Sobolev inequality with constant $\Clogsob>0$ if ~for all locally Lipschitz function $h : \rset^d \to \rset$, we have
\begin{equation*}
  \Ent{\pi}{h^2} \leq 2 \Clogsob \int \norm[2]{\nabla h} \rmd \pi \eqsp.
\end{equation*}
Then  \cite[Theorem 2.7]{cattiaux:guillin:2009} shows that for all $t \geq 0$ and any probability measure $\mu_0 \ll \pi$ satisfying $\densityPiLigne{\mu_0}\log(\densityPiLigne{\mu_0}) \in \Lone(\pi)$,  we have
\begin{equation}
\label{eq:convergence_log_sob}
\tvnorm{\mu_0 \PL_t - \pi} \leq \rme^{-t/\Clogsob }\defEns{2\, \Ent{\pi}{\densityPi{\mu_0}}}^{1/2} \eqsp.
\end{equation}
Under \Cref{assum:pertSt}, \cite[Corollary 5.7.2]{bakry:gentil:ledoux:2014} and the Holley-Stroock perturbation principle \cite[p. 1184]{holley:stroock:1987}, $\pi$ satisfies a log-Sobolev inequality with a constant which only depends on the  strong convexity constant $\mStD$ of $\Uun$ and $\osc[\rset^d] { \Ude }$.
Define
\begin{equation}
\label{eq:definition-varpi}
\kappaS = \frac{2 \mStD \Lun }{\mStD+\Lun} \eqsp.
\end{equation}
Denote by $\xstarun$ the minimizer of $\Uun$.
\begin{proposition}
\label{theo:kind_drift}
Assume \Cref{assum:pertSt}. Let $(\gamma_k)_{k \geq 1}$ be a nonincreasing sequence with $\gamma_1 \leq 2/(\mStD+\Lun)$.
Then for all $p \geq 1$ and $x \in \rset^d$,
\begin{multline*}
 \int_{\rset^d} \norm[2]{y- \xstarun}   \QKer^{p}_\gaStep(x,\rmd y) \leq
 \prod_{k=1}^p (1-\kappaS \gaStep_k/2)   \norm[2]{x - \xstarun}\\
+  2 \kappaS^{-1}(2d+(\gaStep_{1}+2\kappaS^{-1}) \normsup[2]{\nabla \Ude}   )\eqsp.
\end{multline*}
\end{proposition}

% \begin{proposition}
% \label{theo:kind_drift}
% Assume \Cref{assum:regularity} and \Cref{assum:strongConvexityOutsideBallDriftV}$(0)$. Let $(\gamma_k)_{k \geq 1}$ be a nonincreasing sequence with $\gamma_1 \leq 2/(\mStD+L)$.
% Then for all $p \geq 1$,
% \begin{equation*}
%  \int_{\rset^d} \norm[2]{x - x^\star}  \mu_0 \QKer^{p}_\gaStep(\rmd x) \leq
%  \prod_{k=n}^p (1-\varpi \gaStep_k)  \int_{\rset^d} \norm[2]{x - x^\star}  \mu_0(\rmd x) +2d  \varpi^{-1}\eqsp.
% \end{equation*}
% \end{proposition}

\begin{proof}
The proof is postponed to \Cref{proof:theo:kind_drift}.
\end{proof}
\begin{theorem}
\label{theo:convergence_TV_quantitatif_logSob}
Assume \Cref{assum:regularity} and  \Cref{assum:pertSt}. Let $\sequencen{\gaStep_k}[k \in \nset^*]$ be a nonincreasing sequence with $\gaStep_1 \leq 2/(\constSt+\Lun)$.
Then, for all $n, p \geq 1$, $n<p$, and $x \in \rset^d$, \eqref{eq:eq_base} holds with $-\log(\kappa) = \constSt \exp\{- \osc[\rset^d] { \Ude } \}$ and 
\begin{align}
%\label{eq:borneCdeltalogSob}
\nonumber
 & C^2(\delta_x \QKer_{\gaStep}^n) \leq
\Lun\rme^{-\kappaS \GaStep_n/2} \norm{x -\xstarun}^2+ \Lun\gaStep_n(\gaStep_n+2\kappaS^{-1}) \normsup[2]{\nabla \Ude}
 +2\osc[\rset^d]{\Ude} \\
\nonumber
&+2 \Lun\kappaS^{-1}(1-\kappaS \gaStep_n)(2d+(\gaStep_{1}+2\kappaS^{-1}) \normsup[2]{\nabla \Ude}   )
-d(1+\log(2 \gaStep_n m) -2\Lun\gaStep_n )\\
\nonumber
&A(\gaStep,x) \leq 2\Lun^{2}\, \defEns{\norm[2]{\xunstar-\xstar}+2 \kappaS^{-1}(2d+(\gaStep_1+2\kappaS^{-1}) \normsup[2]{\nabla \Ude}   )} +  2\normsup[2]{\nabla \Ude}\eqsp,
\end{align}
where    $\kappaS$ is defined  in \eqref{eq:definition-varpi}.
% \begin{align}
% \nonumber
%   C(\delta_x \QKer_{\gaStep}^n) &\leq
% (L/2)\defEns{\prod_{k=1}^n (1-\varpi \gamma_k)}   \norm[2]{x - x^\star}   +dL \varpi^{-1}(1-\varpi \gamma_n)\\
% \nonumber
%  & \phantom{(L/2)\defEns{\prod_{k=1}^n (1-\varpi \gamma_k)}   \norm[2]{x - x^\star}   } -(d/2)(1+\log(2 \gamma_n/\mStD) -2L\gamma_n )\\
% \end{align}
\end{theorem}
% Let $(\gamma_k)_{k \geq 1}$ be a non increasing sequence with $\gamma_1 \leq 2/(\mStD+L)$.  Let $x^\star$ be the unique minimizer of $U$. Then for all $n\geq 0$, $p \geq 1$, $n < p$ and $x \in \rset^d$,
% \begin{equation*}
% \tvnorm{\delta_x Q^p_\gamma -\pi} \leq  \boundDeuxTVlogSob\rme^{-\Gamma_{n+1,p}/\mStD} +  \boundUnTVlogSob
% \end{equation*}
% where $\kappa$ is defined in \eqref{eq:definition-kappa} and
% \begin{align*}
%   \boundDeuxTVlogSob& =(L/2)\defEns{\prod_{k=1}^n (1-\kappa \gamma_k)}   \norm[2]{x - x^\star}   +dL \kappa^{-1}(1-\kappa \gamma_n)\\
% & \phantom{(L/2)\defEns{\prod_{k=1}^n (1-\kappa \gamma_k)}   \norm[2]{x - x^\star}   } -(d/2)(1+\log(2 \gamma_n/\mStD) -2L\gamma_n ) \\
%  \boundUnTVlogSob&= 2^{-1/2} L \parenthese{ \sum_{k=n}^{p-1}(\gamma_{k+1}^3 L^2/3)\defEns{\prod_{i=1}^k (1-\kappa \gamma_i) \norm[2]{x - x^\star} +2d  \kappa^{-1}} + d \gamma_{k+1}^2}^{1/2} \eqsp.
% \end{align*}
\begin{proof}
The proof is postponed to \Cref{sec:theo:convergence_TV_quantitatif_logSob}.
\end{proof}

Note that $\sup_{n \geq 1}\{ C(\delta_x
\QKer_{\gaStep}^n)/(-\log(\gaStep_n))\} < \plusinfty$, therefore
\Cref{propo:limit-zeros-tv-1}-\ref{propo:limit-zeros-tv-1a} can be
applied and $\lim_{p \to \plusinfty}\tvnorm{\delta_{\gaStep}
  \QKer_{\gaStep}^p-\pi} = 0$ if $\lim_{k \to \plusinfty} \gaStep_k =
0$ and $\lim_{k \to \plusinfty} \GaStep_k = \plusinfty$.
% In the case
% of the constant step size $\gaStep_k = \gaStep$ for all $k \geq 0$, we
% can apply \Cref{propo:precision} and \Cref{propo:fixed_budget}.
% Also,
% the dependency of the step size $\gaStep >0$ and the number of
% iterations $p \geq 1$ on the dimension and the strong convexity
% constant of $\Uun$, provided by \Cref{propo:precision} to get
% $\tvnorm{\delta_x \QKer_{\gaStep}^n - \pi } \leq \varepsilon$ for a
% target precision $\varepsilon >0$ is the same than under \Cref{assum:strongConvexityOutsideBallDriftV}, reported in \Cref{tab:dep-conv-const-stepsize-St}.

%%% Local Variables:
%%% mode: latex
%%% TeX-master: "main"
%%% End:

%\input{convergence-euler}
\section{Proofs}
\label{sec:proof}

% \subsection{Proof of \Cref{propo:precision}}
% \label{proof:precision}

\subsection{Proof of \Cref{lem:bound_inho_moment}}
\label{sec:proof:lem:inh}
By a straightforward induction, we get for all $n \geq 0$ and $x \in \rset^d$,
\begin{equation}
  \label{eq:22}
  \QKer^n_{\gaStep} \Vly(x)\leq \lambda^{ \Gammasum_{n}}\Vly(x) +c \sum_{i=1}^n \gaStep_i \lambda^{ \Gammasum_{i+1,n}} \eqsp.
\end{equation}
Note that for all $n \geq 1$, we have since $(\gaStep_k)_{k \geq 1}$ is nonincreasing and for all $t \geq 0$, $\lambda^t = 1+\int_0^t   \lambda^s \log(\lambda) \rmd s $,
\begin{align*}
%  \label{eq:23}
&\sum_{i=1}^n \gaStep_i \lambda^{ \Gammasum_{i+1,n}} \leq \sum_{i=1}^n \gaStep_i \prod_{j=i+1}^n (1+ \lambda^{\gaStep_1}  \log(\lambda) \gaStep_j) \\
&\leq (- \lambda^{\gaStep_1}\log(\lambda))^{-1}\sum_{i=1}^n \gaStep_i \defEns{\prod_{j=i+1}^n (1+\lambda^{\gaStep_1} \log(\lambda)  \gaStep_j) -\prod_{j=i}^n (1+  \lambda^{\gaStep_1} \log(\lambda) \gaStep_j)}
\\
& \leq  (-\lambda^{\gaStep_1}\log(\lambda) )^{-1} \eqsp.
\end{align*}
The proof is then completed using this inequality in \eqref{eq:22}.

\subsection{Proof of \Cref{lem:majo:U}}
\label{sec:proof:majo:U}
By \Cref{assum:regularity}, \Cref{assum:drift-condition-reinforced}, the Cauchy-Schwarz inequality and $\nabla U(\xstar) = 0$,  for all $x \in \rset^d$, $\norm{x} \geq \Mrho$, we have
  \begin{align*}
&    U(x) - U(\globmin) = \int_{0}^1\ps{\nabla U(\globmin + t(x-\globmin))}{x - \globmin} \rmd t \\
&\geq \int_{0}^{\frac{\Mrho}{\norm{x - \globmin}}} \ps{\nabla U(\globmin + t (x - \globmin))}{x-\globmin} \rmd t  \\
& \phantom{aaaaa}+ \int_{\frac{\Mrho}{\norm{x-\globmin}}}^1 \ps{\nabla U(\globmin + t(x- \globmin))}{t(x - \globmin)} \rmd t \\
& \geq  - \Mrho^2 \LU/2  + \rho \norm[\alpha]{x-\globmin}(\alpha+1)^{-1}\defEns{1-(\Mrho/\norm{x-\globmin})^{\alpha+1}} \eqsp.
  \end{align*}
On the other hand using again \Cref{assum:regularity}, the Cauchy-Schwarz inequality and $\nabla U(\xstar) = 0$, for all $x \in \ball{\globmin}{\Mrho}$,
 \begin{equation*}
    U(x) - U(\globmin) = \int_{0}^1\ps{\nabla U(\globmin + t(x-\globmin))}{x-\globmin} \rmd t \geq - \Mrho^2 \LU/2  \eqsp,
\end{equation*}
which concludes the proof.

\subsection{Proof of \Cref{prop:poincare-constant-lyapunov}}
\label{proof:prop:poincare-constant-lyapunov}
 For all $x \in \rset^d$, we have
\[
\generatorL \VSAlpha(x) = \varsigma(1-\varsigma)\left\{-\norm[2]{\nabla U(x)} + (1-\varsigma)^{-1}\Delta U(x)\right\} \VSAlpha (x) \eqsp.
\]
If $\alpha > 1$, by the Cauchy-Schwarz inequality, under \Cref{assum:regularity}-\Cref{assum:drift-condition-reinforced} for all $x \in \rset^d$, $\Delta U (x) \leq dL$  and $\norm{\nabla U(x)} \geq \rho \norm{x-\xstar}^{\alpha-1} $ for $\norm{x-\xstar} \geq M_{\rho}$. Then, for all $x \not \in \Setdrift_{\varsigma}$,
\[
\generatorL \VSAlpha(x) \leq \varsigma(1-\varsigma)\left\{-\rho \norm{x-\xstar}^{2(\alpha-1)} + (1-\varsigma)^{-1}dL \right\} \VSAlpha(x) \leq -\varsigma dL \VSAlpha(x) \eqsp,
\]
and $\sup_{\{x  \in \Setdrift_{\varsigma}\}} \generatorL \VSAlpha(x) \leq\varsigma dL  \sup\nolimits_{\{y  \in \Setdrift_{\varsigma} \}}\{ \VSAlpha(y)\}$.

\subsection{Proof of \Cref{prop:drift-R-gaStep}}
\label{proof:prop:drift-R-gaStep}
By \Cref{assum:drift-condition-reinforced}, for all $x \not \in \ball{\globmin}{M_{\rho}}$,
\begin{equation}
\label{eq:gradient_explosed}
\norm{\nabla U(x)} \geq \rho \norm{x-\globmin}^{\alpha-1} \eqsp.
\end{equation}
Since under \Cref{assum:regularity}, for all $x,y \in \rset^d$, $U(y) \leq U(x) + \ps{\nabla{U}(x)}{y-x} + (L/2) \|y-x\|^2$, we have for all $\gaStep \in \ooint{0,\bargaStep}$ and $x \in \rset^d$,
\begin{align}
\nonumber
&R_\gaStep \Vly(x)/\Vly(x)\\
\nonumber
& = (4 \pinumber \gaStep)^{-d/2}\int_{\rset^d} \exp \parenthese{\defEns{U(y)-U(x)}/2 - (4\gaStep)^{-1}\norm[2]{y - x + \gaStep \nabla U(x)}} \rmd y \\
\nonumber
&
\leq (4 \pinumber \gaStep)^{-d/2}\int_{\rset^d} \exp \parenthese{-4^{-1}\gaStep\norm[2]{\nabla U(x)} - (4\gaStep)^{-1}(1-\gaStep L)\norm[2]{y - x}} \rmd y \\
%\label{eq:drift_Euler_weak_assumption}
\nonumber
&
\leq
(1-\gaStep L)^{-d/2} \exp(-4^{-1}\gaStep\norm[2]{\nabla U(x)})  \eqsp,
\end{align}
where we used in the last line that $\gaStep < L^{-1}$. Since $\log(1- L \gaStep) = -L \int_0^\gaStep (1-Lt)^{-1} \rmd t$, for all $\gaStep \in \ocint{0,\bargaStep}$, $\log(1-L\gaStep) \geq - L \gaStep (1-L \bargaStep)^{-1}$. Using this inequality, we get
\begin{equation}
\label{eq:drift_Euler_weak_assumption2}
R_\gaStep \Vly(x)/\Vly(x) \leq \lDriftgaStep^{-\gaStep}\exp \parenthese{-4^{-1} \gaStep \norm[2]{\nabla U(x)}} \eqsp.
\end{equation}
By \eqref{eq:gradient_explosed}, for all $x \in \rset^d$, $\norm{x-\xstar} \geq \KDriftRgaStep$, we have
\begin{equation}
\label{eq:drift_Euler_weak_assumption_outside_compact}
\RKer_\gaStep \Vly(x) \leq  \lDriftgaStep^{\gaStep}\Vly(x) \eqsp.
\end{equation}
Also by \eqref{eq:drift_Euler_weak_assumption2} and since for all $t \geq 0$, $\rme^{ t} - 1 \leq t\rme^{ t}$, we get for all $x \in \rset^d$
\[
R_\gaStep \Vly(x) -\lDriftgaStep^{\gaStep}  \Vly(x) \leq \lDriftgaStep^{\gaStep} (\lDriftgaStep^{-2\gaStep}-1) \Vly(x)\leq -2 \gaStep \log(\lDriftgaStep)  \lDriftgaStep^{-\bargaStep} \Vly(x) \eqsp.
\]
The proof is completed combining the last inequality and \eqref{eq:drift_Euler_weak_assumption_outside_compact}.

\subsection{Proof of \Cref{theo:convergence_TV_non_quantitatif}}
\label{proof:theo:convergence_TV_non_quantitatif}
%\phi convex because
%D[Exp[a*(t+b)^(c/2)],{t,2}] = (1/4) a c e^{a (b+t)^(c/2)} (b+t)^(-2+c/2) ((-2+c) +a c (b+t)^(-c/2))
We first bound $A(\gaStep,x)$ for all $x \in \rset^d$. Let $x \in \rset^d$.
By \Cref{assum:regularity}, we have $ \PE_x[ \| \nabla U(X_k) \|^2] \leq L^{2} \PE_x[ \| X_k - \globmin \|^2]$.
% \begin{equation}
% %\label{eq:theo:convergence_TV_non_quantitatif_1}
%   \PE_x[ \| \nabla U(X_k) \|^2] \leq L^{2} \PE_x[ \| X_k - \globmin \|^2] \eqsp.
% \end{equation}
Consider now the function $\phi_{\alpha} : \rset_+ \to \rset_+$ defined for all $t \geq 0$ by $\phi_{\alpha}(t) = \exp( A_{\alpha} (t + B_{\alpha}  ) ^{\alpha/2} )$ where $A_{\alpha} = \rho/(2 (\alpha+1))$ and $B_{\alpha}= \defEns{(2-\alpha)/(\alpha A_{\alpha})}^{2/\alpha}$.  Since $\phi_{\alpha}$ is convex and invertible on $\rset_+$, we get using the Jensen inequality and \Cref{lem:majo:U} for all $k \geq 0$:
\begin{equation*}
\PE_x[ \| X_k - \globmin \|^2]  \leq \phi_{\alpha}^{-1}\parenthese{\PE_x[ \phi_{\alpha}\parenthese{ \| X_k - \globmin \|^2}]}   \leq \phi_{\alpha}^{-1} \parenthese{\rme^{\Calpha /2 + B_{\alpha}^{\alpha/2}}\PE_x[ \Vly( X_k)]} \eqsp,
\end{equation*}
where $\Vly(x)= \exp( U(x)/2)$.
Using that for all $t \geq 0$, $\phi_{\alpha}^{-1}(t) \leq (A_{\alpha}^{-1}\log(t))^{2/\alpha}$ and \Cref{lem:bound_inho_moment}, we get
\begin{equation*}
\sup_{k \geq 0}\PE_x[ \| X_k - \globmin \|^2] \leq \parenthese{A_{\alpha}^{-1}\parentheseDeux{\Calpha /2 + B_{\alpha}^{\alpha/2}+\log\defEns{ \FunfoD(\lDriftgaStep,\bDriftRgaStep(\gaStep_1),\Vly(x)) }}}^{2/\alpha} \eqsp.
\end{equation*}
Eq.~\eqref{eq:definition-B-n-p-0} follows from \Cref{coro:poincare-constant-lyapunov}, \Cref{prop:drift-R-gaStep} and \Cref{lem:bound_inho_moment}.

\subsection{Proof of \Cref{theo:convergence_TV_gaStep_const}}
\label{proof:theo:convergence_TV_gaStep_const}

\begin{lemma}
\label{lem:contol_gamma_const}
Let $\mu$ and $\nu$ be two probability measures on $(\rset^d,\mathcal{B}(\rset^d))$ and $\Vly: \rset^d \to \coint{1,\infty}$ be a measurable function. Then
\[
\Vnorm[\Vly]{\mu - \nu} \leq \sqrt{2} \left\{ \nu(V^2) + \mu(V^2) \right\}^{1/2} \KL^{1/2}(\mu \vert \nu) \eqsp.
\]
\end{lemma}

\begin{proof}
Without losing any generality, we assume that  $\mu \ll \nu$.
 For all $t \in \ccint{0,1}$, $t \log(t) -t +1 = \int_t^1(u-t)u^{-1} \rmd u \geq 2^{-1}(1-t)^2$, and on $\coint{1,\plusinfty}$, $t \mapsto 2(1+t)(t\log(t) -t+1)-(1-t)^2$ is nonincreasing.
 %untiliser for $t \geq 1$, $t \log(t) (t-1)^{-1} \geq 1$.
 %fonction croissante plus grande que f(1) = 1
Therefore, for all $t \geq 0$,
\begin{equation}
\label{eq:bound_Pinsker_Vnorm}
\abs{1-t} \leq (2(1+t)(t\log(t) -t+1))^{1/2} \eqsp.
\end{equation}
Then, we have:
\begin{align}
\nonumber
&\Vnorm[\Vly]{\mu - \nu} = \sup_{f \in \functionspace[]{\rset^d}, \Vnorm[\Vly]{f} \leq 1} \abs{\int_{\rset^d} f(x) \rmd \mu (x) - \int_{\rset^d} f(x) \rmd \nu (x) } \\
%\label{eq;Pinsker_Vnorm_0}
\nonumber
& =  \sup_{f \in \functionspace[]{\rset^d}, \Vnorm[\Vly]{f} \leq 1} \abs{ \int_{\rset^d} f(x) \defEns{\density{\nu}{\mu} - 1} \rmd \nu (x) } \leq
\int_{\rset^d} \Vly(x) \abs{\density{\nu}{\mu} - 1} \rmd \nu (x) \eqsp.
\end{align}
Using \eqref{eq:bound_Pinsker_Vnorm} and the Cauchy-Schwarz inequality in the previous inequality concludes the proof.
\end{proof}

\begin{proof}[Proof of \Cref{theo:convergence_TV_gaStep_const}]
  First note that by the triangle inequality and \Cref{coro:poincare-constant-lyapunov}, for all $p \geq 1$
 \begin{equation}
\label{eq:gaStep_const_proof1}
 \Vnorm[\Vlysqrt]{\pi - \delta_x \QKer_\gaStep^p} \leq \ConstVnormUnquart \kappa^{ p \gaStep}\Vlysqrt(x)+ \Vnorm[\Vlysqrt]{\delta_x \PL_{\GaStep_p} - \delta_x \QKer_\gaStep^p} \eqsp.
 \end{equation}
 We now bound the second term of the right hand side. Let $\kgaStep = \ceil{\gaStep^{-1}}$ and $\qgaStep$ and $\rgaStep$ be respectively the quotient and the remainder of the Euclidean division of $p$ by $\kgaStep$.  The triangle inequality implies $\Vnorm[\Vlysqrt]{\delta_x \PL_{\GaStep_p} - \delta_x \QKer_\gaStep^p} \leq \AtermProofBloc + \BtermProofBloc$
%\begin{equation}
%\label{eq:deconposition_V_norm}
%\Vnorm[\Vlysqrt]{\delta_x \PL_{\GaStep_n} - \delta_x \QKer_\gaStep^n} \leq \AtermProofBloc + \BtermProofBloc
%\end{equation}
with
\begin{align*}
A &= \VnormEq[\Vlysqrt]{\delta_x \QKer_\gaStep^{(\qgaStep-1) \kgaStep }\PL_{\GaStep_{(\qgaStep-1)\kgaStep,p}} - \delta_x \QKer^{(\qgaStep-1) \kgaStep }_\gaStep \QKer^{(\qgaStep-1) \kgaStep+1,p }_\gaStep} \\
B&=\sum_{i=1}^{\qgaStep} \VnormEq[\Vlysqrt]{\delta_x \QKer^{(i-1)\kgaStep}_\gaStep \PL_{\GaStep_{(i-1)\kgaStep+1,p}} - \delta_x \QKer^{i \kgaStep}_\gaStep \PL_{\GaStep_{i\kgaStep +1,p}} } \eqsp.
\end{align*}
It follows from \Cref{coro:poincare-constant-lyapunov} and $\kgaStep \geq \gaStep^{-1}$ that
\begin{equation}
\label{eq:gaStep_cont_proof2}
B  \leq \sum_{i=1}^{\qgaStep} \ConstVnormUnquart \kappa^{\qgaStep-i} \VnormEq[\Vlysqrt]{\delta_x \QKer^{(i-1)\kgaStep}_\gaStep \PL_{\GaStep_{(i-1)\kgaStep+1,i \kgaStep}} - \delta_x \QKer^{i \kgaStep}_\gaStep } \eqsp.
\end{equation}
We now bound each term of the sum in the right hand side.
% Let $\nu_0$ be a probability measure on $(\rset^d,\mathcal{B}(\rset^d))$ and denote by $\mu$ and $\bar{\mu}$ the law of $(Y_t)_{0 \leq t \leq \GaStep_n}$ and $(\Ybar_t)_{0 \leq t \leq \GaStep_n}$ on $\Csetfunction(\ccint{0,\GaStep_n},\rset^d)$, where $(Y_t, \Ybar_t)_{0 \leq t \leq \GaStep_n}$ is the  synchronous coupling  defined by \eqref{eq:definition_couplage}, started at $(Y_0,Y_0)$ with $Y_0$ distributed according to $\nu_0$.
For all initial distribution $\nu_0$ on
$(\rset^d,\mathcal{B}(\rset^d))$ and $i,j \geq 1$, $i <j$, it follows from \Cref{lem:contol_gamma_const},
 \cite[Theorem 4.1, Chapter 2]{kullback:1997} and \eqref{eq:pinsker_girsanov_2}:
\begin{align}
\nonumber
  &\Vnorm[V^{1/2}]{\nu_0  Q^{i,j}_{\gaStep} - \nu_0 \PL_{\GaStep_{i,j}}}^2 \leq 2 \parenthese{ \nu_0 Q^{i,j}_{\gaStep}(V) + \nu_0 \PL_{\GaStep_{i,j}}(V)}  \KL(\nu_0  Q^{i,j}_{\gaStep} \vert \nu_0 \PL_{\GaStep_{i,j}})
\\
%\label{eq:gaStep_cont_proof3}
\nonumber
&\leq
2 \LU^2 \parenthese{ \nu_0 Q^{i,j}_{\gaStep}(V) + \nu_0 \PL_{\GaStep_{i,j}}(V)}\\
\nonumber
& \phantom{\parenthese{  + \nu_0 \PL_{\GaStep_{i,j}}(V)}^{1/2}}\times (j-i) \parenthese{ \gaStep^2 d+(\gaStep^3/3)\sup_{k \in \{i,\cdots,j\}} \nu_0 Q^{i,k-1}_{\gaStep}(\norm[2]{\nabla U})   } \eqsp.
\end{align}
% \begin{align}
% \nonumber
%   &\Vnorm[V^{1/2}]{\nu_0  Q^{i,j}_{\gaStep} - \nu_0 \PL_{\GaStep_{i,j}}} \\
% \nonumber
% &\leq 2^{-1/2} \parenthese{ \nu_0 Q^{i,j}_{\gaStep}(V) + \nu_0 \PL_{\GaStep_{i,j}}(V)}^{1/2}  \{\KL(\nu_0  Q^{i,j}_{\gaStep} \vert \nu_0 \PL_{\GaStep_{i,j}}) \}^{1/2}
% \\
% %\label{eq:gaStep_cont_proof3}
% \nonumber
% &\leq
% 2^{-1/2} \LU \parenthese{ \nu_0 Q^{i,j}_{\gaStep}(V) + \nu_0 \PL_{\GaStep_{i,j}}(V)}^{1/2}\\
% \nonumber
% & \phantom{\parenthese{  + \nu_0 \PL_{\GaStep_{i,j}}(V)}^{1/2}}\times (j-i)^{1/2} \parenthese{ \gaStep^2 d+(\gaStep^3/3)\sup_{k \in \{i,\cdots,j\}} \nu_0 Q^{i,k-1}_{\gaStep}(\norm[2]{\nabla U})   }^{1/2} \eqsp.
% \end{align}
\Cref{prop:poincare-constant-lyapunov} implies by the proof of \cite[Theorem 6.1]{meyn:tweedie:1993:3} that for all $y \in \rset^d$ and $t \geq 0$: $ \PL_t \Vlyb (y) \leq \Vlyb(y) + \bUndemi /\thetaUndemi$.  Then, using  \Cref{prop:drift-R-gaStep}, \Cref{lem:bound_inho_moment} and $\kgaStep \geq \gaStep^{-1}$ in \eqref{eq:gaStep_cont_proof2}, we get
\begin{multline*}
\sup_{i \in \{1,\cdots,\qgaStep\}}\VnormEq[\Vlysqrt]{\delta_x \QKer^{(i-1)\kgaStep}_\gaStep \PL_{\GaStep_{(i-1)\kgaStep+1,i \kgaStep}} - \delta_x \QKer^{i \kgaStep}_\gaStep }^2 \\
\leq 2^{-1}(1+\gaStep) \LU^2   \defEns{2\FunfoD(\lDriftgaStep,\bDriftRgaStep,\Vlyb(x)) +  \bUndemi /\thetaUndemi} \\
\times  \defEns{  \gaStep d+3^{-1}\gaStep^{2} \Vnorm[\Vlyb^{1/2}]{\nabla U}^2 \, \FunfoD(\lDriftgaStep,\bDriftRgaStep,\Vlyb(x))  } \eqsp.
\end{multline*}
Finally, $A$ can be bounded along the same lines.
\end{proof}

% \subsection{Proof of \Cref{coro:asympt_bias}}
% \label{sec:proof-crefcorobias}
% Under \Cref{assum:regularity}, $\RKer_{\gaStep}$ is irreducible with
% respect to the Lebesgue measure and weak Feller, which implies by
% \cite[Proposition 6.2.8]{meyn:tweedie:2009}, every compact set is
% small. And by \Cref{propo:drift-Euler} and \cite[Theorem
% 16.0.1]{meyn:tweedie:2009}, $\RKer_{\gaStep}$ is $\Vlysqrt$-uniformly
% ergodic: $\lim_{p \to \plusinfty}
% \Vnorm[\Vlysqrt]{\delta_x \RKer_{\gaStep}^p-\pi_{\gaStep}} = 0$ for all $x \in \rset^d$.  By
% \Cref{theo:convergence_TV_gaStep_const} and the triangle inequality,
% we have for all $x \in \rset^d$ and $p \geq 1$,
% \begin{equation*}
% \Vnorm[\Vlysqrt]{\pi-\pi_{\gaStep}} \leq  \Vnorm[\Vlysqrt]{\delta_x  \RKer_{\gaStep}^p-\pi}+  \Vnorm[\Vlysqrt]{\delta_x  \RKer_{\gaStep}^p-\pi_{\gaStep}} \leq
% \ConstVnormUnquart \rme^{-\rateExpVnormUnquart\gaStep p} \Vlysqrt(x)  + D(x,\gaStep)  +  \Vnorm[\Vlysqrt]{\delta_x  \RKer_{\gaStep}^p-\pi_{\gaStep}}\eqsp.
% \end{equation*}
% Letting $p \to \plusinfty$ and taking the minimum over $x \in \rset^d$ concludes the proof.

\subsection{Proof of \Cref{theo:convergence_TV_quantitatif_Poincare}}
\label{sec:theo:convergence_TV_quantitatif_Poincare}
Denote for $\gaStep > 0$, $r_\gaStep: \rset^d \times \rset^d \to \rset^d$ the transition density of $R_\gaStep$ defined for $x,y \in \rset^d$ by
\begin{equation}
\label{eq:definition_r_gaStep}
r_\gaStep(x,y) = (4 \pinumber \gaStep)^{-1} \exp(-(4\gaStep)^{-1}\norm[2]{y-x+\gaStep \nabla U(x)}) \eqsp.
 \end{equation}
For all $n \geq 1$, denote by $q_\gaStep^n: \rset^d \times \rset^d \to \rset^d$ the transition density associated with $Q_\gaStep^n$ defined by induction by: for all $x,y \in \rset^d$
\begin{equation}
\label{eq:definition_q_gaStep}
q_\gaStep^1(x,y)  = r_{\gaStep_1}(x,y) \eqsp, \qquad q_\gaStep^{n+1}(x,y) = \int_{\rset^d} q_\gaStep^n(x,z) r_{\gaStep_{n+1}}(z,y) \rmd z \text{ for $n \geq 1$} \eqsp.
\end{equation}
%In the proof of the Theorem, we need to have a bound on $\VarDeux{\pi}{\densityPiLigne{\delta_x Q_\gaStep^n}}$ for all $n \geq 1$,
%For this, we need some estimates for $q_\gaStep^n$ for all $n$ and for the normalizing constant  $\int_{\rset^d} \rme^{-U(y)} \rmd y$, which are given in the following two lemmas.
\begin{lemma}
Assume \Cref{assum:regularity}. Let $(\gaStep_k)_{k \geq 1}$ be a nonincreasing sequence with $\gaStep_1 < L$. Then for all $n \geq 1$ and $x,y \in \rset^d$,
\[
q_{\gaStep}^n(x,y) \leq \frac{\exp \parenthese{2^{-1}(U(x) - U(y)) -(2 \sigma_{\gaStep,n})^{-1} \norm[2]{y-x}}}{(2 \pinumber \sigma_{\gaStep,n} \prod_{i=1}^n(1-L\gaStep_i))^{d/2} } \eqsp,
\]
where $\sigma_{\gaStep,n} = \sum_{i=1}^n 2\gaStep_i(1-L \gaStep_i)^{-1}$.
\end{lemma}
\begin{proof}
Under \Cref{assum:regularity},  we have for all $x,y \in \rset^d$, $U(y) \leq U(x) + \ps{\nabla U(x)}{y-x} + (L/2) \norm[2]{y-x}$, which implies that  for all $\gaStep \in \ooint{0,L^{-1}}$
\begin{align}
\label{eq:bound_r_gaStep}
r_\gaStep(x,y) \leq (4 \pinumber \gaStep)^{-d/2}\exp \parenthese{2^{-1}(U(x) - U(y)) - (1-L \gaStep)(4 \gaStep)^{-1} \norm[2]{y-x} } \eqsp.
\end{align}
Then, the proof of the claimed inequality is by induction. By \eqref{eq:bound_r_gaStep}, the inequality holds for $n=1$. Now assume that it holds for $n \geq 1$. By induction hypothesis and \eqref{eq:bound_r_gaStep} applied for $\gaStep = \gaStep_{n+1}$, we have
\begin{align*}
&q_\gaStep^{n+1}(x,y)%= \int_{\rset^d} q_\gaStep^n(x,z)r_{\gaStep_{n+1}}(z,y) \rmd z
\leq (4 \pinumber \gaStep_{n+1})^{-d/2}\defEns{2 \pinumber \sigma_{\gaStep,n} \prod_{i=1}^n(1-L\gaStep_i)}^{-d/2} \exp \parenthese{2^{-1}(U(x) - U(y))} \\
& \phantom{aaaaa}
\times \int_{\rset^d} \exp \parenthese{-(2 \sigma_{\gaStep,n})^{-1} \norm[2]{z-x}  - (1-L \gaStep_{n+1})(4 \gaStep_{n+1})^{-1} \norm[2]{z-y}} \rmd z \\
&  \leq (4 \pinumber \gaStep_{n+1})^{-d/2}\defEns{2 \pinumber \sigma_{\gaStep,n} \prod_{i=1}^n(1-L\gaStep_i)}^{-d/2}( \sigma_{\gaStep,n}^{-1} + (1-L \gaStep_{n+1})/(2\gaStep_{n+1}))^{-d/2} \\
& \phantom{aaaaa}
\times  (2 \pinumber )^{d/2} \exp\parenthese{2^{-1}(U(x) - U(y)) -(2 \sigma_{\gaStep,n+1})^{-1} \norm[2]{y-x}} \eqsp.
\end{align*}
Rearranging terms in the last inequality concludes the proof.
\end{proof}
\begin{lemma}
\label{lem:control_normalization_constant}
Assume \Cref{assum:regularity} and \Cref{assum:drift-condition-reinforced}. Then $\int_{\rset^d} \rme^{-U(y)} \rmd y  \leq \vartheta_U$ where
\begin{equation}
\label{eq:def_theta_U}
\vartheta_U \eqdef  \rme^{ \Calpha}\frac{ (2\pinumber)^{(d+1)/2} (d-1)!}{\rhoU^d \Gammabf((d+1)/2)} \eqsp,
\end{equation}
and $\Calpha$ is given in \eqref{eq:calpha}.
      \end{lemma}
\begin{proof}
By \Cref{lem:majo:U}, for all $x \in \rset^d$, $ U(x) \geq \rho \norm{x-\xstar}/(\alpha+1) - \Calpha$.
Using the spherical coordinates, we get
\begin{equation*}
  \int_{\rset^d} \rme^{-U(y)} \rmd y \leq \rme^{ \Calpha}\defEns{ (2\pinumber)^{(d+1)/2} / \Gammabf((d+1)/2)} \int_{0}^{\plusinfty} \rme^{-\rho t/(\alpha+1)} t^{d-1} \rmd t \eqsp.
\end{equation*}
Then the proof is concluded by a straightforward calculation.
  \end{proof}

\begin{corollary}
\label{lem:control_variance}
Assume \Cref{assum:regularity} and \Cref{assum:drift-condition-reinforced}.
Let $(\gaStep_k)_{k \geq 1}$ be a nonincreasing sequence with $\gaStep_1 < L$. Then for all $n \geq 1$ and $x \in \rset^d$,
\[
\VarDeux {\pi}{\densityPi{\delta_x Q^n_\gaStep}} \leq \parenthese{\vartheta_U \exp(U(x))} \parenthese{4 \pinumber \defEns{\prod_{k=1}^n (1-L \gaStep_k)}^2\sum_{i=1}^n \frac{\gaStep_i}{1-L\gaStep_i}}^{-d/2} \eqsp,
\]
where $\vartheta_U$ is given by \eqref{eq:def_theta_U}.
  \end{corollary}

\begin{proof}[Proof of \Cref{theo:convergence_TV_quantitatif_Poincare}]
We bound the two terms of the right hand side of \eqref{eq:eq_base}.
The first term is dealt with the same reasoning as for the proof of \Cref{theo:convergence_TV_non_quantitatif}. Regarding the second term, by  \cite[Theorem~1.4]{bakry:barthe:cattiaux:guillin:2008}, $\pi$ satisfies a Poincaré inequality with constant $\log^{-1}(\kappa)$. Then, the claimed bound follows from  \eqref{eq:borne-poincare-TV} and \Cref{lem:control_variance}.
\end{proof}

\subsection{Proof of \Cref{propo:drift-Euler}}
\label{sec:proof-drift-conv}

Set $\cVEa = \rhoUl/4$ and for all $x \in \rset^d$, $\phi(x) =(\norm[2]{x-\xstar}+1)^{1/2}$ . Since $\phi$ is $1$-Lipschitz, we have by the log-Sobolev inequality \cite[Theorem 5.5]{boucheron:lugosi:massart:2013}  for all $x \in \rset^d$,
\begin{equation}
  \label{eq:base-drift-convex}
 \RKer_{\gaStep} \VEa (x) \leq \rme^{\cVEa \RKer_{\gaStep} \phi(x) + \cVEa^2 \gaStep}  \leq \rme^{\cVEa \sqrt{\norm[2]{x-\gaStep \nabla \VUl(x) -\xstar} + 2 \gaStep d+1} + \cVEa^2 \gaStep}  \eqsp.
\end{equation}
Under \Cref{assum:regularity} since $U$ is convex and $\xstar$ is a minimizer of $\VUl$, \cite[Theorem 2.1.5 Equation (2.1.7)]{nesterov:2004} shows that  for all $x \in \rset^d$,
% \begin{equation*}
%   %\label{eq:5}
%   U(x) + \ps{\nabla U(x)}{\xstar-x}+(2\LL)\norm[2]{\nabla U(x)} \leq U(\xstar) \eqsp,
% \end{equation*}
\begin{equation*}
%\label{eq:lem-convex-neste}
  \ps{\nabla \VUl(x)}{x-\xstar}\geq (2\LUl)^{-1}\norm[2]{\nabla \VUl(x)} + \rhoUl\norm{x-\xstar}\1_{\{\norm{x-\xstar} \geq \RUl \}}  \eqsp,
\end{equation*}
which implies that  for all $x \in \rset^d$ and $\gaStep \in \ocint{0,\LUl^{-1}}$, we have
\begin{equation}
  % \norm[2]{x-\gaStep \nabla \VUl(x) - \xstar} \leq \norm[2]{x-\xstar}+\gaStep(\gaStep -\LUl^{-1})\norm[2]{\nabla \VUl(x)} -\rhoUl\norm{x-\xstar}\1_{\{\norm{x-\xstar} \geq \RUl \}} \eqsp.
\label{eq:lem-convex-neste}
  \norm[2]{x-\gaStep \nabla \VUl(x) - \xstar} \leq \norm[2]{x-\xstar} -2\gaStep \rhoUl\norm{x-\xstar}\1_{\{\norm{x-\xstar} \geq \RUl \}} \eqsp.
\end{equation}
Using this inequality and for  all $u \in \ccint{0,1}$, $(1-u)^{1/2}-1 \leq -u/2$, we have for all $x \in \rset^d$, satisfying $\norm{x-\xstar} \geq \KUa = \max(1,2 d \rhoUl^{-1},\RUl)$,
\begin{align*}
&\parenthese{\norm[2]{x-\gaStep \nabla \VUl(x) -\xstar} + 2 \gaStep d +1}^{1/2} - \phi(x)\\
&\qquad \leq \phi(x) \defEns{\parenthese{1 - 2\gaStep \phi^{-2}(x)( \rhoUl \norm{x-\xstar}-d   ) }^{1/2}-1} \\
& \qquad   \leq -\gaStep \phi^{-1}(x)(\rhoUl \norm{x-\xstar}-d  )
% &\leq \phi(x) \defEns{\parenthese{1 -\gaStep \rhoUl \norm{x-\xstar} \phi^{-2}(x)   }-1} \\
% & \text{ where we used } \norm{x-\xstar} \geq 2 \rhoUl^{-1} \\
 \leq -(\rhoUl \gaStep/2)\norm{x-\xstar} \phi^{-1}(x) \leq-  2^{-3/2}\rhoUl \gaStep  \eqsp.
% & \text{ where we used } \norm{x-\xstar} \geq 1 \eqsp, t (t^2+1)^{-1/2} \text{ is nonincreasing on } \rset_+\\
\end{align*}
Combining this inequality and \eqref{eq:base-drift-convex}, we get for all $x \in \rset^d$, $\norm{x-\xstar} \geq \KUa$,
\begin{equation*}
  \RKer_{\gaStep} \VEa (x)/ \VEa (x) \leq \rme^{\gaStep\cVEa(   \cVEa -2^{-3/2} \rhoUl ) } = \thetaEa^{\gaStep} \eqsp.
\end{equation*}
 By \eqref{eq:lem-convex-neste} and  the inequality for all $a,b \geq 0$, $\sqrt{a+1+b}-\sqrt{1+b} \leq a/2$, we get for all $x \in \rset^d$,
\begin{equation*}
  \sqrt{\norm[2]{x-\gaStep \nabla \VUl(x) -\xstar} + 2 \gaStep d  +1} - \phi(x) \leq \gaStep d \eqsp.
\end{equation*}
Then using this inequality in \eqref{eq:base-drift-convex}, we have for all $x \in \rset^d$,
 \begin{equation*}
%   \label{eq:4}
 \RKer_{\gaStep} \VEa (x) \leq \thetaEa^{\gaStep}  \VEa (x) +   \left( \rme^{\chi \gaStep(d+\chi) } - \thetaEa^{\gaStep} \right) \rme^{\rhoUl (\KUa^2+1)^{1/2}/4} \1_{\ball{\xstar}{\KUa}}(x) \eqsp.
 \end{equation*}
Using the inequality for all $t \geq 0$, $\rme^t -1 \leq t \rme^t$ concludes the proof.

\subsection{Proof of  \Cref{coro:convex_A_gamma}}
\label{sec:proof-crefth-stepsb_coro}
We preface the proof by a Lemma.
\begin{lemma}
\label{coro:W_c}
Assume \Cref{assum:regularity} and that $U$ is convex. Let $\sequence{\gaStep}[k][\nset^*]$ be a nonincreasing sequence with $\gaStep_1 \leq \LUl^{-1}$.
For all $n \geq 0$ and $x \in \rset^d$,
\begin{equation*}
\int_{\rset^d} \norm{y-\xstar }^2\QKer^n_{\gaStep}(x,\rmd y) \leq \defEns{4 \rhoUl^{-1} \parentheseDeux{ 1+\log \defEns{\FunfoD(\thetaEa,\bEa, \gaStep_1,\VEa(x))}}}^2\eqsp,
\end{equation*}
where $\VEa,\thetaEa,\bEa$ are given in \eqref{eq:def_Wc} and \Cref{propo:drift-Euler} respectively.
\end{lemma}
\begin{proof}

 Let $n \geq 0$ and $x \in \rset^d$. Consider the function $\phi
  :\rset \to \rset$ defined by for all $t \in \rset$, $\phi(t) =
  \exp\defEns{(\rhoU/4)(t+(4/\rhoU)^2)^{1/2}}$. Since this function is
  convex on $\rset_+$, we have by the Jensen inequality and the
  inequality for all $t \geq 0$, $\phi(t) \leq
  \rme^{1+(\rhoU/4)(t+1)^{1/2}}$,
\begin{equation*}
\phi\parenthese{\int_{\rset^d} \norm{y-\xstar }^2\QKer^n_{\gaStep}(x,\rmd y)}
% \leq \int_{\rset^d}\phi\parenthese{ \norm{y-\xstar }^2}\QKer^n_{\gaStep}(x,\rmd y)
\leq \rme^{1} \QKer^n_{\gaStep}\VEa(x) \\
\eqsp.
\end{equation*}
The proof is then completed using \Cref{propo:drift-Euler}, \Cref{lem:bound_inho_moment} and that $\phi$ is one-to-one with for all $t \geq 1 $, $ \phi^{-1}(t) \leq  \parenthese{4 \rhoU^{-1} \log(t)}^2$.
% \begin{equation*}
% \phi^{-1}(t) \leq  \parenthese{4 \rhoU^{-1} \log(t)}^2\eqsp.
% \end{equation*}
\end{proof}
\begin{proof}[Proof of \Cref{coro:convex_A_gamma}]
Using $\nabla \VUl(\xstar) = 0$,
  \Cref{assum:regularity} and \Cref{coro:W_c}, we have for all $k \geq
  0$,
\begin{equation*}
%  \label{eq:11}
  \int_{\rset^d}\norm[2]{\nabla \VUl(y)} \QKer^k_{\gaStep}(x,\rmd y) \leq \LU^2 \parenthese{4 \rhoUl^{-1} \defEns{ 1+\log \defEns{\FunfoD(\thetaEa,\bEa,\gaStep_1,\VEa(x))}}}^2 \eqsp.
\end{equation*}
\end{proof}
\subsection{Proof of \Cref{theo:convergence_TV_dec-stepsize-Bob}}
\label{sec:proof-crefth-stepsb}

We preface the proof by a Lemma.

\begin{lemma}
\label{lem:control_normalization_constantb}
Assume \Cref{assum:regularity} and that $U$ is convex. Then
\begin{equation}
\label{eq:def_theta_Ub}
\int_{\rset^d} \rme^{-U(y)} \rmd y  \leq   \parenthese{\frac{(2\pinumber)^{(d+1)/2}(d-1)!}{\rhoU^d \Gammabf((d+1)/2)} +
   \frac{\pinumber^{d/2} \RUl^{d}}{\Gammabf(d/2+1)}} \eqsp.
\end{equation}
      \end{lemma}
\begin{proof}
By \eqref{eq:superexpo_potential}  and $U(\xstar) = 0$, we have
\[
\int_{\rset^d} \rme^{-U(y)} \rmd y \leq \int_{\rset^d} \rme^{-
   \rhoU \norm{y-\xstar}} \rmd y+  \int_{\rset^d} \1_{\{ \norm{y-\xstar} \leq   \RUl\}} \rmd y \eqsp.
\]
Then the proof is concluded using the spherical coordinates.
  \end{proof}

\begin{proof}[Proof of \Cref{theo:convergence_TV_dec-stepsize-Bob}]
  By \cite[Theorem 1.2]{bobkov:1999}, $\pi$ satisfies a Poincaré
  inequality with constant $\log^{-1}(\kappa)$. Therefore, the second
  term in \eqref{eq:eq_base} is dealt as in the proof of
  \Cref{theo:convergence_TV_quantitatif_Poincare} using
  \eqref{eq:borne-poincare-TV},
  \Cref{lem:control_normalization_constantb} and
  \Cref{lem:control_normalization_constant}. 
\end{proof}

\subsection{Proof of \Cref{propo:drift_exp_conv}}
\label{proof:prop:drift_exp_conv}
For all $x \in \rset^d$, we have
\begin{multline*}
\generatorL \VEa(x) =  \frac{\rhoU \VEa(x)}{4(\norm{x-\xstar}^2+1)^{1/2} }   \left \lbrace   (\rhoU/4)(\norm{x-\xstar}^2+1)^{-1/2} \norm[2]{x-\xstar} \right.\\
\left. -\ps{\nabla U(x)}{x-\xstar} - (\norm{x-\xstar}^2+1)^{-1} \norm[2]{x-\xstar} + d \right \rbrace \eqsp.
\end{multline*}
By \eqref{eq:superexpo_potential}, $\ps{\nabla U(x)}{x-\xstar} \geq \rhoU\norm{x-\xstar}$ for all $x \in \rset^d$,
$\norm{x-\xstar} \geq \RUl$.  Then, for all $x$, $\norm{x-\xstar}
\geq \KOne = \max(\RUl,4d/ \rhoU,1)$, $\generatorL \VEa(x) \leq
-(\rhoU^2/8) \VEa(x)$. In addition, since $U$ is convex and $\nabla U(\xstar)=0$, for all $x \in \rset^d$, $\ps{\nabla U(x)}{x-\xstar} \geq 0$ and we get $\sup_{\{x \in \Setdrift\}} \generatorL
\VEa(x) \leq \bOne$.

\subsection{Proof of \Cref{propo:drift_R_gamma_Stconv}}
\label{sec:proof_drift_strong_conv}
Under \Cref{assum:regularity}, using that $\nabla \VUl(\xstar) = 0$, we get for all $x \in \rset^d$,
 \begin{multline}
\label{eq:15}
  \int_{\rset^d }\norm[2]{y-\xstar}   \RKer_{\gaStep}(x,\rmd y) = \norm[2]{x-\xstar+\gaStep (\nabla \VUl(\xstar)-\nabla \VUl(x)) } + 2 \gaStep d \\  \leq (1+(\LF \gaStep)^{2}) \norm[2]{x-\xstar} - 2 \gaStep \ps{\nabla \VUl(x) - \nabla \VUl(\xstar) }{x-\xstar} + 2 \gaStep d \eqsp.
\end{multline}
Then for all $x \in \rset^d$, $\norm{x-\xstar} \geq \RSt$, we get using for all $t \geq 0$, $1-t \leq \rme^{-t}$
\begin{equation*}
%\label{eq:16}
    \int_{\rset^d }\norm[2]{y-\xstar}   \RKer_{\gaStep}(x,\rmd y)  \leq \thetaESt^{\gaStep} \norm[2]{x-\xstar} + 2 \gaStep d \eqsp.
\end{equation*}
Using again \eqref{eq:15} and the convexity of $U$, it yields for all $x \in \rset^d$, $\norm{x-\xstar} \leq \RSt$,
\begin{equation*}
%  \label{eq:17}
   \int_{\rset^d }\norm[2]{y-\xstar}   \RKer_{\gaStep}(x,\rmd y)  \leq \gaStep \bESt \eqsp,
\end{equation*}
which concludes the proof.

\subsection{Proof of \Cref{theo:kind_drift}}
\label{proof:theo:kind_drift}
We preface the proof by a lemma.

\begin{lemma}
  \label{lem:pert_norm}
Assume \Cref{assum:pertSt}. Then, for all $x \in \rset^d$,
\begin{equation*}
   \norm{x-\gaStep \nabla U(x) -\xstarun}^2 \leq (1-\kappaS\gaStep/2)\norm{x -\xstarun}^2+\gaStep(\gaStep+2\kappaS^{-1}) \normsup[2]{\nabla \Ude}
% + \gaStep\left(\gaStep - \frac{2}{m+\Lun} \right)\norm{\nabla \Uun(x) - \nabla \Uun(\xstarun)}^2
\eqsp.
\end{equation*}
\end{lemma}
\begin{proof}
Using that for all $y,z \in \rset^d$, $\norm[2]{y+z} \leq (1+\kappaS\gaStep/2)\norm[2]{y}+(1+2(\kappaS\gaStep)^{-1})\norm[2]{z}$, we get under \Cref{assum:pertSt}-\ref{assum:pertSt-b}:
\begin{multline}
  \label{eq:lem:pert_norm}
  \norm{x-\gaStep \nabla U(x) -\xstarun}^2 \leq (1+\kappaS\gaStep/2)\norm{x-\gaStep \nabla \Uun(x) -\xstarun}^2\\
+\gaStep(\gaStep+2\kappaS^{-1}) \normsup[2]{\nabla \Ude} \eqsp.
\end{multline}
By \cite[Theorem 2.1.12, Theorem 2.1.9]{nesterov:2004}, \Cref{assum:pertSt}-\ref{assum:pertSt-b} implies that for all $x,y \in \rset^d$:
\begin{equation*}
%\label{eq:convex_forte_contra1}
\ps{\nabla \Uun(y) - \nabla \Uun(x)}{y-x} \geq (\kappaS/2)\norm[2]{y-x} + \frac{1}{\mStD+\Lun} \norm[2]{\nabla \Uun(y) - \nabla \Uun(x)} \eqsp,  \\
\end{equation*}
Using this inequality and  $\nabla \Uun (\xstarun ) = 0$ in \eqref{eq:lem:pert_norm} concludes the proof.
\end{proof}

\begin{proof}[Proof of \Cref{theo:kind_drift}]
For any $\gaStep \in \ooint{0,2/(m+\Lun)}$, we have for all $x \in \rset^d$:
\begin{multline*}
%\label{eq:kind_drift_1}
\int_{\rset^d} \norm[2]{y-\xstarun}   R_\gaStep (x,\rmd y) =  \norm{x-\gaStep \nabla U(x) -\xstarun}^2 + 2 \gaStep d \\
\phantom{aaa}\leq (1-\kappaS\gaStep/2)\norm{x -\xstarun}^2+\gaStep\defEns{(\gaStep+2\kappaS^{-1}) \normsup[2]{\nabla \Ude}+2d} \eqsp,
\end{multline*}
where we have used \Cref{lem:pert_norm} for the last inequality.
Since $\gaStep_1 \leq 2/(m+\Lun)$ and $(\gaStep_k)_{k\geq 1}$ is nonincreasing, by a straightforward induction,  for $p \geq 1$ and $x \in \rset^d$,
\begin{multline}
\label{eq:lem_kind_drift_first_estimate}
\int_{\rset^d} \norm{y-\xstarun}^2   Q^{p}_\gaStep(x,\rmd y) \leq \prod_{k=1}^p (1-\kappaS \gaStep_k/2)\norm[2]{x-\xstarun}\\
 + ((\gaStep_1+2\kappaS^{-1}) \normsup[2]{\nabla \Ude} +2d)\sum_{i=n}^p \prod_{k=i+1}^p (1-\kappaS \gaStep_k/2) \gaStep_i\eqsp.
\end{multline}
Consider the second term in the right hand side of  \eqref{eq:lem_kind_drift_first_estimate}. Since $\gaStep_1 \leq 2/(m+\Lun)$, $ m \leq \Lun$ and $(\gaStep_k)_{k \geq 1}$ is nonincreasing, $\max_{k \geq 1} \gaStep_k \leq \kappaS^{-1}$ and therefore:
\begin{multline*}
%\label{eq:lem_kind_drift_third_estimate}
\sum_{i=n}^p \prod_{k=i+1}^p (1-\kappaS \gaStep_k/2) \gaStep_i
\\
\leq \kappaS^{-1}  \sum_{i=n}^p \defEns{\prod_{k=i+1}^p \left( 1-\kappaS \gaStep_k/2  \right) -\prod_{k=i}^p \left( 1-\kappaS \gaStep_k /2\right) }  \leq 2 \kappaS^{-1} \eqsp.
\end{multline*}
\end{proof}

\subsection{Proof of \Cref{theo:convergence_TV_quantitatif_logSob}}
\label{sec:theo:convergence_TV_quantitatif_logSob}
We preface the proof of the Theorem by a preliminary lemma.
%This gives explicit bounds on $\Ent{\pi}{\densityPiLigne{\delta_x Q^n_\gaStep}}$ for all $x \in \rset^d$ and $n \geq 1$, which will be used combined with \Cref{theo:log_sob_fort_convex}.
\begin{lemma}
\label{lem:control_entropie}
Assume \Cref{assum:pertSt}.
 % \begin{enumerate}[label=(\roman*)]
 % \item
 % \label{item1:control_entropie}
  Let $\gaStep \in \ooint{0,2/(m+\Lun)}$, then for all $x \in \rset^d$,
 \begin{multline*}
 \Ent{\pi}{\densityPi{\delta_x R_\gaStep}} \leq (\Lun/2)\defEns{ (1-\kappaS\gaStep/2)\norm{x -\xstarun}^2+\gaStep(\gaStep+2\kappaS^{-1}) \normsup[2]{\nabla \Ude}} \\+\osc[\rset^d]{\Ude} -(d/2)(1+\log(2 \gaStep m) -2\Lun\gaStep ) \eqsp.
 \end{multline*}
%  \item
%   \label{item2:control_entropie}
%  Let $(\gaStep_k)_{k \geq 1}$ be a nonincreasing sequence with $\gaStep_1 \leq 2/(m+\Lun)$. Then for all $x \in \rset^d$ and $n \geq 1$,
% \begin{multline*}
% \Ent{\pi}{\densityPiLigne{\delta_x Q^n_\gaStep}} \leq
%  (\Lun\rme^{-\kappaS \GaStep_n/2}/2) \norm{x -\xstarun}^2+ \gaStep_n(\gaStep_n+2\kappaS^{-1}) \normsup[2]{\nabla \Ude}
%  +\osc[\rset^d]{\Ude} \\
% +(\Lun\kappaS^{-1}/2)(1-\kappaS \gaStep_n)(2d+(\gaStep_{1}+2\kappaS^{-1}) \normsup[2]{\nabla \Ude}   )
% -(d/2)(1+\log(2 \gaStep_n m) -2\Lun\gaStep_n )
%  \eqsp.
% \end{multline*}
% \end{enumerate}
\end{lemma}
\begin{proof}
% \begin{enumerate}[label=(\roman*), wide=0pt, labelindent=\parindent]
% \item
Let $\gaStep \in \ooint{0,2/(m+\Lun)}$ and $r_\gaStep$ be the transition density of $R_\gaStep$ given by \eqref{eq:definition_r_gaStep}. Under \Cref{assum:pertSt}-\ref{assum:pertSt-a} by \cite[Theorems 2.1.8-2.1.9]{nesterov:2004}, we have for all $x \in \rset^d$, $\Uun(x) \leq \Uun(\xstarun) + (\Lun/2)\norm[2]{x-\xstarun}$
% \begin{equation}
%   \label{eq:st_convex_1}
% \Uun(x) \leq \Uun(\xstarun) + (\Lun/2)\norm[2]{x-\xstarun} \eqsp.
% \end{equation}
Therefore we have for all $x \in \rset^d$
\begin{multline}
  \label{eq:controle_entropie_1}
  \Ent{\pi}{\densityPi{\delta_x R_\gaStep}} = \int_{\rset^d} \log(r_\gaStep(x,y)/\pi(y)) r_\gaStep(x,y) \rmd x \\
\leq  R_{\gaStep} \funtestcontrolEntropie(x) - (d/2)(1+\log(4\pinumber \gaStep) ) \eqsp,
\end{multline}
where $\funtestcontrolEntropie: \rset^d \to \rset$ is the function defined for all $y \in \rset^d$ by
\begin{equation*}
  \funtestcontrolEntropie(y) = \Ude(y) + \Uun(\xstarun) + (\Lun/2)\norm[2]{y - \xstarun} + \log\parenthese{\int_{\rset^d} \rme^{-U(z)} \rmd z} \eqsp.
\end{equation*}
By \Cref{assum:pertSt}-\ref{assum:pertSt-b} and \Cref{lem:pert_norm}, we get for all $x \in \rset^d$:
\begin{align*}
  \nonumber
R_{\gaStep} \funtestcontrolEntropie(x) &\leq  (\Lun/2)\norm[2]{x-\gaStep \nabla U(x) - \xstarun}+ \log\parenthese{\int_{\rset^d} \rme^{-\Uun(z)+\Uun(\xstarun)} \rmd z}\\
\nonumber & \phantom{(\Lun/2)\norm[2]{x-\gaStep \nabla U(x) - \xstarun}}+\osc[\rset^d]{\Ude} + d\Lun\gaStep  \\
\nonumber
& \leq   (\Lun/2)\defEns{ (1-\kappaS\gaStep/2)\norm{x -\xstarun}^2+\gaStep(\gaStep+2\kappaS^{-1}) \normsup[2]{\nabla \Ude}}\\
\nonumber &  \phantom{(\Lun/2)\norm[2]{x-\gaStep \nabla U(x) - \xstarun}} +\osc[\rset^d]{\Ude} + d\Lun\gaStep \eqsp.
\end{align*}
Plugging this bound in \eqref{eq:controle_entropie_1} gives the desired result.
\end{proof}
\begin{proof}[Proof of \Cref{theo:convergence_TV_quantitatif_logSob}]
  We first deal with the second term in the right hand side of
  \eqref{eq:eq_base}. Under \Cref{assum:pertSt}, \cite[Corollary
  5.7.2]{bakry:gentil:ledoux:2014} and the Holley-Stroock perturbation
  principle \cite[p. 1184]{holley:stroock:1987} show that  $\pi$ satisfies a
  log-Sobolev inequality with constant $\Clogsob =
  -\log^{-1}(\kappa)$. So by \eqref{eq:convergence_log_sob} we have
\begin{equation*}
  \tvnorm{\delta_x \QKer_{\gaStep}^n \PL_t - \pi} \leq \kappa^t \defEns{2\, \Ent{\pi}{\densityPi{\delta_x \QKer_{\gaStep}^n}}}^{1/2}  \eqsp.
\end{equation*}
We now bound $\Ent{\pi}{\densityPiLigne{\delta_x \QKer_{\gaStep}^n}}$
  which will imply the upper bound of $C(\delta_x \QKer_{\gaStep}^n)$.
We proceed by induction. For $n =1$, it is \Cref{lem:control_entropie}. For $n \geq 2$, by \eqref{eq:definition_q_gaStep} and the Jensen inequality applied to the convex function $t \mapsto t \log(t)$, we have for all $x \in \rset^d$ and $ n \geq 1$,
\begin{align}
\nonumber
&\Ent{\pi}{\densityPiLigne{\delta_x Q^n_\gaStep}} \\
\nonumber
&= \int_{\rset^d} \log\defEns{\pi^{-1}(y) \int_{\rset^d} q^{n-1}_\gaStep (x,z) r_{\gaStep_n}(z,y) \rmd z } \int_{\rset^d} q^{n-1}_\gaStep (x,z) r_{\gaStep_n}(z,y) \rmd z  \rmd y \\
\label{eq:control_entro_1}
& \leq \int_{\rset^d} \int_{\rset^d} \log\defEns{ r_{\gaStep_n}(z,y) \pi^{-1}(y)}  q^{n-1}_\gaStep (x,z) r_{\gaStep_n}(z,y) \rmd z  \rmd y \eqsp.
\end{align}
Using Fubini's theorem, \Cref{lem:control_entropie}, \Cref{theo:kind_drift}, and the inequality $t\geq 0$, $1-t \leq \rme^{-t}$ in \eqref{eq:control_entro_1} implies the bound of $C(\delta_x \QKer_{\gaStep}^n) $.

Finally,  $A(\gaStep,x)$ is bounded using the inequality for all $y,z \in \rset^d$, $\norm[2]{y+z} \leq 2(\norm[2]{y}+\norm[2]{z})$, \Cref{assum:pertSt} and \Cref{theo:kind_drift}.
\end{proof}

%%% Local Variables:
%%% mode: latex
%%% TeX-master: "main"
%%% End:

\section{Quantitative convergence bounds in total variation for diffusions}
\sectionmark{Quantitative convergence bounds  for diffusions}
\label{sec:quant-conv-bounds}
In this part, we derived  quantitative convergence results in
total variation norm for $d$-dimensional SDEs of the form
\begin{equation}
  \label{eq:SDE_basics}
  \rmd \XSDE_t = b(\XSDE_t) \rmd t +  \rmd B_t^d \eqsp,
\end{equation}
started at $\XSDE_0$, where $(B_t^d)_{t \geq 0}$ is a $d$-dimensional
standard Brownian motion and $b:\rset^d \to \rset^d$ satisfies the following assumptions.
\begin{assumptionG}
  \label{assum:drift}
   $b$ is Lipschitz and  for all $x,y \in \rset^d$, $
  \ps{b(x) - b(y)}{x-y} \leq 0$.
\end{assumptionG}
Under \Cref{assum:drift}, \cite[Theorems 2.4-3.1-6.1, Chapter
IV]{ikeda:watanabe:1989} imply that there exists a unique solution
$(\XSDE_t)_{t \geq 0}$ to \eqref{eq:SDE_basics} for all initial point
$x \in \rset^d$, which is strongly Markovian. Denote by $(\PSDE_t)_{ t
  \geq 0}$ the transition semigroup associated with
\eqref{eq:SDE_basics}.  To derive explicit bound for
$\tvnorm{\PSDE_t(x,\cdot) - \PSDE_t(y,\cdot)}$, we use the coupling by
reflection, introduced in \cite{lindvall:rogers:1986} to show convergence in total variation
norm for solution of SDE, and recently used by \cite{eberle:2015} to
obtain exponential convergence in the Wasserstein distance of order
$1$. This coupling
is defined as (see \cite[Example 3.7]{chen:shao:1989}) the unique
strong Markovian process $(\XSDE_t,\YSDE_t)_{ t \geq 0}$ on
$\rset^{2d}$, solving the SDE:
\begin{equation}
  \label{eq:def_couplage_par_reflection}
\begin{cases}
  \rmd \XSDE_t &= b(\XSDE_t) \rmd t + \rmd B_t^d \\
 \rmd \YSDE_t &= b(\YSDE_t) \rmd t +  (\Id - 2 e_t e_t^T)
  \rmd B_t^d   \eqsp,
\end{cases}
\quad \text{ where } e_t =\fune(\XSDE_t -\YSDE_t)
\end{equation}
with  $\fune(z) = z/\norm{z}$ for $z \not=0$ and $\fune(0) = 0$ otherwise.
Define the coupling time
\begin{equation}
  \label{eq:8}
\tau_c =  \inf \{s \geq 0 \ | \ \XSDE_s \not = \YSDE_s \} \eqsp.
\end{equation}
By construction $\XSDE_t = \YSDE_t$ for $t\geq \tau_c$. We denote in the sequel by
$\tilde{\mathbb{P}}_{(x,y)}$ and $\tilde{\mathbb{E}}_{(x,y)}$ the
probability and the expectation associated with the SDE \eqref{eq:def_couplage_par_reflection} started
at $(x,y) \in \rset^{2d}$ on the canonical space of continuous
function from $\rset_+$ to $\rset^{2d}$. We denote by $(\filtrationTilde_t)_{t
  \geq 0}$ the canonical filtration. Since $\barB_t^d = \int_0^t (\Id
- 2 e_s e_s^T) \rmd B_s^d$ is a $d$-dimensional Brownian motion, the
marginal processes $(\XSDE_t)_{t \geq 0}$ and $(\YSDE_t)_{t \geq 0}$
are under $\tilde{\mathbb{P}}_{(x,y)}$ weak solutions to \eqref{eq:SDE_basics} started at $x$ and $y$
respectively.   The results in \cite{lindvall:rogers:1986} are derived under less
stringent conditions than \Cref{assum:drift}, but do not provide quantitative estimates.
\begin{proposition}[\protect{\cite[Example 5]{lindvall:rogers:1986}}]
  \label{propo:reflection_coupling_SDE}
  Assume \Cref{assum:drift} and let $(\XSDE_t,\YSDE_t)_{t \geq 0}$ be the solution of \eqref{eq:def_couplage_par_reflection}. Then for all $t \geq 0$ and $x,y \in \rset^d$, we have
  \begin{equation*}
     \PPMb{(x,y)}{\tau_c > t } = \PPMb{(x,y)}{\XSDE_t \not = \YSDE_t} \leq 2\left( \bfPhi\defEns{ \left(2t^{1/2}\right)^{-1}\norm{x-y}} -1/2 \right) \eqsp.
  \end{equation*}
 %  \begin{equation*}
 % \eqsp.
 %  \end{equation*}
\end{proposition}
\begin{proof}
For $t < \tau_c$, $\XSDE_t-\YSDE_t$ is
the solution of the SDE
\begin{equation*}
  % \label{eq:difference_reflection}
  \rmd \{ \XSDE_t - \YSDE_t \} = \defEns{b(\XSDE_t)-b(\YSDE_t)} \rmd t + 2  e_t \rmd B_t^1 \eqsp,
\end{equation*}
where $B_t^1 = \int_0^t \2{s < \tau_c} e_s^T \rmd B_s^d$.  Using
the Itô's formula and \Cref{assum:drift}, we have for all $t <
\tau_c$,
\begin{equation*}
%  \label{eq:difference_norm_reflection}
  \norm{\XSDE_t-\YSDE_t} = \norm{x-y} + \int_0^{t}\ps{b(\XSDE_s)-b(\YSDE_s)}{e_s} \rmd s +2B_t^1  \leq \norm{x-y} + 2 B_t^1 \eqsp.
\end{equation*}
Therefore, for all $x,y \in \rset^d $ and $t \geq 0$, we get
\begin{align*}
\PPMb{(x,y)}{\tau_c > t} &\leq  \PPMb{(x,y)}{\min_{0 \leq s \leq t} B_s^1 \geq \norm{x-y}/2} \\
&= \PPMb{(x,y)}{\max_{0 \leq s \leq t} B_s^1 \leq \norm{x-y}/2}
= \PPMb{(x,y)}{|B_t^1| \leq \norm{x-y}/2} \eqsp,
\end{align*}
where we have used the reflection principle in the last identity.
% \begin{equation}
%   \label{eq:reflection_small_set}
%   \PPMb{(x,y)}{\XSDE_t \not = \YSDE_t} \leq 2 \PPMb{(x,y)}{0 \leq 2 \rmB_t \leq \norm[][\sigma^{-2}]{x-y}} \leq 2\left( \bfPhi\left( (2t)^{-1}\norm[][\sigma^{-2}]{x-y}\right) -1/2 \right)\eqsp,
% \end{equation}
\end{proof}
% Therefore, we get back \cite[Example 5]{lindvall:rogers:1986}, which
% implies that under \Cref{assum:drift}$(b,\sigma)$, $\lim_{t \to \plusinfty}
% \tvnorm{\PSDE_t(x,\cdot) - \PSDE_t(y,\cdot)} = 0$. However, the obtained
% bounds in this case are quite pessimistic. For example for the
% Ornstein–Uhlenbeck process, for which exponential convergence to
% stationary distribution is well known (see (ref...)).
Define  for $R >0$ the set
$ \diagSet_R= \{x,y \in \rset^d \ | \ \norm{x-y} \leq R\}$.
\Cref{propo:reflection_coupling_SDE} and Lindvall's inequality give that, for all
$\epsilon \in \ooint{0,1}$ and $t \geq \Fsmall(\epsilon,R)$,
\begin{equation}
  \label{eq:uniform-bound}
  \sup_{(x,y) \in \diagSet_R} \tvnorm{\PSDE_t(x,\cdot) - \PSDE_t(y,\cdot)} \leq 2 (1-\epsilon) \eqsp,
\end{equation}
where $\Fsmall$ is defined in \eqref{eq:Fsmall}.
To obtain quantitative exponential bounds in total variation for any $x,y \in \rset^d$, it is required to control some exponential moments of the successive
return times to $\diagSet_R$. This is first achieved by using a drift
condition for the generator $\generator$ associated with the
SDE \eqref{eq:SDE_basics} defined for all
$f \in C^2(\rset^d)$ by
\begin{equation*}
%\label{eq:generator-diffusion}
\generator f = \ps{b}{\nabla f} + (1/2) \Delta f  \eqsp.
\end{equation*}
 Consider the following assumption:
\begin{assumptionG}
  \label{assum:drift-simple}
\begin{enumerate}[label=(\roman*)]
\item
\label{item:assum:drift-simple}
There exist a twice continuously differentiable function
$ \Vdrifta: \rset^d \mapsto \coint{1,\infty}$ and constants $\thetadrift
  > 0$, $\bdrift \geq 0$ such that
\begin{equation}
\label{eq:lyapunov-conditionb}
\generator \Vdrifta \leq -\thetadrift \Vdrifta + \bdrift  \eqsp.
\end{equation}
\item
\label{item:2}
%There exists $\deltaSS, \RdriftSS >0$ and $\xbar \in \rset^d$ such that  $\{ z \in \rset^d \ | \ \Vdrifta(z) \leq \bdrift + \deltaSS/2 \} \subset \ball{\xbar}{\RdriftSS/2}$.
There exists $\deltaSS >0$ and $\RdriftSS > 0$ such that  $\Thetadrift \subset \diagSet_{\RdriftSS}$ where
\begin{equation}
  \label{eq:definition-ThetaDrift}
\Thetadrift = \{ (x,y) \in \rset^{2d} \ | \ \Vdrifta(x) + \Vdrifta(y) \leq 2 \thetadrift^{-1} \bdrift +\delta \} \eqsp.
\end{equation}
\end{enumerate}
\end{assumptionG}
For $\time > 0$, and $\closed$ a closed
subset of $\rset^{2d}$, define by  $\tpsRetour_1^{\closed,\time}$ the first return time to $\closed$ delayed by
$\time$:
\begin{equation*}
\label{eq:def_tpsRetour_1}
\tpsRetour_1^{\closed,\time} = \inf\defEns{s \geq \time \ | \ (\XSDE_s,\YSDE_s) \in \closed  } \eqsp.
\end{equation*}
For $j \geq 2$, define recursively the $j$-th return time to $\closed$
delayed by $\time$ by
\begin{equation}
  \label{eq:def_tpsRetour_j}
\tpsRetour_j^{\closed,\time} = \inf\{s \geq
\tpsRetour_{j-1}^{\closed,\time}+\time \ | \ (\XSDE_s,\YSDE_s) \in \closed \} =
\tpsRetour_{j-1}^{\closed,\time} + \tpsRetour_1^{\closed,\time} \circ \shift_{\tpsRetour_{j-1}^{\closed,\time}} \eqsp,
\end{equation}
where $\shift$ is the shift operator on the canonical space.  By \cite[Proposition 1.5 Chapter 2]{ethier:kurtz:1986},
the sequence $(\tpsRetour_j^{\closed,\time})_{j \geq 1}$ is a
sequence of stopping time with respect to $(\filtrationTilde_t)_{t
  \geq 0}$.

%\begin{equation}
%  \label{eq:definition_tps_retour_convex}
%\tpsRetour_j^{\epsilon} =
%\tpsRetour_j^{\Thetadrift,\eqspD \Fsmall(\epsilon,\RdriftSS)} \eqsp,
%\end{equation}
% where
%$\deltaSS$, $\RdriftSS$ are given in
%\Cref{assum:drift-simple}, $\Fsmall$ in \eqref{eq:Fsmall} and $\Thetadrift$ in \eqref{eq:definition-ThetaDrift}.

\begin{proposition}
  \label{propo:drift-condition}
  Assume \Cref{assum:drift} and \Cref{assum:drift-simple}.
For all $x,y \in \rset^{d}$, $\epsilon \in \ooint{0,1}$ and $j \geq 1$, we have
  \begin{equation*}
    \PEMb{(x,y)}{\rme^{\tilde{\thetadrift} \tpsRetour_j^{\Thetadrift,\eqspD \Fsmall(\epsilon,\RdriftSS)}}} \leq  \{\boundM(\epsilon)\}^{j-1} \defEns{  (1/2)(\Vdrifta(x) + \Vdrifta(y)) + \rme^{\tilde{\thetadrift}\Fsmall(\epsilon,\RdriftSS)}\tilde{\thetadrift}^{-1}\bdrift } \eqsp,
  \end{equation*}
\begin{equation}
\label{eq:definition_boundM}
 \tilde{\thetadrift} = \thetadrift^2 \delta(2 \bdrift + \thetadrift \delta)^{-1} \eqsp, \ \boundM(\epsilon) = \tilde{\thetadrift}^{-1} \bdrift\left(1+\rme^{\tilde{\thetadrift} \Fsmall(\epsilon,\RdriftSS)}\right)+\deltaSS/2   \eqsp,
  % \boundM(\epsilon) = \sup_{(x,y) \in \smallSet_{\deltaSS}} \{(1/2)(\Vdrifta(x) + \Vdrifta(y))\} + \rme^{\tilde{\thetadrift} \Fsmall(\epsilon,\RdriftSS)}\tilde{\thetadrift}^{-1} \bdrift  \eqsp.
\end{equation}
where $\Fsmall$ is defined in \eqref{eq:Fsmall}.
% \begin{equation}
%   \label{eq:definition-thetatilde}
%  \tilde{\thetadrift} = \thetadrift^2 \delta(2 \bdrift + \thetadrift \delta)^{-1} \eqsp.
% \end{equation}
\end{proposition}
\begin{proof}
For notational simplicity, set $\tpsRetour_j = \tpsRetour_j^{\Thetadrift,\eqspD \Fsmall(\epsilon,\RdriftSS)}$.
Note that for all $x, y \in \rset^d$,
\begin{equation*}
%\label{eq:drift-simple-double}
  \generator \Vdrifta(x) + \generator \Vdrifta(y) \leq -\tilde{\thetadrift}(\Vdrifta(x) + \Vdrifta(y)) + 2\bdrift\1_{\Thetadrift}(x,y) \eqsp.
\end{equation*}
Then by the Dynkin formula (see
  e.g. \cite[Eq. (8)]{meyn:tweedie:1993:III})  the
  process
\begin{equation*}
t \mapsto (1/2)\rme^{\tilde{\thetadrift} \parenthese{\tpsRetour_1 \wedge
    t}} \defEns{\Vdrifta\parenthese{\XSDE_{\tpsRetour_1 \wedge
    t}} + \Vdrifta\parenthese{\YSDE_{\tpsRetour_1 \wedge t}} } \eqsp, \quad  t \geq  \Fsmall(\epsilon,\RdriftSS) \eqsp,
\end{equation*}
is a  positive supermartingale. Using the optional stopping theorem and the
  Markov property, we have, using that for all $t \geq 0$ $\PEMb{(x,y)}{\rme^{\tilde{\thetadrift}t} V(X_t)} \leq V(x) + \beta \tilde{\thetadrift}^{-1} \rme^{\tilde{\thetadrift} t}$,
  \begin{equation*}
    \PEMb{(x,y)}{\rme^{\tilde{\thetadrift} \tpsRetour_1}} \leq  (1/2)(\Vdrifta(x) + \Vdrifta(y)) + \rme^{\tilde{\thetadrift} \Fsmall(\epsilon,\RdriftSS)}\tilde{\thetadrift}^{-1}\bdrift  \eqsp.
  \end{equation*}
 The result  then follows  from this inequality and the strong Markov property.
\end{proof}

\begin{theorem}
  \label{theo:quantitative_bound_SDE}
  Assume \Cref{assum:drift} and \Cref{assum:drift-simple}. Then for all $\epsilon \in \ooint{0,1}$, $t \geq 0$ and $x,y \in \rset^d$,
  \begin{equation*}
%    \label{eq:quantitative_bound_SDE_1}
    \tvnorm{\PSDE_t(x,\cdot) - \PSDE_t(y,\cdot)} \leq
2 \rme^{-\tilde{\thetadrift} t/2}\defEns{ (1/2)( \Vdrifta(x)+\Vdrifta(y)) + \rme^{\tilde{\thetadrift} \Fsmall(\epsilon,\RdriftSS)}\tilde{\thetadrift}^{-1}\bdrift }
+4\tauConv^t \eqsp, 
 % \tvnorm{\PSDE_t(x,\cdot) - \nu} &\leq ((1-\epsilon)^{-2}+(1/2)\defEns{\Vdrifta(x) +\thetadrift^{-1} \bdrift} + \rme^{\tilde{\thetadrift} \Fsmall(\epsilon,\RdriftSS)}\tilde{\thetadrift}^{-1} \bdrift)\tauConv^t
  \end{equation*}
where $\Fsmall$ is defined in \eqref{eq:Fsmall}, $\tilde{\thetadrift}, \boundM(\epsilon)$ in \eqref{eq:definition_boundM} and
\begin{equation*}
%  \label{eq:definition-rate-conv}
\log(\tauConv ) = (\tilde{\thetadrift}/2) \log(1-\epsilon) \{\log(\boundM(\epsilon)) -\log(1-\epsilon) \}^{-1} \eqsp.
\end{equation*}
\end{theorem}
\begin{proof}
  Let $x,y \in \rset^d$ and $t \geq 0$. For all $\nbTpsR \geq 1$ and $\epsilon \in \ooint{0,1}$,
  \begin{equation}
    \label{eq:decomposition_proba-1}
    \PPMb{(x,y)}{\tau_c > t} \leq \PPMb{(x,y)}{\tau_c > t , \tpsRetour_{\nbTpsR} \leq t} + \PPMb{(x,y)}{ \tpsRetour_{\nbTpsR} > t} \eqsp,
  \end{equation}
where $\tpsRetour_\ell = \tpsRetour_\ell^{\Thetadrift,\eqspD \Fsmall(\epsilon,\RdriftSS)}$.
  We now bound the two terms in the right hand side of this
  equation. For the first term, since $\Thetadrift \subset \diagSet_{\RdriftSS}$, by \eqref{eq:uniform-bound}, we have conditioning successively on $\filtrationTilde_{\tpsRetour_j}$, for $j=\nbTpsR,\dots,1$, and using the strong Markov property,
  \begin{equation}
\label{eq:first_bound_RHS_decomp_proba-1}
    \PPMb{(x,y)}{\tau_c > t , \tpsRetour_{\nbTpsR} \leq t} \leq (1-\epsilon)^{\nbTpsR} \eqsp.
  \end{equation}
For the second term,  using \Cref{propo:drift-condition} and the Markov inequality, we get
\begin{align*}
%  \label{eq:second_bound_RHS_decomp_proba-1}
&   \PPMb{(x,y)}{ \tpsRetour_{\nbTpsR} > t}  \leq    \PPMb{(x,y)}{ \tpsRetour_{1} > t/2} + \PPMb{(x,y)}{ \tpsRetour_{\nbTpsR}-\tpsRetour_{1} > t/2}
\\
& \qquad\leq \rme^{-\tilde{\thetadrift} t/2}\defEns{ (1/2)( \Vdrifta(x)+\Vdrifta(y)) + \rme^{\tilde{\thetadrift} \Fsmall(\epsilon,\RdriftSS)}\tilde{\thetadrift}^{-1}\bdrift }
+ \rme^{-\tilde{\thetadrift} t/2}  \{\boundM(\epsilon)\}^{\nbTpsR-1}  \eqsp.
\end{align*}
The proof is completed combining this inequality and \eqref{eq:first_bound_RHS_decomp_proba-1} in \eqref{eq:decomposition_proba-1} and taking $\nbTpsR = \ceil{2^{-1}\tilde{\thetadrift}t \big / (\log(\boundM(\epsilon)) - \log(1-\epsilon))}$.
\end{proof}
More precise bounds can be obtained under more stringent assumption on the drift $b$; see \cite{bolley:gentil:guillin:2012} and \cite{eberle:2015}.
\begin{assumptionG}
  \label{assum:strongConvexityOutsideBallDrift}
  There exist $\RStD \geq 1$ and $\constStD >0$, such that for all
  $x,y \in \rset^d$, $\norm{x-y} \geq \RStD$,
  \[
  \ps{b(x) -b(y)} {x-y} \leq -\constStD
  \norm[2]{x-y} \eqsp.
  \]
\end{assumptionG}
% Under \Cref{assum:strongConvexityOutsideBallDrift} we bound
% exponential moments on the successive return time to $
% \diagSet_{\RStD}$ for the coupling by reflection specifically.
% The condition $\RStD \geq 1$ could be weaken by $\RStD \geq 0$,
% but we impose it to simplify the calculations and notations below.

\begin{proposition}
  \label{propo:moment-exp-return-time-ST}
Assume \Cref{assum:drift} and \Cref{assum:strongConvexityOutsideBallDrift}.
  \begin{enumerate}[label=(\alph*)]
  \item
\label{item:propo:moment-exp-return-time-ST-1}
For all $x,y \in \rset^d$ and $\epsilon \in \ooint{0,1}$
    \begin{equation*}
    \hspace{-1cm}  \PEMb{(x,y)}{\exp\parenthese{\frac{\constStD}{2} \parenthese{\tau_c \wedge \tpsRetour^{\diagSet_{\RStD},\Fsmall(\epsilon,\RStD)}_1}}} \leq 1 +\norm{x-y} + (1+\RStD) \rme^{\constStD \Fsmall(\epsilon,\RStD)/2} \eqsp.
    \end{equation*}
  \item
\label{item:propo:moment-exp-return-time-ST-2}
For all $x,y \in \rset^d$, $\epsilon \in \ooint{0,1}$ and $j \geq 1$
    \begin{multline*}
    \PEMb{(x,y)}{\exp\parenthese{(\constStD/2) \parenthese{\tau_c \wedge \tpsRetour^{\diagSet_{\RStD},\Fsmall(\epsilon,\RStD)}_j}}}\\
 \leq \{ \boundMd(\epsilon)\}^{j-1}\defEns{1 +\norm{x-y} + (1+\RStD) \rme^{\constStD \Fsmall(\epsilon,\RStD)/2}}\eqsp,
    \end{multline*}
\begin{equation}
    \label{eq:definition-borne-ST}
\boundMd(\epsilon) =   (1+\rme^{\constStD \Fsmall(\epsilon,\RStD)/2})(1+\RStD) \eqsp,
\end{equation}
where $\Fsmall$ is given in \eqref{eq:Fsmall}.
  \end{enumerate}
\end{proposition}

\begin{proof}
In the proof, we set $\tpsRetour_j = \tpsRetour_j^{\diagSet_{\RStD},\Fsmall(\epsilon,\RStD)}$.
\begin{enumerate}[label=(\alph*),wide=0pt,labelindent=\parindent]
\item
Consider  the sequence of increasing stopping time
$$
\tau_k =  \inf\{ t > 0 \ | \ \norm{\XSDE_t-\YSDE_t}  \not \in \ccint{k^{-1},k} \} \eqsp,  \quad k \geq 1 \eqsp,
$$
and set $\zeta_k = \tau_k \wedge \tpsRetour_1$. We derive a
  bound on $\PEMbLigne{(x,y)}{\exp\{ (\constStD/2)\zeta_k\}}$ independent
  on $k$. Since $\lim_{k \to \plusinfty } \uparrow\tau_k = \tau_c$  almost surely, the monotone convergence theorem implies that the same bound holds for  $\PEMbLigne{(x,y)}{\exp\{(\constStD/2)(\tau_c \wedge
    \tpsRetour_1)\}}$. Set now $\VdriftS(x,y) = 1 + \norm{x-y}$.  Since $\VdriftS \geq
  1$ and $\tau_c < \infty$~\as~by \Cref{propo:reflection_coupling_SDE}, it suffices to give a bound on $\PEMbLigne{(x,y)}{\exp\{ (\constStD/2)\zeta_k\}\VdriftS(\XSDE_{\zeta_k},\YSDE_{\zeta_k})}$. By Itô's formula, we
  have for all $v,t \leq \tau_c$, $v \leq t$
  \begin{multline}
\label{eq:ito-St-base}
\rme^{\constStD t/2}\VdriftS(\XSDE_t,\YSDE_t)
    % & = \norm{x-y} + \constantStrong \int_0^{t} \rme^{\constantStrong s} \norm[][\sigma^{-2}]{\XSDE_s-\YSDE_s} \rmd s \\
    % \nonumber
    % & +\int_0^{t}\rme^{\constantStrong s} \norm[][\sigma^{-2}]{\XSDE_s-Z_s}^{-1} \ps[\sigma^{-2}]{b(\XSDE_s)-b(\YSDE_s)}{\XSDE_s-\YSDE_s} \rmd s +2\int_0^t \rme^{\constantStrong s/2}\rmd \rmB_s \\
    = \rme^{\constStD v/2} \VdriftS(\XSDE_v,\YSDE_v) +(\constStD/2) \int_v^{t} \rme^{\constStD u/2} \VdriftS(\XSDE_u,\YSDE_u)  \rmd u \\
 +\int_v^{t} \rme^{\constStD u/2}\ps{b(\XSDE_u)-b(\YSDE_u)}{e_u} \rmd u +2\int_v^t \rme^{\constStD u/2}\rmd B_u^1 \eqsp.
\end{multline}
Using \Cref{assum:strongConvexityOutsideBallDrift}$(b)$, we have for all $k \geq 1$ and  $t_{\mathrm{s}}= \Fsmall(\epsilon,\RStD) \leq v \leq t$
  \begin{multline*}
%    \label{eq:2}
    \rme^{(\constStD/2) (\zeta_k \wedge t)} \VdriftS(\XSDE_{\zeta_k \wedge t},\YSDE_{\zeta_k \wedge t}) \leq \rme^{(\constStD/2) (\zeta_k \wedge v) } \VdriftS(\XSDE_{\zeta_k \wedge v},\YSDE_{\zeta_k \wedge v}) \\+2\int_{\zeta_k \wedge v}^{\zeta_k\wedge t} \rme^{\constStD u/2}\rmd B_u^{1} \eqsp.
  \end{multline*}
So the process
\begin{equation*}
\defEns{\exp\parenthese{(\constStD/2) (\zeta_k \wedge t)} \VdriftS(\XSDE_{\zeta_k \wedge t},\YSDE_{\zeta_k \wedge t}) }_{ t \geq t_{\mathrm{s}}}  \eqsp,
\end{equation*}
is a positive supermartingale and  by the optional stopping theorem, we get
\begin{equation}
  \label{eq:supermartingale-1}
\PEMb{(x,y)}{ \rme^{(\constStD/2) \zeta_k} \VdriftS(\XSDE_{\zeta_k },\YSDE_{\zeta_k })} \leq \PEMb{(x,y)} {\rme^{(\constStD/2) (\tau_k \wedge t_{\mathrm{s}})} \VdriftS(\XSDE_{\tau_k \wedge t_{\mathrm{s}}},\YSDE_{\tau_k \wedge t_{\mathrm{s}} })} \eqsp,
\end{equation}
where we used that   $\zeta_k \wedge t_{\mathrm{s}} = \tau_k \wedge t_{\mathrm{s}}$.
By \eqref{eq:ito-St-base}, \Cref{assum:drift} and \Cref{assum:strongConvexityOutsideBallDrift}, we have
\begin{equation*}
  %\label{eq:3}
\PEMb{(x,y)} {\rme^{(\constStD/2) (\tau_k \wedge t_{\mathrm{s}})} \VdriftS(\XSDE_{\tau_k \wedge t_{\mathrm{s}}},\YSDE_{\tau_k \wedge t_{\mathrm{s}} })}
 \leq \VdriftS(x,y) + (1+\RStD) \rme^{\constStD t_{\mathrm{s}}/2} \eqsp,
\end{equation*}
and \eqref{eq:supermartingale-1} becomes
\begin{equation*}
  \PEMb{(x,y)}{ \rme^{(\constStD/2) \zeta_k} \VdriftS(\XSDE_{\zeta_k },\YSDE_{\zeta_k })} \leq \VdriftS(x,y) + (1+\RStD) \rme^{\constStD t_{\mathrm{s}}/2} \eqsp.
\end{equation*}
\item
 The proof is by induction. The case $j=1$ has been established above. Now let $j \geq
2$. Since on the event $\{ \tau_c >
\tpsRetour_{j-1} \}$, we have
\[
\tau_c \wedge \tpsRetour_j = \tpsRetour_{j-1} + (\tau_c \wedge \tpsRetour_{1}) \circ \shift_{\tpsRetour_{j-1}} \eqsp,
\]
 where $\shift$ is the shift operator, we have conditioning on $\filtrationTilde_{\tpsRetour_{j-1}}$, using the strong Markov property, \Cref{propo:reflection_coupling_SDE} and the first part,
 \begin{equation*}
%   \label{eq:}
   \PEMb{(x,y)}{\1_{\{ \tau_c > \tpsRetour_{j-1} \}} \rme^{(\constStD/2) \parenthese{\tau_c \wedge \tpsRetour_j}}}
\leq \boundMd(\epsilon) \, \PEMb{(x,y)}{\1_{\{ \tau_c > \tpsRetour_{j-1} \}}\rme^{(\constStD/2)\tpsRetour_{j-1}}  } \eqsp.
 \end{equation*}

%  \begin{equation*}
% %   \label{eq:}
%    \PEMb{(x,y)}{\1_{\{ \tau_c > \tpsRetour_{j-1} \}} \exp\parenthese{(\frac{\constStD}{2} \parenthese{\tau_c \wedge \tpsRetour_j}}}
% \leq \boundMd(\epsilon) \, \PEMb{(x,y)}{\1_{\{ \tau_c > \tpsRetour_{j-1} \}}\exp\parenthese{\frac{\constStD}{2}\tpsRetour_{j-1}}  } \eqsp,
%  \end{equation*}
Then the proof follows since $\boundMd(\epsilon) \geq 1$.
\end{enumerate}
\end{proof}

\begin{theorem}
  \label{theo:quantitative_bound_SDE_2}
  Assume \Cref{assum:drift} and \Cref{assum:strongConvexityOutsideBallDrift}. Then for all $\epsilon \in \ooint{0,1}$, $t \geq 0$ and $x,y \in \rset^d$,
  \begin{equation*}
%    \label{eq:quantitative_bound_SDE_1}
    \tvnorm{\PSDE_t(x,\cdot) - \PSDE_t(y,\cdot)} \leq 2 \defEns{(1-\epsilon)^{-1}+1 +\norm{x-y} }\tauConvSt^t
  \end{equation*}
\begin{equation*}
 % \label{eq:definition-rate-conv}
\log(\tauConvSt ) = (\constStD/2) \log(1-\epsilon) (\log(\boundMd(\epsilon)) -\log(1-\epsilon) )^{-1} \eqsp,
\end{equation*}
 where $\boundMd(\epsilon)$ is defined in  \eqref{eq:definition-borne-ST}.
\end{theorem}
\begin{proof}
The proof is along the same lines as  \Cref{theo:quantitative_bound_SDE}.
Set $\tpsRetour_j = \tpsRetour_j^{\diagSet_{\RStD},\Fsmall(\epsilon,\RStD)}$ for $j \geq 1$.
  Let $x,y \in \rset^d$ and $t \geq 0$. For all $\nbTpsR \geq 1$ and $\epsilon \in \ooint{0,1}$,
  \begin{equation}
    \label{eq:decomposition_proba}
    \PPMb{(x,y)}{\tau_c > t} \leq \PPMb{(x,y)}{\tau_c > t , \tpsRetour_{\nbTpsR} \leq t} + \PPMb{(x,y)}{ \tpsRetour_{\nbTpsR} \wedge \tau_c > t} \eqsp.
  \end{equation}
   For the first term, by \eqref{eq:uniform-bound} we have conditioning successively on $\filtrationTilde_{\tpsRetour_j}$, for $j=\nbTpsR,\cdots,1$, and using the strong Markov property,
  \begin{equation}
\label{eq:first_bound_RHS_decomp_proba}
    \PPMb{(x,y)}{\tau_c > t , \tpsRetour_{\nbTpsR} \leq t} \leq (1-\epsilon)^{\nbTpsR} \eqsp.
  \end{equation}
For the second term, using \Cref{propo:moment-exp-return-time-ST}-\ref{item:propo:moment-exp-return-time-ST-2} and the Markov inequality, we get
\begin{equation}
  \label{eq:second_bound_RHS_decomp_proba}
\hspace{-0.1cm}
   \PPMb{(x,y)}{ \tpsRetour_{\nbTpsR} \wedge \tau_c > t} \leq \rme^{-\frac{\constStD t}{2} }  \{\boundMd(\epsilon)\}^{\nbTpsR-1} \defEns{ 1 +\norm{x-y} + (1+\RStD) \rme^{\frac{\constStD \Fsmall(\epsilon,\RStD)}{2}}} \eqsp.
\end{equation}
Taking $\nbTpsR = \floor{(\constStD t/2)  \big/ (\log(\boundMd(\epsilon)) - \log(1-\epsilon)) }$ and combining \eqref{eq:first_bound_RHS_decomp_proba}-\eqref{eq:second_bound_RHS_decomp_proba} in \eqref{eq:decomposition_proba}   concludes the proof.
\end{proof}

\subsection{Proof of \Cref{theo:quantative-bound-lang-convex} and \Cref{theo:convergence_TV_dec-stepsize-StV}}
\label{ssec:proof:theo:quantative-bound-lang-stong-convex}
\label{sec:proof-crefth-bound-quant-conv}
\label{sec:proof-crefth-steps-1}

Recall that $(\PLang_t)_{t \geq 0}$ is the Markov semigroup
of the Langevin equation associated with $\UD$ and let $\geneLang$ be the
corresponding generator.  Since $(\PLang_t)_{t \geq 0}$ is reversible
with respect to $\pi$, we deduce from
\Cref{theo:quantitative_bound_SDE} and \Cref{theo:quantitative_bound_SDE_2} quantitative bounds for the
exponential convergence of $(\PLang_t)_{t \geq 0}$ to $\pi$ in
total variation noting that if $(\YL_t)_{t \geq 0}$ is a solution of \eqref{eq:langevin}, then $(\YL_{t/2})_{t \geq 0}$ is a weak solution of the  rescaled Langevin diffusion:
\begin{equation}
\label{eq:langevin-scale}
  \rmd \YLtilde_t= -(1/2)\nabla U(\YLtilde_t) \rmd t + \rmd \BM^d_t \eqsp.
\end{equation}

\begin{proof}[Proof of \Cref{theo:quantative-bound-lang-convex}]
Since the generator associated with the SDE  \eqref{eq:langevin-scale} is $(1/2) \generatorL$, \Cref{propo:drift_exp_conv} shows that
\eqref{eq:lyapunov-conditionb}  holds for $\VOne$
with constants $\thetaOne/2$ and $\bOne/2$.
Using that for all $a_1,a_2 \in \rset$, $\rme^{(a_1+a_2)/2} \leq (1/2)(\rme^{a_1}+\rme^{a_2})$, \Cref{assum:drift-simple}-\ref{item:2} holds for $\deltaSS=2 \thetaOne^{-1}\bOne$ and $\RdriftSS = (8/\rhoU) \log( 4 \thetaOne^{-1} \bOne)$.
By  \Cref{theo:quantitative_bound_SDE} with $\epsilon = 1/2$, we get for all $x,y \in \rset^d$ and $t \geq 0$
%%%%Avant simplification
% \begin{equation}
% \label{eq:9}
%     \tvnorm{\PLang_t(x,\cdot) - \PLang_t(y,\cdot)} \leq  
% \rme^{-\tilde{\thetadrift} t/2}\defEns{ (1/2)( \Vdrifta(x)+\Vdrifta(y)) + \rme^{\tilde{\thetadrift} \Fsmall(\epsilon,\RdriftSS)}\tilde{\thetadrift}^{-1}\bdrift }
% +2\varpi^t \eqsp,
% \end{equation}

%%%%Après simplification
\begin{multline}
\label{eq:9}
    \tvnorm{\PLang_t(x,\cdot) - \PLang_t(y,\cdot)} \leq 4\varpi^t \\
+  
2 \rme^{-\thetaOne t/4}\defEns{ (1/2)( \VOne(x)+\VOne(y)) 
+ 2 \thetaOne^{-1} \bOne\rme^{4 \thetaOne^{-1} \Fsmall(2^{-1},(8/\rhoU) \log( 4 \thetaOne^{-1} \bOne))} }
 \eqsp,
\end{multline}
where $\varpi$ is defined in \eqref{eq:kappa-convex}.
By \cite[Theorem 4.3-(ii)]{meyn:tweedie:1993:III}, \eqref{eq:lyapunov-conditionb} implies that $\int_{\rset^d} \VOne(y) \pi(\rmd y ) \leq \bOne \thetaOne^{-1}$.  The proof  is then concluded using this bound, \eqref{eq:9} and that $\pi$ is invariant for $(\PL_t)_{t \geq 0}$
% , we get for all $x \in \rset^d$,
% \begin{equation*}
% \tvnorm{\delta_x \PLang_{\Gamma_{n+1,p}} - \piVUl} \leq   \rateConvconv^{\Gammasum_{n+1,p}}\defEns{4+(1/2)(\bOne \thetaOne^{-1} +\VEa(x)) } \eqsp.
% \end{equation*}
%The proof is then concluded by \Cref{propo:drift-Euler} and \Cref{lem:bound_inho_moment}.
% \begin{equation*}
% \tvnorm{\delta_x Q^n_\gamma \PLang_{\Gamma_{n+1,p}} - \piVUl} \leq   (5/4) \rateConvconv^{\Gammasum_{n+1,p}}\defEns{\bOne \thetaOne^{-1} +\Funfo(\thetaEa,\Gammasum_{n},\bEa,\VEa(x)) } \eqsp.
% \end{equation*}
\end{proof}

\begin{proof}[Proof of \Cref{theo:convergence_TV_dec-stepsize-StV}]
By applying \Cref{theo:quantitative_bound_SDE_2} with $\epsilon = 1/2$, the triangle inequality and using that $\pi$ is invariant for $(\PL_t)_{t \geq 0}$, we have
\begin{equation*}
    \tvnorm{\PLang_t(x,\cdot) -\pi} \leq
    2 \defEns{3+\norm{x-\xstar}+\int_{\rset^d}\norm{y-\xstar}\rmd \pi(y) }\tauConvSt^{t} \eqsp.
  \end{equation*}
It remains to show that  $\int_{\rset^d}\norm{y-\xstar}\rmd \pi(y) \leq (d/\constSt + \RSt^2)^{1/2}$. For this, we establish a drift inequality for the generator $\generatorL$ of the Langevin SDE associated with $\UD$. Consider the function $\VStD(x) = \norm[2]{x-\xstar}$. For all $x \in \rset^d$, we have using $\nabla U(\xstar)=0$,
\begin{equation*}
%\label{eq:15}
  \generatorL \VStD(x) \leq 2(d -\ps{\nabla U(x)-\nabla U(\xstar)}{x-\xstar} ) \eqsp.
\end{equation*}
Therefore by \Cref{assum:strongConvexityOutsideBallDrift}, for all $x \in \rset^d$, $\norm{x-\xstar} \geq \RSt$, we get
\begin{equation*}
%\label{eq:16}
  \generatorL \VStD(x) \leq -2\constSt\VStD(x) + 2d \eqsp,
\end{equation*}
and for all $x \in \rset^d$,
\begin{equation*}
%\label{eq:16}
  \generatorL \VStD(x) \leq  -2 \constSt\VStD(x) + 2 (d + \constSt \RSt^2) \eqsp.
\end{equation*}
By \cite[Theorem 4.3-(ii)]{meyn:tweedie:1993:III}, we get $ \int_{\rset^d}\VStD(y)\rmd \pi(y) \leq d/\constSt +  \RSt^2$.
The bound on $C(\delta_x \QKer_{\gaStep}^n)$ is a consequence of the Cauchy-Schwarz inequality, \Cref{propo:drift_R_gamma_Stconv} and \Cref{lem:bound_inho_moment}.
The bound for $A(\gaStep,x)$ similarly follows from \Cref{assum:regularity}, \Cref{propo:drift_R_gamma_Stconv} and \Cref{lem:bound_inho_moment}.

\end{proof}

\section*{Acknowledgements}
The authors are indebted to Arnaud Guillin for sharing his knowledge of Poincaré and log-Sobolev inequalities.
The authors are grateful to Andreas Eberle for very careful readings and many useful comments.
The author thank the anonymous referees for their constructive feedback.
The work of A.D. and E.M. is supported by the Agence Nationale de la Recherche, under grant ANR-14-CE23-0012 (COSMOS).
\bibliographystyle{plain}
\bibliography{biblio}

\end{document}